%% file: main.tex
\documentclass[sn-mathphys,Numbered]{sn-jnl} 
\usepackage{etoolbox}   
\makeatletter
\patchcmd{\@maketitle}{\artauthors}{{\artauthors}}{}{}
\makeatother

\usepackage{graphicx}%
\usepackage{multirow}%
\usepackage[title]{appendix}
\usepackage{xcolor}
\usepackage{textcomp}
\usepackage{booktabs}
\usepackage{algorithm}
\usepackage{algorithmic}
\usepackage{listings}%
\usepackage{lmodern}
\usepackage{microtype}
\usepackage{enumitem}
\input{math.tex}
\usepackage[numbers,sort&compress]{natbib}

\theoremstyle{plain}%
\newtheorem{lemma}{Lemma}
\newtheorem{theorem}{Theorem}
\newtheorem{proposition}{Proposition}%
\newtheorem{corollary}{Corollary}
\newtheorem{assumption}{Assumption}
\newtheorem{definition}{Definition}

\theoremstyle{remark}
\newtheorem{remark}{Remark}

\begin{document}

\title[On the Complexity of Decentralized Smooth Nonconvex Finite-Sum Optimization]{On the Complexity of Decentralized Smooth Nonconvex Finite-Sum Optimization\footnote{The conference version of this manuscript is published on ICML 2024 \cite{bai2024complexity}, which contains the results under the PL condition (Section \ref{sec:PL}) in the special case of $L=\bL$ and $n=1$.}}

\author[1]{\fnm{Luo} \sur{Luo}} \email{luoluo@fudan.edu.cn}

\author[1]{\fnm{Yunyan} \sur{Bai}} \email{yybai22@m.fudan.edu.cn}

\author[2]{\fnm{Lesi} \sur{Chen}} \email{chenlc23@mails.tsinghua.edu.cn}

\author[3]{\fnm{Yuxing} \sur{Liu}} \email{yuxing6@illinois.edu}

\author*[4]{\fnm{Haishan} \sur{Ye}} \email{yehaishan@xjtu.edu.cn}

\affil[1]{\normalsize \orgdiv{School of Data Science}, \orgname{Fudan University}}

\affil[2]{\normalsize \orgdiv{Institute for Interdisciplinary Information Sciences}, \orgname{Tsinghua Univerisity}}

\affil[3]{\normalsize  \orgdiv{Siebel School of Computing and Data Science}, \orgname{University of Illinois Urbana-Champaign}}

\affil[4]{\normalsize \orgdiv{School of Management}, \orgname{Xi’an Jiaotong University}}

\abstract{We study the decentralized optimization problem $\min_{{\bf x}\in{\mathbb R}^d} f({\bf x})\triangleq \frac{1}{m}\sum_{i=1}^m f_i({\bf x})$, where the local function on the $i$-th agent has the form of $f_i(\vx)\triangleq \frac{1}{n}\sum_{j=1}^n f_{i,j}(\vx)$ 
and every individual $f_{i,j}$ is smooth but possibly nonconvex.
We propose a stochastic algorithm called DEcentralized probAbilistic Recursive gradiEnt deScenT (\DEAREST) method, 
which achieves an $\epsilon$-stationary point at each agent with the communication rounds of $\tilde{\mathcal O}(L\epsilon^{-2}/\sqrt{\gamma}\,)$,
the computation rounds of~$\tilde{\mathcal O}(n+(L+\min\{nL, \sqrt{n/m}\bar L\})\epsilon^{-2})$,  
and the local incremental first-oracle calls of~${\mathcal O}(mn + {\min\{mnL, \sqrt{mn}\bar L\}}{\epsilon^{-2}})$, where $L$ is the smoothness parameter of the objective function, $\bar L$ is the mean-squared smoothness parameter of all individual functions, and $\gamma$ is the spectral gap of the mixing matrix associated with the network.
We then establish the lower bounds to show that the proposed method is near-optimal.
Notice that the smoothness parameters $L$ and $\bar L$ used in our algorithm design and analysis are global, leading to sharper complexity bounds than existing results that depend on the local smoothness.
We further extend \DEAREST~to solve the decentralized finite-sum optimization problem under the Polyak--{\L}ojasiewicz condition, also achieving the near-optimal complexity bounds.  
}


\maketitle

\section{Introduction}
We study the decentralized optimization problem
\begin{align}\label{prob:main}
\min_{\vx\in\BR^d} f(\vx)\triangleq\frac{1}{m}\sum_{i=1}^m f_i(\vx),
\end{align}
over a connected network with $m$ agents, where $f_i:\BR^d\to\BR$ is the local function on the~$i$-th agent that has the finite-sum structure with $n$ individual functions as follows
\begin{align}\label{func:local}
    f_i(\vx) \triangleq \frac{1}{n}\sum_{j=1}^n f_{i,j}(\vx).
\end{align}
We suppose each individual function $f_{i,j}:\BR^d\to\BR$ is smooth but possibly nonconvex.
This formulation includes a lot of applications in statistics \cite{antoniadis2011penalized,zhong2023online}, signal processing \cite{facchinei2015parallel,nedic2018network,sun2018geometric,sun2016complete}, and machine learning \cite{bottou2018optimization,hendrikx2021optimal,wai2017decentralized,yang2019survey,allen2016variance}.
In decentralized scenario, all of agents target to collaboratively solve problem (\ref{prob:main}) and each of them can only communicate with its neighbors.
We focus on the complexity for achieving the approximate stationary point of the global objective at every agent.

For decentralized optimization, the limitation on the communication protocol leads to the local agents cannot access the exact global information at each round, which leads to the requirement of communication rounds to reduce the consensus error.
The gradient tracking~\cite{nedic2017achieving,pu2021distributed,qu2019accelerated,shi2015extra} is a useful technique to approximate the average of local gradients and make the local first-order estimator accurate.
Directly extending (stochastic) gradient descent methods to decentralized setting \cite{lian2017can,lu2021optimal,sundhar2010distributed,chen2015learning,chen2012diffusion,yuan2016convergence,sun2019distributed,nedic2009distributed} cannot take the advantage of the popular finite-sum structure in local functions.
It is well-known that the algorithms with stochastic recursive gradient estimator \cite{nguyen2017sarah,fang2018spider,pham2019proxsarah,li2021page,carmon2020first} can achieve the optimal incremental first-order oracle (IFO) complexity for finding the approximate stationary point of the finite-sum nonconvex function under the mean-squared smooth assumption.
\citet{sun2020improving} first combined stochastic recursive gradient estimator with gradient tracking to solve decentralized nonconvex finite-sum optimization problem by proposing Decentralized Gradient Estimation and Tracking (\DGET).
Later, \citet{xin2022fast} and \citet{zhan2022efficient} proposed \GTSARAH~and efficient decentralized
stochastic gradient descent (\EDSGD) respectively to improve the complexity in terms of the dependency on the numbers of agents and individual functions.
\citet{li2022destress} proposed DEcentralized STochastic REcurSive gradient methodS (\DESTRESS), which introduces Chebyshev acceleration \cite{arioli2014chebyshev} to achieve the tighter dependency on the spectral gap of the mixing matrix associated with the network.
\citet{metelev2024decentralized} further considered the nonconvex problem over the time-varying network.
Additionally, \citet{lu2021optimal,yuan2022revisiting} studied the tightness of decentralized nonconvex optimization in the online setting.

It is worth noting that existing works \cite{sun2020improving,xin2022fast,li2022destress,metelev2024decentralized,lu2021optimal,yuan2022revisiting} for decentralized nonconvex optimization only consider the local smoothness parameters, which may be arbitrary larger than the global ones.
Furthermore, their analysis for the computation complexity focuses on one of the overall local incremental first-order oracle (LIFO) calls and the number of computation rounds.
Noticing that in each computation round, a distributed algorithm can make partial agents access their LIFO and allow other agents skip their local computation steps \cite{maranjyan2022gradskip,mishchenko2022proxskip}.
Therefore, the LIFO calls and the computation rounds should be addressed separately.
Recently, \citet{ye2023multi,liu2024decentralized} studied the distributed optimization by considering the global smoothness dependency and the partial participated protocol, while their results only address the convex problem.

In this paper, we refine the setting of decentralized smooth
nonconvex finite-sum optimization (\ref{prob:main}) by distinguishing the different smoothness parameters and considering the partial participation protocol.
We proposed a novel stochastic algorithm called DEcentralized probAbilistic Recursive  gradiEnt deScenT (\DEAREST), achieving the $\epsilon$-stationary point at every agent with the communication rounds of $\tilde{\mathcal O}(L\epsilon^{-2}/\sqrt{\gamma}\,)$,
the computation rounds of $\tilde{\mathcal O}(n+(L+\min\{nL, \sqrt{n/m}\bar L\})\epsilon^{-2})$,  
and the LIFO calls of~${\mathcal O}(mn + {\min\{mnL, \sqrt{mn}\bar L\}}{\epsilon^{-2}})$, 
where $L$ is the smoothness parameter of the global objective function $f$, $\bL$ is the mean-squared smoothness parameter of all individual functions $\{f_{i,j}\}_{i,j=1}^{m,n}$, and $\gamma$ is the spectral gap of the mixing matrix associated with the network.
We then establish lower complexity bounds with respect to $\epsilon$, $m$, $n$, $L$, $\bL$, and~$\gamma$ to show the near-optimality of our method.
Notice that the smoothness parameters $L$ and $\bar L$ in our results are global, leading to sharper complexity bounds than existing ones that depend on the local smoothness.
For single-machine scenario (i.e., $m=1$), our theory indicates incremental first-order (IFO) complexity of ${\mathcal O}(n + {\min\{nL, \sqrt{n}\bar L\}}{\epsilon^{-2}})$, 
which is a trade-off between the complexity of ${\mathcal O}(nL{\epsilon^{-2}})$ from vanilla gradient descent~\cite{nesterov2018lectures,carmon2020lower,carmon2021lower} and the complexity of ${\mathcal O}(n + \sqrt{n}\bar L{\epsilon^{-2}})$ from stochastic variance-reduced methods \cite{fang2018spider,pham2020proxsarah,wang2019spiderboost,li2021page,zhou2018stochastic,zhou2019lower}.
We further apply \DEAREST~to solve the decentralized finite-sum problem under the Polyak--{\L}ojasiewicz (PL) condition \cite{lojasiewicz1963topological,polyak1963gradient,yue2023lower}, 
which achieves the $\epsilon$-suboptimal solution at every agent with the communication rounds of~$\tilde\fO(\kappa\ln(1/\epsilon))$, the computation rounds $\tilde\fO((n+\kappa+\min\{n\kappa, \sqrt{{n}/{m}}\bar \kappa\})\ln(1/\epsilon))$, and the LIFO calls of~$\fO((mn + \min\{mn\kappa, \sqrt{mn}\bar\kappa\})\ln(1/\epsilon))$, where $\kappa\triangleq L/\mu$ is the global condition \mbox{number}, $\bar\kappa\triangleq\bL/\mu$ is the mean-squared condition number, and $\mu$ is the PL parameter.
We also provide lower bounds to show above upper complexity bounds under the PL condition are near-optimal.

\section{Preliminaries}\label{sec:preliminaries}

We formally introduce the notations and problem settings used in this paper.

\subsection{Notations}

We use bold lower-case letters for vectors and bold upper-case letters for matrices. 
The notations $\Vert \cdot \Vert$ and~$\Vert \cdot \Vert_2$ are used to denote the Frobenius norm and the spectral norm of the matrix, respectively, as well as the Euclidean norm of the vector.
We let $\vone=[1,\cdots,1]^\top\in\BR^m$ and denote $\mI\in\BR^{m\times m}$ as the identity matrix.
We define aggregated variables for all agents as
\begin{align}\label{eq:def-aggregate-x}
\mX=\begin{bmatrix}
\vx_1 \\ \vdots \\\vx_m
\end{bmatrix}\in\BR^{m\times d},
\end{align}
where each $\vx_i\in\BR^{1\times d}$ are the local variable on the $i$-th agent.
We use the lower case with the bar to represent the mean vector, such that
\begin{align*}
\bx = \frac{1}{m}\sum_{i=1}^m \vx_i.    
\end{align*}
We also introduce the matrix for aggregated gradients of local functions $f_1,\dots,f_m$ as
\begin{align}\label{eq:def-aggregate-grad}
\nabla \mF(\mX)=\begin{bmatrix}
\nabla f_1(\vx_1) \\ \vdots \\ \nabla f_m(\vx_m)
\end{bmatrix}
\in\BR^{m\times d}.
\end{align}
For ease of presentation, we let the input of a function can be also organized as a row vector, such as $f(\bx)$, $f_i(\vx_i)$ and $\nabla f_i(\vx_i)$ for some $i\in[m]$.

\subsection{Problem Settings}\label{sec:settings}

We suppose the formulations (\ref{prob:main})--(\ref{func:local}) satisfy the following assumptions.

\begin{assumption}[lower bounded] \label{asm:lower-bounded}
We suppose the objective function $f:\BR^d\to\BR$ is lower bounded, i.e., we have
\begin{align}\label{eq:lower-value}
f^*\triangleq\inf_{\vx\in\BR^{d}}f(\vx)>-\infty.
\end{align}
\end{assumption}

\begin{assumption}[global smooth]\label{asm:smooth-global}
We suppose the differentiable function $f:\BR^d\to\BR$ is $L$-smooth for some $L>0$, i.e.,
\begin{align}\label{eq:smooth-global}
\Norm{\nabla f(\vx)-\nabla f(\vy)} \leq L\Norm{\vx-\vy} 
\end{align}
for all $\vx,\vy\in\BR^d$.  
\end{assumption} \vskip0.1cm

\begin{assumption}[global mean-squared smooth]\label{asm:smooth-mean-squared}
We suppose the individual functions $\{f_{i,j}\}_{i,j=1}^{m,n}$ are $\bL$-mean-squared smooth for some $\bL>0$, i.e., we have
\begin{align}\label{eq:smooth-mean-squared}
\frac{1}{mn}\sum_{i=1}^m\sum_{j=1}^n\Norm{\nabla f_{i,j}(\vx)-\nabla f_{i,j}(\vy)}^2 \leq \bL^2\Norm{\vx-\vy}^2
\end{align}
for all $\vx,\vy\in\BR^d$.  \vskip0.1cm
\end{assumption}

We present the relationship between the smoothness parameters $L$ and $\bL$ as follows.

\begin{proposition}\label{prop:smooth}
The smoothness conditions in Assumptions \ref{asm:smooth-global} and \ref{asm:smooth-mean-squared} have the following relationships:
\begin{enumerate}[itemsep=3pt,topsep=2pt]
    \item[(a)] If the individual functions $\{f_{i,j}\}_{i,j=1
    }^{m,n}$ are $\bL$-mean-squared smooth, then each of $\{f_{i,j}\}_{i,j=1
    }^{m,n}$  and $\{f_{i}\}_{i=1
    }^{m}$ is $ \sqrt{mn}\bL$-smooth, i.e., we have
    \begin{align}\label{eq:smooth-individual1}
        \Norm{\nabla f_{i,j}(\vx) - \nabla f_{i,j}(\vy)} \leq \sqrt{mn}\bL\Norm{\vx-\vy} 
    \end{align}    
    and
    \begin{align}\label{eq:smooth-individual2}    
        \Norm{\nabla f_{i}(\vx) - \nabla f_{i}(\vy)} \leq \sqrt{mn}\bL\Norm{\vx-\vy}
    \end{align}
    for all $\vx,\vy\in\BR^d$, $i\in[m]$, and $j\in[n]$.
    \item[(b)] If the individual functions $\{f_{i,j}\}_{i,j=1}^{m,n}$ are $\bar L$-mean-squared smooth, then the objective function~$f$ is $\bar L$-smooth, i.e., we have
    \begin{align*}
        \Norm{\nabla f(\vx) - \nabla f(\vy)} \leq \bL \Norm{\vx-\vy}    
    \end{align*}
    for all $\vx,\vy\in\BR^d$.
    \item[(c)] 
    For any $L>0$ and $\bL>0$ such that $\bL\geq L$,
    there exist functions $\{f_{i,j}\}_{i,j=1}^{m,n}$ which satisfy Assumption \ref{asm:smooth-global} and \ref{asm:smooth-mean-squared} with the tight smoothness parameters $L$ and $\bar L$, respectively.
\end{enumerate}
\end{proposition}

\begin{remark}
The statements (a) and (b) of Proposition \ref{prop:smooth} imply the upper bounds with respect to the tight global smoothness parameter $L$ in Assumption \ref{asm:smooth-global} is potential sharper than the upper bounds with respect to the  global mean-squared smoothness parameter $\bL$ in Assumption \ref{asm:smooth-mean-squared} (global mean-squared smooth). 
The statement (c) of Proposition \ref{prop:smooth} means the ratio between the tight parameters $\bL$ and $L$ can be arbitrary large, which means considering the difference between $\bL$ and $L$ is very necessary in finite-sum nonconvex optimization.   
\end{remark}

We also present other smoothness assumptions used in related work \cite{sun2020improving,xin2022fast,zhan2022efficient,li2022destress,metelev2024decentralized} for comparison.

\begin{assumption}[local smooth]\label{asm:smooth-local}
We suppose each local function $f_i:\BR^d\to\BR$ is $L_\ell$-smooth for some~$L_\ell>0$, i.e., we have
\begin{align}\label{eq:smooth-local}
\Norm{\nabla f_i(\vx)-\nabla f_i(\vy)} \leq L_\ell\Norm{\vx-\vy}
\end{align}
for all $i\in[m]$ and $\vx,\vy\in\BR^d$. 
\end{assumption}

\begin{assumption}[local mean-squared smooth]\label{asm:smooth-local-mean-squared}
We suppose the individual functions $\{f_{i,j}\}_{j=1}^n$ on each agent are~$\bL_\ell$-mean-squared smooth for some $\bL_\ell>0$, 
i.e., we have
\begin{align}\label{eq:smooth-mean-squared-local}
\frac{1}{n}\sum_{j=1}^n\Norm{\nabla f_{i,j}(\vx)-\nabla f_{i,j}(\vy)}^2 \leq \bL_\ell^2\Norm{\vx-\vy}^2
\end{align}
for all $i\in[m]$ and $\vx,\vy\in\BR^d$. 
\end{assumption}

\begin{assumption}[individual smooth]\label{asm:smooth-max}
We suppose each individual function $f_{i,j}$ is $\Lm$-smooth for some~$\Lm>0$, 
i.e., we have
\begin{align*}
\Norm{\nabla f_{i,j}(\vx)-\nabla f_{i,j}(\vy)} \leq \Lm\Norm{\vx-\vy}
\end{align*}
for all $i\in[m]$, $j\in[n]$ and $\vx,\vy\in\BR^d$. 
\end{assumption}

We present the relationship between the assumptions in related work (Assumptions~\ref{asm:smooth-local} and \ref{asm:smooth-local-mean-squared}) and ours (Assumptions \ref{asm:smooth-global} and \ref{asm:smooth-mean-squared}) in the following proposition.

\begin{proposition}\label{prop:smooth-2}
The smoothness conditions in Assumptions \ref{asm:smooth-global}--\ref{asm:smooth-local-mean-squared} holds that:
\begin{enumerate}[itemsep=3pt,topsep=2pt]
    \item[(a)] If each local function $f_i$ is $L_\ell$-smooth, then the objective function $f$ is $L_\ell$-smooth. 
    \item[(b)] If the individual functions $\{f_{i,j}\}_{j=1}^n$ on each agent are $\bL_\ell$-mean-squared smooth, then all of the individual functions $\{f_{i,j}\}_{i,j=1}^{m,n}$ are $\bar L_\ell$-mean-squared smooth.
    \item[(c)] For any $L$ and $L_\ell$ such that $L_\ell\geq L>0$, there exist functions $\{f_{i}\}_{i=1}^{m}$ which satisfy Assumption \ref{asm:smooth-global} and \ref{asm:smooth-local} with the tight smoothness parameters $L$ and $L_\ell$, respectively.
    \item[(d)] For any $\bL$ and $\bL_\ell$ such that $\bL_\ell\geq \bL>0$, there exist functions $\{f_{i,j}\}_{i=1,j=1}^{m,n}$ which satisfy Assumption \ref{asm:smooth-mean-squared} and \ref{asm:smooth-local-mean-squared} with the tight smoothness parameters $\bL$ and $\bL_\ell$, respectively.
\end{enumerate}
\end{proposition}
\begin{remark}\label{remark:smooth}
We consider the tight smoothness parameters $L$, $L_\ell$, $\bL$, $\bL_\ell$, and $\Lm$ that satisfy Assumptions \ref{asm:smooth-global}--\ref{asm:smooth-max}.
Then the statements (a) and (c) of Proposition \ref{prop:smooth-2} imply~$L_\ell\geq L$ and~$\bL_\ell\geq\bL$,
and the statements (c) and (d) of Proposition \ref{prop:smooth-2} imply the ratio $L_\ell/L$ and~$\bL_\ell/\bL$ can be arbitrary large. 
Following the proof of Proposition~\ref{prop:smooth-2} (Appendix~\ref{appendix:smooth-2}), we can also show that  $\Lm$ is no smaller than $L$, $L_\ell$, $\bL$, and $\bL_\ell$, and the ratio $\Lm/L$, $\Lm/L_\ell$, $\Lm/\bL$, and $\Lm/\bL_\ell$ can be arbitrary large.
\end{remark} \vskip0.05cm
\begin{remark}
It is worth noting that ratio between local (individual) and global smoothness parameters can be bounded if all individual functions are convex, which is different from our nonconvex setting.
Please see Appendix~\ref{appendix:convex-nonconvex} for detailed discussion.
\end{remark}

Besides the general nonconvex problem, we also study the objective under the additional PL condition \cite{lojasiewicz1963topological,polyak1963gradient} as follows.

\begin{assumption}\label{asm:PL}
We suppose the objective function $f:\BR^d\to\BR$ satisfies the PL condition with the parameter $\mu>0$, i.e., we have
\begin{align}\label{eq:PL-condition}
f(\vx) - \inf_{\vy\in\BR^{d}}f(\vy) \leq \frac{1}{2\mu}\Norm{\nabla f(\vx)}^2 
\end{align}
for all $\vx\in\BR^d$. 
\end{assumption}
Under the smoothness and PL conditions, we introduce  different types of condition numbers as follows
\begin{align}\label{def:kappa}
    \kappa \triangleq \frac{L}{\mu}, \qquad 
    \bar\kappa \triangleq \frac{\bL}{\mu}, \qquad 
    \kappa_\ell \triangleq \frac{L_\ell}{\mu}, \qquad 
    \bar\kappa_\ell \triangleq \frac{\bL_\ell}{\mu}, \qquad \text{and} \qquad 
    \kappa_{\max} \triangleq \frac{\Lm}{\mu}.
\end{align}
Following the discussion in Remark \ref{remark:smooth}, the tight condition numbers hold
\begin{align*}
\kappa_{\max} \geq \kappa_\ell \geq \kappa
\qquad\text{and}\qquad
\kappa_{\max} \geq \bar\kappa_\ell \geq \bar\kappa \geq \kappa.
\end{align*}

We characterize the behavior of one communication round on the connected network with $m$ agents by multiplying the mixing matrix $\mW\in\BR^{m\times m}$ on the aggregated variable.
We impose the following standard assumption for matrix $\mW$.

\begin{assumption}[\cite{schmidt2017minimizing}]\label{asm:W}
We suppose the mixing matrix $\mW=[w_{ij}]\in\BR^{m\times m}$ is symmetric and satisfies $w_{ij}\neq 0$ if the $i$-th agent and the $j$-th agent are connected or $i=j$. Furthermore, we suppose it holds that $\mW\vone=\mW^\top\vone=\vone$ and $\vzero \preceq \mW \preceq \mI$.
\end{assumption}

Under Assumption \ref{asm:W}, the largest eigenvalue of $\mW\in\BR^{m\times m}$ is 1 and we define the spectral gap of $\mW$ as
\begin{align*}
    \gamma = 1 - \lambda_2(\mW) \in (0,1],
\end{align*}
where $\lambda_2(\mW)$ is the second-largest eigenvalue of $\mW$.

We consider the black-box optimization procedure based on the local incremental first-order oracle and decentralized communication protocol as follows.

\begin{definition}[{\cite{hendrikx2021optimal,liu2024decentralized}}]\label{dfn:LIFO}
A local incremental first-order oracle (LIFO) algorithm over a network of $m$ agents satisfies the following constraints:
\begin{itemize}
\item \textbf{Local memory:} Each agent $i$ can store past values in a local memory $\mathcal{M}_{i,s}$ at time $s>0$. 
These values can be accessed and used at time $s$ by running the algorithm on agent $i$. Additionally, for all $i\in[m]$, we have 
\begin{align*}
\mathcal{M}_i^s\subset\mathcal{M}_{{\rm comp},i}^s\bigcup\mathcal{M}_{{\rm comm},i}^s,
\end{align*}
where $\mathcal{M}_{{\rm comp},i}^s$ and $\mathcal{M}_{{\rm comm},i}^s$ are the values come from the computation and communication respectively. 
\item \textbf{Local computation:} Each agent $i$ can access its local first-order oracle $\{f_{i,j^s}(\vx), \nabla f_{i,j^s}(\vx)\}$ for given $\vx\in\mathcal{M}^s_i$ and index $j^s\in[n]$ at time $s$. That is, for all $i\in[m]$, we have
\begin{align*}
\mathcal{M}_{{\rm comp},i}^s={\rm span}\big(\{\vx,\nabla f_i(\vx):\vx\in\mathcal{M}_i^{s-1}
\}\big),
\end{align*}
where ${\rm span}(\cdot)$ is the linear span.
\item \textbf{Local communication:} Each agent $i$ can share its value to all or part of its neighbors at time $s$. That is, for all $i\in[m]$, we have
\begin{align*}
\mathcal{M}_{{\rm comm},i}^{s}={\rm span}\Bigg(\bigcup_{j\in{\rm nbr}_i}\mathcal{M}_j^{s-\tau}\Bigg),
\end{align*}
where ${\rm nbr}_i$ is the set consists of the indices for the neighbors of agent $i$ and~$\tau<s$.
\item \textbf{Output value:} Each agent $i$ can specify one vector in its memory as local output of the algorithm at time $s$. That is, for all $i\in[m]$, we have
$\vx^s_i\in\mathcal{M}_i^s$.  
\end{itemize}
\end{definition} \vskip 0.1cm

For the general nonconvex setting, we desire to achieve an $\epsilon$-stationary point for every agent in expectation, i.e., output the local variables $\vx_1,\dots,\vx_m\in\BR^d$ such that 
$\BE\Norm{\nabla f(\vx_i)} \leq \epsilon$ for all $i\in[m]$, where $\vx_i$ comes from the memory of the $i$-th agent.
For the PL condition, we desire to achieve an $\epsilon$-suboptimal solution for every agent in expectation, i.e., $\BE[f(\vx_i)-f^*]\leq \epsilon$ for all $i\in[m]$.

\section{The Algorithm and Main Results}\label{sec:alg-main}

\begin{algorithm*}[t]
\caption{\DEAREST} \label{alg:DEAREST}
\begin{algorithmic}[1]
\STATE \textbf{Input:} initial parameter $\bx^0\in\BR^d$, stepsize $\eta>0$, probability $p\in(0,1]$, mini-batch size $b$, numbers of communication rounds $\hK$ and $ K$. \\[0.1cm]
\STATE $\mX^0=\vone \bx^0$,\quad $\mG^0=\nabla\mF(\vx^0)$ \label{line:x0-g0}\\[0.1cm]
\STATE $\mS^0=\FM(\mG^0, \hK)$ \label{line:s0-g0} \\[0.1cm]
\STATE \textbf{for} $t = 0, \dots, T-1$ \textbf{do}\\[0.1cm]
\STATE\quad $\zeta^t\sim {\rm Bernoulli}(p)$ \\[0.1cm]
\STATE\quad $\mX^{t+1} = \FM(\mX^t - \eta \mS^t, \mW, K)$ \label{line:update-x}\\[0.1cm]
\STATE\quad $[\xi^t_{1,1},\dots,\xi^t_{m,n}]^\top\sim {\rm Multinomial}\left(b,q\vone\right)$ ~~\text{with}~~ $q=1/(mn)$ \\[0.1cm]
\STATE\quad \textbf{parallel for} $i = 1, \dots, m$ \textbf{do}\\[0.1cm]
\STATE\quad\quad $\displaystyle{
\vg^{t+1}_i=\begin{cases}
\nabla f_i(\vx^{t+1}_i),~~~~~~~~~~~~~~~~~~~~~~~~~~~~~~~~~~~~~~~~~~~~\,\text{if~} \zeta^t=1, \\[0.1cm]
\displaystyle{\vg^t_i + \dfrac{1}{n}\sum_{j=1}^{n}\frac{\xi^t_{i,j}}{bq}\big(\nabla f_{i,j}(\vx^{t+1}_i) - \nabla f_{i,j}(\vx^t_i)\big)},~~~~~\text{otherwise}. 
\end{cases}
}$ \label{line:update-g} \\[0.1cm]
\STATE\quad\textbf{end parallel for} \\[0.1cm]
\STATE\quad $\mS^{t+1} = \FM(\mS^t+\mG^{t+1}-\mG^t, K)$ \label{line:tracking} \\[0.1cm]
\STATE\textbf{end for} \\[0.1cm]
\STATE \textbf{Output:} 
    $\displaystyle{\begin{cases}
     \vx^{\rm out}_i\sim {\rm Uniform}(\{\vx^0_i,\vx^1_i,\dots,\vx^{T-1}_i\}), & \text{general nonconvex}, \\[0.2cm]
     \vx^{\rm out}_i = \vx^T_i, & \text{PL condition}.
 \end{cases}}$  \\[0.1cm]
\end{algorithmic}
\end{algorithm*}

\begin{algorithm}[t]
	\caption{$\FM(\mY^{0},\mW, K)$} \label{alg:fm}
	\begin{algorithmic}[1]
		\STATE \textbf{Initialize:} $\mY^{-1}=\mY^{0}$, $\eta_y=(1-\sqrt{1-\lambda_2^2(\mW)}\,)/(1+\sqrt{1-\lambda_2^2(\mW)}\,)$. \\[0.02cm]
		\STATE \textbf{for} $k = 0, 1, \dots, K$ \textbf{do}\\[0.1cm]
		\STATE\quad $\mY^{k+1}=(1+\eta_y)\mW\mY^{k}-\eta_y\mY^{k-1}$ \\[0.1cm]
		\STATE\textbf{end for} \\[0.1cm]
		\STATE \textbf{Output:} $\mY^{K}$.
	\end{algorithmic}
\end{algorithm}

We propose DEcentralized probAbilistic Recursive  gradiEnt deScenT (\DEAREST) in Algorithm~\ref{alg:DEAREST}, which is based on variance reduction \cite{fang2018spider,pham2020proxsarah,wang2019spiderboost,li2021page,zhou2018stochastic,zhou2019lower}, gradient tracking \cite{nedic2017achieving,pu2021distributed,qu2019accelerated,li2019communication}, and the multi-consensus step by Chebyshev acceleration (Algorithm~\ref{alg:fm}) \cite{arioli2014chebyshev,liu2011accelerated}.
The random variables $\zeta^t$ and $[\xi^t_{1,1},\dots,\xi^t_{m,n}]^\top$ in our algorithm are shared by all agents, which can be implemented by using the random number generators with the same seed \cite{scaman2018optimal,liu2024decentralized,li2022variance}.

The main difference between \DEAREST~and existing decentralized stochastic nonconvex optimization methods \cite{sun2020improving,xin2022fast,li2022destress,metelev2024decentralized,lu2021optimal,yuan2022revisiting} is that we use the multinomial distribution to mimic the mini-batch sampling on single machine.
Specifically, the local gradient estimators in line \ref{line:update-g} of Algorithm \ref{alg:DEAREST} indicate
\begin{align}\label{eq:gradient-estimator}
\frac{1}{m}\sum_{i=1}^m \vg^{t+1}_i 
= \frac{1}{m}\sum_{i=1}^m \vg^t_i + \frac{1}{b}\sum_{i=1}^m\sum_{j=1}^{n}\xi^t_{i,j}\big(\nabla f_{i,j}(\vx^{t+1}_i) - \nabla f_{i,j}(\vx^t_i)\big)
\end{align}
in the case of $\zeta^t=0$. 
The multinomial distribution of $(\xi^t_{1,1},\dots,\xi^t_{m,n})$ means
the right-hand side of equation (\ref{eq:gradient-estimator}) has the identical distribution to
\begin{align} \label{eq:gradient-estimator-2}   
\frac{1}{m}\sum_{i=1}^m \vg^t_i+\frac{1}{b}\sum_{j=1}^b\big(\nabla f_{i_k^t,j_k^t}(\vx^{t+1}_i) - \nabla f_{i_k^t,j_k^t}(\vx^t_i)\big),
\end{align}
where each $(i_k^t,j_k^t)$ is uniformly and independently distributed on the index set
\begin{align}\label{eq:index-set}
\{(i,j): i\in[m],\, j\in[n]\}.    
\end{align}
Furthermore, the steps of gradient tracking (line \ref{line:tracking} of Algorithm \ref{alg:DEAREST}) and multi-consensus (Algorithm \ref{alg:fm}) encourage 
$\vg^{t+1}(\bx^{t+1})$, $\vg^t(\bx^t)$, $\vx^{t+1}_i$, and $\vx^{t}_i$ to approach $\bg^{t+1}$, $\bg^t$, $\bx^{t+1}$, and $\bx^t$ respectively, where 
$\bg^t=\frac{1}{m}\sum_{i=1}^m \vg^t_i$, $\bg^{t+1}=\frac{1}{m}\sum_{i=1}^m \vg^{t+1}_i$,
$\bx^t=\frac{1}{m}\sum_{i=1}^m \vx^t_i$ and $\bx^{t+1}=\frac{1}{m}\sum_{i=1}^m \vx^{t+1}_i$.
Therefore, equation (\ref{eq:gradient-estimator}) can be regarded as an approximation of
\begin{align}\label{eq:gradient-estimator-3} 
\bg^{t+1}\
= \bg^t + \frac{1}{b}\sum_{j=1}^b\big(\nabla f_{i_k^t,j_k^t}(\bx^{t+1}) - \nabla f_{i_k^t,j_k^t}(\bx^t)\big),
\end{align}
which follows the form of the well-known stochastic gradient recursive estimator for the finite-sum problem with $mn$ individual functions on single machine \cite{fang2018spider,pham2020proxsarah,wang2019spiderboost,li2021page,zhou2018stochastic,zhou2019lower}.
Additionally, our sampling strategy does not fix the mini-batch size on each agent, which allows the algorithm obtain the LIFO complexity bounds with respect to the global mean-squared smoothness parameter $\bL$ defined in Assumption \ref{asm:smooth-local-mean-squared}.

Compared with existing stochastic recursive gradient methods \cite{fang2018spider,pham2020proxsarah,wang2019spiderboost,li2021page,zhou2018stochastic,zhou2019lower,sun2020improving,li2022destress,zhan2022efficient,xin2022fast,metelev2024decentralized}, \DEAREST~iterates with the larger mini-batch size and the higher probability to access the exact gradient by considering the difference between the global smoothness parameter $L$ and the mean-squared smoothness parameter $\bL$. That is, we take
\begin{align}\label{eq:bp}
    b=\left\lceil\frac{\sqrt{mn}\bL}{L}\,\right\rceil
    \qquad \text{and} \qquad p=\min\left\{\frac{\bL}{\sqrt{mn}L},1\right\}
\end{align}
for Algorithm \ref{alg:DEAREST}.
As a comparison, the stochastic ProbAbilistic Gradient Estimator (PAGE) method \cite{li2021page} set $b=\Theta(\sqrt{mn}\,)$ and $p=\Theta(1/\sqrt{mn}\,)$.
Intuitively, the larger batch-size and the higher probability for exact gradient computation are potential to reduce the communication rounds of the stochastic distributed algorithms but possibly increase the computational cost.
Fortunately, our analysis show that the setting of equation~(\ref{eq:bp}) achieves the near-optimal complexity bounds for communication and computation with respect to both $L$ and $\bL$.

\begin{table}[t]
\caption{We present the communication rounds and the computation rounds for finding $\epsilon$-stationary points in the general nonconvex setting, where $\Delta=f(\bx^0)-f^*$ and $\vx^0\in\BR^d$ is the initial point of the algorithm.
$^\dag$The design of \GTPAGE~focuses on the time-varying network \cite{metelev2024decentralized}. 
For the static network in our setting, the communication rounds of \GTPAGE~can looks can be improved to $\fO(L_\ell\Delta/(\sqrt{\gamma}\epsilon^2))$ by introducing the steps of Chebyshev acceleration \cite{arioli2014chebyshev}.
}\label{table:nc-1}
\begin{tabular}{cccc}
    \hline Methods &  \#Communication & \#Computation & Reference \\\hline\hline\addlinespace

    \DGET 
    & $\fO\bigg(\dfrac{\sqrt{n}\Lm\Delta}{\gamma^2\epsilon^2}\bigg)$ 
    & $\fO\bigg(n + \dfrac{\sqrt{n}\Lm\Delta}{\gamma^2\epsilon^2}\bigg)$ 
    &  \citet{sun2020improving} \\
    \addlinespace
    
    \GTSARAH 
    & $\fO\bigg(\dfrac{\bL_\ell\Delta}{\gamma^2\epsilon^2}\bigg)$ 
    & $\fO\bigg(n+\bigg(\dfrac{n}{\gamma^2}+\dfrac{m^{1/3}}{n^{1/3}\gamma}+\sqrt{\dfrac{n}{m}}\,\bigg)\dfrac{\bL_\ell\Delta}{\epsilon^2}\bigg)$  
    & \citet{xin2022fast} \\ 
    \addlinespace

    \EDSGD
    & {$\fO\bigg(\dfrac{\Lm\Delta}{\gamma\epsilon^2}\bigg)$} 
    & $\fO\bigg(n + \dfrac{\sqrt{n/m}\Lm\Delta}{\epsilon^2}\bigg)$ 
    & \citet{zhan2022efficient} \\
    \addlinespace

    \DESTRESS  
    & {$\fO\bigg(\sqrt{mn}+\dfrac{\Lm\Delta}{\sqrt{\gamma}\epsilon^2}\bigg)$} 
    & $\fO\bigg(n + \dfrac{\sqrt{n/m}\Lm\Delta}{\epsilon^2}\bigg)$ 
    & \citet{li2022destress} \\
    \addlinespace
    
    $^\dag$\GTPAGE~
    & {$\fO\bigg(\dfrac{L_\ell\Delta}{\gamma\epsilon^2}\bigg)$}
    & $\fO\bigg(n + \dfrac{\sqrt{n}\bL_\ell\Delta}{\epsilon^2}\bigg)$  
    & \!\citet{metelev2024decentralized}\!\! \\
    \addlinespace    
    
    \DEAREST 
    & {$\tilde\fO\bigg(\dfrac{L\Delta}{\sqrt{\gamma}\epsilon^2}\bigg)$}
    & $\tilde\fO\bigg( n+\dfrac{(L+\min\{nL,\sqrt{{n}/{m}}  \bar L\})\Delta}{\epsilon^2}\bigg)$ 
    &  Corollary \ref{cor:complexity-general} \\
    \addlinespace 
    
    \hline \addlinespace
    
    \!\!\!Lower Bound\!\!\! 
    & {$\Omega\bigg(\dfrac{L\Delta}{\sqrt{\gamma}\epsilon^2}\bigg)$}
    & $\Omega\bigg( n+\dfrac{(L+\min\{nL,\sqrt{{n}/{m}}\bar L\})\Delta}{\epsilon^2}\bigg)$ 
    & \!Theorems \ref{thm:lower-communication} and \ref{thm:lower-computation}\!  \\
    \addlinespace \hline
\end{tabular} 
\end{table} 

\begin{table}[t]
\caption{We present the number of LIFO calls for finding $\epsilon$-stationary points in the general nonconvex setting, where $\Delta= f(\bx^0)-f^*$ and $\vx^0\in\BR^d$ is the initial point of the algorithm.}\label{table:nc-2}\vskip-0.1cm
\begin{tabular}{ccc}
    \hline
    Methods &  \#LIFO & Reference
    \\\hline\hline\addlinespace
    
    \DGET 
    & $\fO\bigg(mn + \dfrac{m\sqrt{n}\Lm\Delta}{\gamma^2\epsilon^2}\bigg)$ 
    & \citet{sun2020improving} \\
    \addlinespace
    
    \GTSARAH 
    & $\fO\bigg(mn+\bigg(\dfrac{n}{\gamma^2}+\dfrac{m^{1/3}n^{2/3}}{\gamma}+\sqrt{mn}\bigg)\dfrac{\bL_\ell\Delta}{\epsilon^2}\bigg)$ 
    & \citet{xin2022fast} \\
    \addlinespace

    \EDSGD  
    & $\fO\bigg(mn + \dfrac{\sqrt{mn}\,\Lm\Delta}{\epsilon^2}\bigg)$ 
    & \citet{zhan2022efficient} \\
    \addlinespace

    \DESTRESS 
    & $\fO\bigg(mn + \dfrac{\sqrt{mn}\,\Lm\Delta}{\epsilon^2}\bigg)$ 
    & \citet{li2022destress} \\
    \addlinespace
    
    \GTPAGE 
    & $\fO\bigg(mn + \dfrac{m\sqrt{n}\,\bL_\ell\Delta}{\epsilon^2}\bigg)$ 
    & \citet{metelev2024decentralized} \\
    \addlinespace  
    
    \DEAREST 
    & $\fO\bigg(mn + \dfrac{\min\{mnL, \sqrt{mn}\bL\}\Delta}{\epsilon^2}\bigg)$ 
    & Corollary \ref{cor:complexity-general}  \\    
    \addlinespace 
    
    \hline \addlinespace
    
    Lower Bound 
    & $\Omega\bigg(mn + \dfrac{\min\{mnL, \sqrt{mn}\bL\}\Delta}{\epsilon^2}\bigg)$ 
    & Theorem \ref{thm:LFO} \\
    \addlinespace \hline
\end{tabular} 
\end{table}

We formally describe our main theoretical result in Theorem \ref{thm:main}, and the detailed proofs for results in this section are provided in Section \ref{sec:complexity-dearest}.

\begin{theorem}\label{thm:main}
Under Assumptions \ref{asm:lower-bounded}--\ref{asm:smooth-mean-squared}, and \ref{asm:W} with $L\leq\bar L$, we run Algorithm~\ref{alg:DEAREST} with
\begin{align*}
    & \eta = \frac{1}{8L}, \qquad b=\left\lceil\frac{\bL\sqrt{mn}}{L}\right\rceil,
    \qquad {p=\min\left\{\frac{\bL}{\sqrt{mn}L},1\right\}},  \qquad   T = \left\lceil \frac{33L\Delta}{\epsilon^2} \right\rceil, \\
    &~\quad K= \fO\left(\frac{\ln(mn\bL/L)}{\sqrt{\gamma}}\right) \qquad\text{and}\qquad
    \hK = \fO\left(\frac{\ln(mn\bL/(L\epsilon))}{\sqrt{\gamma}}\right),
\end{align*}
where $\Delta=f(\bx^0) - f^*$. Then the output satisfies
$\BE\Norm{\nabla f(\vx^{\rm out}_i)} \leq \epsilon$ for~all~$i\in[m]$.
\end{theorem} \vskip0.1cm

Based on Theorem \ref{thm:main}, we can directly obtain the 
the communication rounds and the LIFO complexity to achieve the desired approximate stationary points.
However, the setting of  $[\xi^t_{1,1},\dots,\xi^t_{m,n}]^\top\sim {\rm Multinomial}\left(b,q\vone\right)$ means the agents may access different numbers of LIFO in each iteration of \DEAREST.
Therefore, the computation rounds and the LIFO complexity should be analyzed separately.
We can bound the expectation of the computation rounds by using the concentration inequality \cite{motwani1995randomized,liu2024decentralized}.
Finally, we achieve three kinds of upper complexity bounds for \DEAREST~as follows.

\begin{corollary}\label{cor:complexity-general}
Under the assumptions and settings of Theorem \ref{thm:main}, Algorithm~\ref{alg:DEAREST} can find an expected $\epsilon$-stationary point at every agent with the communication rounds of 
\begin{align*}
     \tilde\fO\left(\frac{L\Delta}{\sqrt{\gamma}\epsilon^{2}}\right),
\end{align*}
the expected LIFO complexity of
\begin{align}\label{eq:complexity-lifo}
\fO\left(mn+ \frac{\min\{mnL,\,\sqrt{mn}\bL\}\Delta}{\epsilon^2}\right),
\end{align}
and the expected computation rounds of
\begin{align}\label{eq:complexity-compuation-rounds}
\tilde\fO\left(n+\frac{(L+\min\{nL, \sqrt{n/m}  \bar L\})\Delta}{\epsilon^2}\right).
\end{align}
\end{corollary}

We compare our results with related work in Table \ref{table:nc-1} and \ref{table:nc-2}, where we use the notation $\tilde\fO(\cdot)$ to hide the logarithmic term with respect to $L$, $\bL$, $m$, and $n$. 
Proposition~\ref{prop:smooth-2} indicates all of our  complexity bounds are sharper than the ones in existing work. \vskip0.1cm

\begin{remark}\label{remark:not=proportion}
    In the case of $\bL/L<\sqrt{m/n}$, the expected mini-batch size $b=\lceil{\bL\sqrt{mn}}/{L}\rceil$ is smaller than the number of agents $m$, which leads to  several agents may not perform the LIFO calls in a iteration. 
    The relation between $\bL$ and $L$ in this case also implies the term $L$ before $\min\{nL, \sqrt{{n}/{m}} \bar L\}$ in equation (\ref{eq:complexity-compuation-rounds}) cannot be omitted.  
    Hence, the upper bounds on the computation rounds shown in equation (\ref{eq:complexity-compuation-rounds}) is not always proportion to the LIFO complexity shown in equation (\ref{eq:complexity-lifo}). 
\end{remark}

\begin{remark}
The very recent proposed \GTPAGE~\cite{metelev2024decentralized} considers the time-varying network and takes the local mini-batch size $b_\ell=\Theta(\sqrt{n}\bL_\ell/L_\ell)$ on every agent (corresponds to mini-batch size $\Theta(m\sqrt{n}\bL_\ell/L_\ell)$ in total), which ignores the naturally existing global smoothness parameters in the problem.
In contrast, the mini-batch size $b=\lceil{\bL\sqrt{mn}}/{L}\rceil$ used in \DEAREST~is with respect to all agents, which fully leverages the global properties of the objective function to achieve the sharper dependency on smoothness parameters and the number of agents~(see Tables \ref{table:nc-1} and \ref{table:nc-2}).  
\end{remark} \vskip0.05cm

\begin{remark}\label{remark:IFO-tradeoff}
    In the case of $m=1$ (single machine scenario), Corollary \ref{cor:complexity-general} indicates the IFO complexity of $\fO(n+ \min\{nL,\,\sqrt{n}\bL\}\Delta\epsilon^{-2})$.
    Recall that the IFO complexity of vanilla gradient descent \cite{nesterov2018lectures} and stochastic recursive gradient methods \cite{nguyen2017sarah,fang2018spider,pham2019proxsarah,li2021page,carmon2020first} are~$\fO(nL\Delta\epsilon^{-2})$ and $\fO(n+\sqrt{n}\bL\Delta\epsilon^{-2})$, respectively.
    Hence, vanilla gradient descent has the sharper IFO complexity bound than stochastic recursive gradient methods in the case of~$\bL>\sqrt{n}L$.     
    Furthermore, Proposition \ref{prop:smooth}(c) implies the complexity of stochastic recursive gradient methods can be arbitrary more expensive than vanilla gradient descent.
    In other words, the optimality of the stochastic recursive gradient methods \cite{fang2018spider,zhou2019lower} no longer holds if we distinguish between the global smoothness parameter $L$ and the mean-square smoothness parameter $\bL$.
\end{remark}

\section{The Complexity Analysis for DEAREST$^+$}\label{sec:complexity-dearest} 

We start the complexity analysis of \DEAREST~(Algorithm \ref{alg:DEAREST}) by introduce the following quantities:

\begin{itemize}[itemsep=0.08cm]
\item the global gradient estimation error 
\begin{align}\label{eq:def-U}
U^{t} \triangleq \Norm{\frac{1}{m}\sum_{i=1}^m\left(\vg^{t}_i-\nabla f_i(\vx^{t}_i)\right)}^2,    
\end{align}
\item the local gradient estimation error 
\begin{align}\label{eq:def-V}
V^{t} \triangleq \frac{1}{m}\Norm{\mG^{t}-\nabla \mF(\mX^{t})}^2=\frac{1}{m}\sum_{i=1}^m\Norm{\vg^t_i-\nabla f(\vx^t_i)},  
\end{align}
\item the consensus error 
\begin{align}\label{eq:def-C}
C^{t} \triangleq \Norm{\mX^{t} - \vone\bx^{t}}^2 + \eta^2\Norm{\mS^{t} - \vone\bs^{t}}^2.    
\end{align}
\end{itemize}
We then define the Lyapunov function
\begin{align}\label{eq:Lyapunov}
\Phi^t \triangleq  f(\bx^{t}) - f^* + \frac{2\eta}{p}U^{t} + \frac{\eta}{m^3n^3bp} V^{t} + \frac{132m^2n^3\bar L^2\eta}{p} C^{t}.
\end{align}
The remains of this section is organized as follows.
In Section \ref{sec:basic-lemma}, we provide some basic lemmas.
In Section \ref{sec:recursion}, we establish the recursion for the Lyapunov function.
In Sections \ref{sec:thm:main} and \ref{sec:cor:complexity-general}, we formally prove Theorem \ref{thm:main} and Corollary \ref{cor:complexity-general}, respectively.

\subsection{Some Basic Lemmas}\label{sec:basic-lemma}

We provide some basic lemma for the later analysis of Algorithm \ref{alg:DEAREST}.

\begin{lemma}[{\cite[Lemma 3]{ye2023multi}}]\label{lem:avg-norm}
For any $\mZ\in\BR^{m\times d}$, we have
\begin{align}\label{eq:avg-norm}
\Norm{\mZ - \vone\bz} \leq \Norm{\mZ}
\qquad\text{where}\quad\bz=\frac{1}{m}\vone^\top\mZ.
\end{align}
\end{lemma}

\begin{lemma}\label{lem:stst12}
For Algorithm \ref{alg:DEAREST}, we have
\begin{align}\label{eq:st-gt}
    \bs^t=\bg^t.
\end{align}
\end{lemma}
\begin{proof}
We use induction to prove this lemma by following the analysis in decentralized convex optimization {\cite[Lemma 2]{ye2023multi}}. 
The steps in lines \ref{line:x0-g0} and \ref{line:s0-g0} of Algorithm~\ref{alg:DEAREST} and Proposition \ref{prop:FM} means $\bs^0=\bg^0$.
Suppose that we have~$\bs^t=\bg^t$, then it holds
\begin{align*}
\bs^{t+1} 
& \overset{\hphantom{(\ref{eq:chebyshev-average})}}{=}\frac{1}{m}\vone^\top\FM(\mS^t+\mG^{t+1}-\mG^t, K) \\
& \overset{(\ref{eq:chebyshev-average})}{=} \bs^t+\bg^{t+1}-\bg^t = \bg^{t+1},
\end{align*}
where the first equality is based on line \ref{line:tracking} of Algorithm \ref{alg:DEAREST}; the second equality is based on Proposition \ref{prop:FM}; the third equality is based on inductive hypothesis $\bs^t=\bg^t$.
\end{proof}

The multi-consensus steps of Algorithm \ref{alg:fm} holds the following proposition.

\begin{proposition}[{\cite[Proposition 1]{ye2023multi}}]\label{prop:FM}
Under Assumption \ref{asm:W}, Algorithm~\ref{alg:fm} holds that
\begin{align}\label{eq:chebyshev-average}
\frac{1}{m}\vone^\top\mY^{(K)} = \by^{0} 
\end{align}
and
\begin{align}\label{eq:chebyshev-converge}
\big\|\mY^{(K)}-\vone\by^{0}\big\| \leq  c_1\big(1-c_2\sqrt{\gamma}\,\big)^K \big\|\mY^{0}-\vone\by^{0}\big\|,
\end{align}
where $\by^{0} = \frac{1}{m}\vone^\top\mY^{0}$, $c_1=\sqrt{14}$, and $c_2=1-1/\sqrt{2}$.
\end{proposition} \vskip0.1cm

We describe the decease of function value as follows.

\begin{lemma}[Lemma 2 of \citet{li2021page}\footnote{\citet{li2021page} assume the individual functions are mean-squared smooth. In fact, all steps in the proof of this lemma still hold even if we only impose the global smoothness assumption (Assumption~\ref{asm:smooth-global}).}]\label{lem:descent-Li}
Under Assumption \ref{asm:smooth-global}, the update
\begin{align*}
    \vx^+ = \vx - \eta \vs
\end{align*}
for all $\vx,\vs\in\BR^d$ and $\eta>0$ holds that
\begin{align}\label{eq:descent-Li}
\begin{split}
f(\vx^+)
\leq f(\vx) - \frac{\eta}{2}\Norm{\nabla f(\vx)}^2 - \left(\frac{1}{2\eta}-\frac{L}{2}\right)\Norm{\vx^+ -\vx}^2 + \frac{\eta}{2}\Norm{\vs-\nabla f(\vx)}^2.
\end{split}
\end{align}
\end{lemma}

\begin{lemma}\label{lem:decrease-f}
Algorithm~\ref{alg:DEAREST} holds that
\begin{align}\label{ieq:decrease-f}
f(\bx^{t+1})
\leq f(\bx^t)  - \frac{\eta}{2}\Norm{\nabla f(\bx^t)}^2  + \eta U^t +  n\bL^2\eta   C^t - \left(\frac{1}{2\eta}-\frac{L}{2}\right)\Norm{\bx^{t+1}-\bx^t}^2.
\end{align}
\end{lemma}
\begin{proof}
According to Proposition \ref{prop:FM} and line \ref{line:tracking} of Algorithm \ref{alg:DEAREST}, we have
\begin{align}\label{eq-update:avg}
\begin{split}    
\bx^{t+1} &\overset{\hphantom{(\ref{eq:chebyshev-average})}}{=}
 \frac{1}{m}\vone^\top\FM(\mX^t-\eta\mS^t,K) \\
& \overset{(\ref{eq:chebyshev-average})}{=}  \frac{1}{m}\vone^\top(\mX^t-\eta\mS^t) \\
 & \overset{\hphantom{(\ref{eq:chebyshev-average})}}{=}
 \bx^t - \eta \bs^t \\
& \overset{(\ref{eq:st-gt})}{=}  \bx^t - \eta \bg^t,
\end{split}
\end{align}
where the last step is based on Lemma \ref{lem:stst12}.

According to Lemma \ref{lem:descent-Li} with $\vx=\bx^t$, $\vx^+=\bx^{t+1}$, and $\vs=\bs^t$, we have
\begin{align}\label{ieq:decrease-f1}
\begin{split}
f(\bx^{t+1})
\overset{(\ref{eq:descent-Li}),\,(\ref{eq-update:avg})}{\leq} &  f(\bx^t) - \frac{\eta}{2}\Norm{\nabla f(\bx^t)}^2 \\
& - \left(\frac{1}{2\eta}-\frac{L}{2}\right)\Norm{\bx^{t+1}-\bx^t}^2 + \frac{\eta}{2}\Norm{\bs^t-\nabla f(\bx^t)}^2.
\end{split}
\end{align}
We bound the last term of equation (\ref{ieq:decrease-f1}) as follows
\begin{align}\label{ieq:error-grad}
\begin{split}
&  \Norm{\bs^t-\nabla f(\bx^t)}^2 \\
\overset{(\ref{eq:st-gt})}{=}& \Norm{\frac{1}{m}\sum_{i=1}^m \big(\vg^t_i-\nabla f_i(\bx^t)\big)}^2 \\
\leq\,\, &  2\Norm{\frac{1}{m}\sum_{i=1}^m\big(\vg^t_i-\nabla f_i(\vx^t_i)\big)}^2 +  2\Norm{\frac{1}{m}\sum_{i=1}^m\big(\nabla f_i(\vx^t_i)-\nabla f_i(\bx^t)\big)}^2 \\
\leq\,\, & 2\Norm{\frac{1}{m}\sum_{i=1}^m\big(\vg^t_i-\nabla f_i(\vx^t_i)\big)}^2 +  \frac{2}{m}\sum_{i=1}^m\Norm{\nabla f_i(\vx^t_i)-\nabla f_i(\bx^t)}^2 \\
\overset{(\ref{eq:smooth-individual2})}{\leq}\, & 2\Norm{\frac{1}{m}\sum_{i=1}^m\big(\vg^t_i-\nabla f_i(\vx^t_i)\big)}^2 +  \frac{2mn\bL^2}{m}\sum_{i=1}^m\Norm{\vx^t_i-\bx^t}^2 \\
\overset{(\ref{eq:def-aggregate-x})}{=}\, & 2\Norm{\frac{1}{m}\sum_{i=1}^m\big(\vg^t_i-\nabla f_i(\vx^t_i)\big)}^2 +  2n\bL^2\Norm{\mX^t-\vone\bx^t}^2 \\
=\,\, & 2U^t +  2n\bL^2C^t,
\end{split}
\end{align}
where the equality is based on Lemma \ref{lem:stst12}; the first inequality uses Young's inequality; the second inequality uses the fact $\Norm{\frac{1}{m}\sum_{i=1}^m \va_i}^2 \leq \frac{1}{m} \sum_{i=1}^m \Norm{\va_i}^2$ for all $a_1,\dots,a_m\in\BR^d$; the third inequality is based on Proposition \ref{prop:smooth}(a) which implies each $f_i$ is $\sqrt{mn}\bL$-smooth; the last step is based on the definitions of $U^t$ and $C^t$ as equations (\ref{eq:def-U}) and (\ref{eq:def-C}).

We finish the proof by combining the results of (\ref{ieq:decrease-f1}) and~(\ref{ieq:error-grad}).
\end{proof}

\subsection{The Recursion for the Lyapunov Function}\label{sec:recursion}

Following the notations of Algorithm \ref{alg:DEAREST} and Proposition \ref{prop:FM},
we define 
\begin{align}\label{eq:def-rho}
\rho\triangleq c_1\big(1-c_2\sqrt{\gamma}\,\big)^{K}<1
\qquad\text{and}\qquad
\hrho \triangleq c_1\big(1-c_2\sqrt{\gamma}\,\big)^{\hK}<1
\end{align}
to characterize the convergence of Algorithm~\ref{alg:fm}, where the inequalities $\rho<1$ and $\hat\rho<1$ can be guaranteed by the settings of $K$ and $\hat K$ in Theorem~\ref{thm:main} and we present their detailed expressions in the proof of Corollary \ref{cor:complexity-general} (Section \ref{sec:thm:main}).

We then provide recursions for quantities $C^t$, $U^t$, and $V^t$ in following lemmas.

\begin{lemma}\label{lem:consensus}
Under the setting of Theorem~\ref{thm:main}, we have
\begin{align*}
    \BE[C^{t+1}]
\leq & 2\rho^2(27m^3n^3\bL^2\eta^2 + 2) \BE[C^t]  
 + 4\rho^2pm\eta^2\BE[V^t]  \\
 & + 18\rho^2m^4n^3\bL^2\eta^2\BE\Norm{\bx^{t+1}-\bx^t}^2. 
\end{align*}
\end{lemma}
\begin{proof}
The update of $\vx^{t+1}$ (line \ref{line:update-x} of Algorithm \ref{alg:DEAREST}) means
\begin{align}\label{eq:x-bx}
\begin{split}    
   & \Norm{\mX^{t+1} - \vone\bx^{t+1}} \\
\quad\,=\,\quad & \Norm{\FM(\mX^t - \eta\mS^t,K^t) - \frac{1}{m}\vone\vone^\top\FM(\mX^t - \eta\mS^t,K^t)} \\
\overset{(\ref{eq:chebyshev-converge}),\,(\ref{eq:def-rho})}{\leq}  & \rho \Norm{(\mX^t - \eta\mS^t)  - \frac{1}{m}\vone\vone^\top\left(\mX^t - \eta\mS^t\right)} \\
\quad\,=\,\quad  & \rho \Norm{\mX^t - \eta\mS^t  - \vone(\bx^t - \eta\bs^t)} \\
\quad\,\leq\,\quad  & \rho\big(\Norm{\mX^t - \vone\bx^t} + \eta\Norm{\mS^{t} - \vone\bs^{t}}\big) 
\end{split}
\end{align}
where the first inequality is based on Proposition~\ref{prop:FM} and the definition of $\rho$;  the last step is based on triangle inequality.

Consequently, we applying Young's inequality and equation (\ref{eq:x-bx}) to obtain
\begin{align}\label{recursive:x}
\begin{split}
  & \Norm{\mX^{t+1} - \vone\bx^{t+1}}^2 \\
\overset{(\ref{eq:x-bx})}{\leq} & 2\rho^2\Norm{\mX^t - \vone\bx^t}^2 + 2\rho^2\eta^2\Norm{\mS^{t} - \vone\bs^{t}}^2 \\
\overset{(\ref{eq:def-C})}{=} & 2\rho^2C^t 
\overset{(\ref{eq:def-rho})}{\leq} 2C^t.
\end{split}
\end{align}
The update of $\vg^{t+1}$ (line \ref{line:update-g} of Algorithm \ref{alg:DEAREST}) means
\begin{align}\label{ieq:rho_g}
\begin{split}
  & \BE\Norm{\vg^{t+1}_i-\vg^t_i}^2 \\
= & p\BE\Norm{\nabla f_i(\vx^{t+1}_i)-\vg^t_i}^2 \\
 & + (1-p) 
\BE\Norm{\frac{1}{n}\sum_{j=1}^{n}\dfrac{\xi_{i,j}^t}{bq}\big(\nabla f_{i,j}(\vx^{t+1}_i) - \nabla f_{i,j}(\vx^t_i)\big)}^2.
\end{split}
\end{align}
For the first term on the right-hand side of equation (\ref{ieq:rho_g}), we have 
\begin{align}\label{eq:gt-1-2}
\begin{split}    
  & \BE\Norm{\nabla f_i(\vx^{t+1}_i)-\vg^t_i}^2 \\  
\leq\, & 2\BE\Norm{\nabla f_i(\vx^{t+1}_i)-\nabla f_i(\vx^t_i)}^2
 + 2\BE\Norm{\nabla f_i(\vx^t_i)-\vg^t_i}^2 \\
\overset{(\ref{eq:smooth-individual2})}{\leq} & 2mn\bL^2\BE\Norm{\vx^{t+1}_i-\vx^t_i}^2 + \BE\Norm{\nabla f_i(\vx^t_i)-\vg^t_i}^2,
\end{split}
\end{align}
where the first step is based on Young's inequality; the second step is based on Proposition \ref{prop:smooth}(a).

For the second term on the right-hand side of equation (\ref{ieq:rho_g}), we have 
\begin{align}\label{eq:gt-2-2}
\begin{split}
 & \BE\Norm{\frac{1}{n}\sum_{j=1}^{n}\dfrac{\xi_{i,j}^t}{bq}\big(\nabla f_{i,j}(\vx^{t+1}_i) - \nabla f_{i,j}(\vx^t_i)\big)}^2
 \\  
\leq\, & \frac{1}{n}\sum_{j=1}^{n}\BE\Norm{\dfrac{\xi_{i,j}^t}{bq}\big(\nabla f_{i,j}(\vx^{t+1}_i) - \nabla f_{i,j}(\vx^t_i)\big)}^2 \\
\leq\, & \frac{1}{n}\sum_{j=1}^{n}\frac{1}{q^2}\BE\Norm{\nabla f_{i,j}(\vx^{t+1}_i) - \nabla f_{i,j}(\vx^t_i)}^2 \\
\overset{(\ref{eq:smooth-individual1})}{\leq} & \frac{1}{n}\sum_{j=1}^{n}\frac{mn\bL^2}{q^2}\BE\Norm{\vx^{t+1}_i - \vx^t_i}^2 \\
=\, & m^3n^3\bL^2\BE\Norm{\vx^{t+1}_i - \vx^t_i}^2,
\end{split}
\end{align}
where the first inequality is based on the fact  $\|\frac{1}{n}\sum_{i=1}^n \va_i\|^2\leq \frac{1}{n}\sum_{i=1}^n \va_i^2$ for all $a_1,\dots,a_n$; 
the second inequality is based on fact $\xi_{i,j}^t\leq b$; the last inequality is based on Proposition \ref{prop:smooth}(a) and the setting $q=1$.

Combining equations (\ref{ieq:rho_g}), (\ref{eq:gt-1-2}), and (\ref{eq:gt-2-2}), we achieve
\begin{align}\label{ieq:rho_g2-2}
\begin{split}
  & \BE\Norm{\vg^{t+1}_i-\vg^t_i}^2 \\
\leq & p\big(2mn\bL^2\BE\Norm{\vx^{t+1}_i-\vx^t_i}^2 + 2\BE\Norm{\nabla f_i(\vx^t_i)-\vg^t_i}^2\big) \\
 & + (1-p) m^3n^3\bL^2\BE\Norm{\vx^{t+1}_i - \vx^t_i}^2 \\
\leq & 2p\BE\Norm{\nabla f_i(\vx^t_i)-\vg^t_i}^2
 + 3m^3n^3\bL^2\BE\Norm{\vx^{t+1}_i - \vx^t_i}^2.
\end{split}
\end{align}

Summing equation (\ref{ieq:rho_g2-2}) over $i=1,\dots,m$, we obtain
\begin{align}\label{inq:g^t-g^t+1-2}
\begin{split}
&   \BE\Norm{\mG^{t+1}-\mG^t}^2 \\
=\,\, & \sum_{i=1}^m\BE\Norm{\vg^{t+1}_i-\vg^t_i}^2 \\
\overset{(\ref{inq:g^t-g^t+1-2})}{\leq} & 2p\sum_{i=1}^m\BE\Norm{\nabla f_i(\vx^t_i)-\vg^t_i}^2
 +  3m^3n^3\bL^2\sum_{i=1}^m\BE\Norm{\vx^{t+1}_i - \vx^t_i}^2 \\
\overset{(\ref{eq:def-V})}{=} &  2pm\BE[V^t] + 3m^3n^3\bL^2\BE\Norm{\vx^{t+1}-\vx^t}^2 \\
\leq\,\, & 2pm\BE[V^t] 
 + 9m^3n^3\bL^2\BE\left[\Norm{\vx^{t+1}-\vone\bx^{t+1}}^2 + \Norm{\vone\bx^{t+1}-\vone\bx^t}^2+\Norm{\mX^t-\vone\bx^t}^2\right] \\
\overset{(\ref{eq:x-bx})}\leq &  2pm\BE[V^t] + 9m^3n^3\bL^2\BE\left[2\rho^2\big(\Norm{\mX^t - \vone\bx^t}^2 + \eta^2\Norm{\mS^{t} - \vone\bs^{t}}^2\big)\right] \\
 & + 9m^3n^3\bL^2\BE\left[m\Norm{\bx^{t+1}-\bx^t}^2 + \Norm{\mX^t-\vone\bx^t}^2\right] \\
\overset{(\ref{eq:def-rho})}{\leq} &  2pm\BE[V^t]
 + 27m^3n^3\bL^2\Norm{\mX^t - \vone\bx^t}^2 
 + 9m^3n^3\bL^2\eta^2\Norm{\mS^{t} - \vone\bs^{t}}^2 \\
 & + 9m^4n^3\bL^2\BE\Norm{\bx^{t+1}-\bx^t}^2,
\end{split}
\end{align}
where the second inequality is based on Young's inequality; the third inequality uses the result of (\ref{eq:x-bx}); the last inequality is based on the fact $\rho<1$.

We also have
\begin{align}\label{inq:s^t+1-2}
\begin{split}
   &\! \Norm{\mS^{t+1} - \vone\bs^{t+1}}^2 \\
\,\quad=\,\quad & \Big\|\FM(\mS^t + \mG^{t+1} - \mG^t, K) - \frac{1}{m}\vone\vone^\top\FM(\mS^t + \mG^{t+1} - \mG^t, K)\Big\|^2 \\
\overset{(\ref{eq:chebyshev-converge}),\,(\ref{eq:def-rho})}{\leq} & \rho^2\Big\|\mS^t + \mG^{t+1} - \mG^t - \frac{1}{m}\vone\vone^\top(\mS^t + \mG^{t+1} - \mG^t)\Big\|^2 \\
\,\quad=\,\quad & 2\rho^2\Norm{\mS^t - \vone\bs^t}^2 + 2\rho^2\Big\|\mG^{t+1} - \mG^t - \frac{1}{m}\vone\vone^\top(\mG^{t+1} - \mG^t)\Big\|^2 \\
\,~~\overset{(\ref{eq:avg-norm})}{\leq}~~\, & 2\rho^2\Norm{\mS^t - \vone\bs^t}^2 + 2\rho^2\Norm{\mG^{t+1} - \mG^t}^2,
\end{split}
\end{align}
where the first inequality is based on Proposition \ref{prop:FM} and the definition of $\rho$; the second inequality is based on Young's inequality and the last step is based on  Lemma~\ref{lem:avg-norm}. 

Combining equations (\ref{inq:g^t-g^t+1-2}) and (\ref{inq:s^t+1-2}), we have 
\begin{align}\label{recursive:s}
\begin{split}
  & \BE\Norm{\mS^{t+1} - \vone\bs^{t+1}}^2 \\
\overset{(\ref{inq:g^t-g^t+1-2}),\,(\ref{inq:s^t+1-2})}{\leq}\, & 2\rho^2(9m^3n^3\bL^2\eta^2+1)\BE\Norm{\mS^{t} - \vone\bs^{t}}^2  + 4\rho^2pm\BE[V^t] \\
 & + 54\rho^2m^3n^3\bL^2\BE\Norm{\mX^t - \vone\bx^t}^2 
 + 18\rho^2m^4n^3\bL^2\BE\Norm{\bx^{t+1}-\bx^t}^2.
\end{split}
\end{align}

Consequently, results of equations (\ref{recursive:x}) and (\ref{recursive:s}) indicate 
\begin{align*}
  &  \BE[C^{t+1}] \\
\,~~\overset{(\ref{eq:def-C})}{=}\,~~\! & \BE\big[\Norm{\mX^{t+1} - \vone\bx^{t+1}}^2 + \eta^2\Norm{\mS^{t+1} - \vone\bs^{t+1}}^2\big] \\
\overset{(\ref{recursive:x}),\,(\ref{recursive:s})}{\leq}\! &  2\rho^2\BE\Norm{\mX^t - \vone\bx^t}^2 + 2\rho^2\eta^2\BE\Norm{\mS^{t} - \vone\bs^{t}}^2 \\
& + 2\rho^2(9m^3n^3\bL^2\eta^2+1)\eta^2\BE\Norm{\mS^{t} - \vone\bs^{t}}^2  + 4\rho^2pm\eta^2\BE[V^t] \\
 & + 54\rho^2m^3n^3\bL^2\eta^2\BE\Norm{\mX^t - \vone\bx^t}^2 
 + 18\rho^2m^4n^3\bL^2\eta^2\BE\Norm{\bx^{t+1}-\bx^t}^2 \\
\quad=\quad\! & 2\rho^2(27m^3n^3\bL^2\eta^2 + 1) \BE\Norm{\mX^t - \vone\bx^t}^2 + 2\rho^2(9m^3n^3\bL^2\eta^2+2)\eta^2\BE\Norm{\mS^{t} - \vone\bs^{t}}^2   \\
 & + 4\rho^2pm\eta^2\BE[V^t] + 18\rho^2m^4n^3\bL^2\eta^2\BE\Norm{\bx^{t+1}-\bx^t}^2 \\
\quad=\quad\! & 2\rho^2(27m^3n^3\bL^2\eta^2 + 2) \BE[C^t]  
 + 4\rho^2pm\eta^2\BE[V^t] + 18\rho^2m^4n^3\bL^2\eta^2\BE\Norm{\bx^{t+1}-\bx^t}^2\!\!, 
\end{align*}
which finishes the proof.
\end{proof}

\begin{lemma}\label{lem:global}
Under the setting of Theorem~\ref{thm:main}, we have 
\begin{align*}
  \BE[U^{t+1}] 
\leq (1-p)\BE[U^t] + 9(1-p)m^2n^3\bar L^2\BE[C^t] + \frac{3(1-p)\bL^2}{b}\BE\Norm{\bx^{t+1}-\bx^t}^2.
\end{align*}
\end{lemma}
\begin{proof}
The setting $[\xi^t_{1,1},\dots,\xi^t_{m,n}]^\top\sim {\rm Multinomial}\left(b,q\vone\right)$ with $q=1/(mn)$ implies that we have
\begin{align*}
    \xi^t_{i,j}\sim{\rm Binomial}(b,q),
\end{align*} 
for all $i\in[m]$ and $j\in[n]$. This leads to
\begin{align}\label{eq:U-expecation0}
\begin{split}    
    & \BE_{\xi_{i,j}^t}\left[\dfrac{1}{mn}\sum_{i=1}^{m}\sum_{j=1}^{n}\dfrac{\xi^t_{i,j}}{bq}\big(\nabla f_{i,j}(\vx^{t+1}_i) - \nabla f_{i,j}(\vx^t_i)\big)\right]  \\
    = & \dfrac{1}{mn}\sum_{i=1}^{m}\sum_{j=1}^{n}\dfrac{\BE\left[\xi^t_{i,j}\right]}{bq}\big(\nabla f_{i,j}(\vx^{t+1}_i) - \nabla f_{i,j}(\vx^t_i)\big)  \\ 
    = & \frac{1}{mn}\sum_{i=1}^{m}\sum_{i=1}^n\big(\nabla f_{i,j}(\vx^{t+1}_i) - \nabla f_{i,j}(\vx^t_i\big) \\
    = & \frac{1}{m}\sum_{i=1}^{m}\big(\nabla f_{i}(\vx^{t+1}_i) - \nabla f_{i}(\vx^t_i)\big),
\end{split}
\end{align}
where we use the fact $\BE[\xi_{i,j}^t]=bq$ for all $i\in[m]$ and $j\in[n]$.

Then the update of $\vg^{t+1}_i$ (line \ref{line:update-g} of Algorithm \ref{alg:DEAREST}) means
\begin{align}\label{eq:U-initial}
\small\begin{split}    
 & \BE[U^{t+1}] \\
~=~& p\BE\Norm{\frac{1}{m}\sum_{i=1}^m(\nabla f_i(\vx^{t+1}_i-\nabla f_i(\vx^{t+1}_i))}^2 \\
& + (1-p)\BE\Norm{\frac{1}{m}\sum_{i=1}^m\!\left(\!\vg^t_i\!+\!\dfrac{1}{n}\sum_{j=1}^{n}\frac{\xi^t_{i,j}}{bq}\big(\nabla f_{i,j}(\vx^{t+1}_i)\!-\!\nabla f_{i,j}(\vx^t_i)\big)\!-\!\nabla f_i(\vx^{t+1}_i)\!\right)\!}^2 \\
\overset{(\ref{eq:U-expecation0})}{=} & (1-p)\BE\Norm{\frac{1}{m}\sum_{i=1}^m\left(\vg^t_i - \nabla f_i(\vx^t_i)\right)}^2 \\
& + (1-p)\BE\Bigg\|\,\dfrac{1}{mn}\sum_{i=1}^m\sum_{j=1}^{n}\dfrac{\xi^t_{i,j}}{bq}\big(\nabla f_{i,j}(\vx^{t+1}_i) - \nabla f_{i,j}(\vx^t_i)\big) - \frac{1}{m}\sum_{i=1}^m\big(\nabla f_i(\vx^{t+1}_i)-\nabla f_i(\vx^t_i)\big)\,\Bigg\|^2 \\
\overset{(\ref{eq:def-U})}{=} & (1-p)\BE[U^t] \\
& + (1-p)\BE\Bigg\|\,\dfrac{1}{mn}\sum_{i=1}^m\sum_{j=1}^{n}\dfrac{\xi^t_{i,j}}{bq}\big(\nabla f_{i,j}(\vx^{t+1}_i) - \nabla f_{i,j}(\vx^t_i)\big) 
 - \frac{1}{m}\sum_{i=1}^m\big(\nabla f_i(\vx^{t+1}_i)-\nabla f_i(\vx^t_i)\big)\,\Bigg\|^2, 
\end{split}
\end{align}
where the second step is due to the property of Martingale~\cite[Proposition~1]{fang2018spider} and equation (\ref{eq:U-expecation0}); the last step is based on the definition of $U^t$.

We organize the terms in the second norm term in equation (\ref{eq:U-initial}) as follows
\begin{align*}   
& \dfrac{1}{mn}\sum_{i=1}^m\sum_{j=1}^{n}\dfrac{\xi^t_{i,j}}{bq}\big(\nabla f_{i,j}(\vx^{t+1}_i) - \nabla f_{i,j}(\vx^t_i)\big) \\
 & - \frac{1}{mn}\sum_{i=1}^m\sum_{j=1}^n\big(\nabla f_{i,j}(\vx^{t+1}_i)-\nabla f_{i,j}(\vx^t_i)\big) \\
= &  \dfrac{1}{mn}\sum_{i=1}^m\sum_{j=1}^{n}\dfrac{\xi^t_{i,j}}{bq}\big(\nabla f_{i,j}(\vx^{t+1}_i) - \nabla f_{i,j}(\bx^{t+1}) - (\nabla f_{i,j}(\vx^t_i) - \nabla f_{i,j}(\bx^t)\big) \\ 
& + \dfrac{1}{mn}\sum_{i=1}^m\sum_{j=1}^{n}\dfrac{\xi^t_{i,j}}{bq}\big(\nabla f_{i,j}(\bx^{t+1}) - \nabla f_{i,j}(\bx^t)\big) \\
& - \dfrac{1}{mn}\sum_{i=1}^m\sum_{j=1}^{n}\big(\nabla f_{i,j}(\vx^{t+1}_i) - \nabla f_{i,j}(\bx^{t+1}) - (\nabla f_{i,j}(\vx^t_i) - \nabla f_{i,j}(\bx^t))\big) \\  
& - \dfrac{1}{mn}\sum_{i=1}^m\sum_{j=1}^{n}\big(\nabla f_{i,j}(\bx^{t+1}) - \nabla f_{i,j}(\bx^t))\\
= &  \dfrac{1}{mn}\sum_{i=1}^m\sum_{j=1}^{n}\left(\frac{\xi^t_{i,j}}{bq}-1\right)\big(\nabla f_{i,j}(\vx^{t+1}_i) - \nabla f_{i,j}(\bx^{t+1})\big) \\
&  - \dfrac{1}{mn}\sum_{i=1}^m\sum_{j=1}^{n}\left(\frac{\xi^t_{i,j}}{bq}-1\right)\big(\nabla f_{i,j}(\vx^t_i) - \nabla f_{i,j}(\bx^t)\big) \\
& + \dfrac{1}{mn}\sum_{i=1}^m\sum_{j=1}^{n}\dfrac{\xi^t_{i,j}}{bq}\big(\nabla f_{i,j}(\bx^{t+1}) - \nabla f_{i,j}(\bx^t)\big) \\
& - \dfrac{1}{mn}\sum_{i=1}^m\sum_{j=1}^{n}\big(\nabla f_{i,j}(\bx^{t+1}) - \nabla f_{i,j}(\bx^t)).
\end{align*}
Taking the square of norm on above equation, we achieve
\begin{align}\label{eq:U2}
\small\begin{split} 
& \Norm{\dfrac{1}{mn}\sum_{i=1}^m\sum_{j=1}^{n}\dfrac{\xi^t_{i,j}}{bq}\big(\nabla f_{i,j}(\vx^{t+1}_i) - \nabla f_{i,j}(\vx^t_i)\big) - \frac{1}{mn}\sum_{i=1}^m\sum_{j=1}^n\big(\nabla f_{i,j}(\vx^{t+1}_i)-\nabla f_{i,j}(\vx^t_i)\big)}^2 \\
\leq &  3\Norm{\dfrac{1}{mn}\sum_{i=1}^m\sum_{j=1}^{n}\left(\frac{\xi^t_{i,j}}{bq}-1\right)\big(\nabla f_{i,j}(\vx^{t+1}_i) - \nabla f_{i,j}(\bx^{t+1})\big)}^2 \\
& + 3\Norm{\dfrac{1}{mn}\sum_{i=1}^m\sum_{j=1}^{n}\left(\frac{\xi^t_{i,j}}{bq}-1\right)\big(\nabla f_{i,j}(\vx^t_i) - \nabla f_{i,j}(\bx^t)\big)}^2 \\
& + 3\Norm{\dfrac{1}{mn}\sum_{i=1}^m\sum_{j=1}^{n}\dfrac{\xi^t_{i,j}}{bq}\big(\nabla f_{i,j}(\bx^{t+1}) - \nabla f_{i,j}(\bx^t)\big) - \dfrac{1}{mn}\sum_{i=1}^m\sum_{j=1}^{n}\big(\nabla f_{i,j}(\bx^{t+1}) - \nabla f_{i,j}(\bx^t))}^2 \\
\leq &  \dfrac{3}{mn}\sum_{i=1}^m\sum_{j=1}^{n}\left(\frac{\xi^t_{i,j}}{bq}-1\right)^2\Norm{\nabla f_{i,j}(\vx^{t+1}_i) - \nabla f_{i,j}(\bx^{t+1})}^2 \\
& + \dfrac{3}{mn}\sum_{i=1}^m\sum_{j=1}^{n}\left(\frac{\xi^t_{i,j}}{bq}-1\right)^2\Norm{\nabla f_{i,j}(\vx^t_i) - \nabla f_{i,j}(\bx^t)}^2 \\
& + 3\Norm{\dfrac{1}{mn}\sum_{i=1}^m\sum_{j=1}^{n}\dfrac{\xi^t_{i,j}}{bq}\big(\nabla f_{i,j}(\bx^{t+1}) - \nabla f_{i,j}(\bx^t)\big) - \dfrac{1}{mn}\sum_{i=1}^m\sum_{j=1}^{n}\big(\nabla f_{i,j}(\bx^{t+1}) - \nabla f_{i,j}(\bx^t))}^2,
\end{split}
\end{align}
where the first step is based on Young's inequality; the last step 
is based on the fact $\|\frac{1}{mn}\sum_{i=1}^m\sum_{j=1}^n \va_{i,j}\|^2 \leq \frac{1}{mn}\sum_{i=1}^m \sum_{j=1}^n \Norm{\va_{i,j}}^2$ for all $a_{1,1},\dots,a_{m,n}\in\BR^d$.

\newpage
We bound the terms in the last step in equation (\ref{eq:U2}) as follows:
\begin{itemize}
\item  For the first term, we have
\begin{align}\label{eq:U21}
\begin{split}    
  &  \dfrac{3}{mn}\sum_{i=1}^m\sum_{j=1}^{n}\left(\frac{\xi^t_{i,j}}{bq}-1\right)^2\Norm{\nabla f_{i,j}(\vx^{t+1}_i) - \nabla f_{i,j}(\bx^{t+1})}^2 \\
\overset{(\ref{eq:xi-bq-1})}{\leq} &  \dfrac{3}{mn}\sum_{i=1}^m\sum_{j=1}^{n}\frac{1}{q^2}\Norm{\nabla f_{i,j}(\vx^{t+1}_i) - \nabla f_{i,j}(\bx^{t+1})}^2  \\
\overset{\,(\ref{eq:smooth-individual1})\,}{\leq} &   \dfrac{3}{mn}\sum_{i=1}^m \sum_{j=1}^n  \frac{mn\bar L^2}{q^2}\Norm{\vx^{t+1}_i - \bx^{t+1}}^2 \\
\,=\, &   \frac{3n\bar L^2}{q^2}\Norm{\mX^{t+1} - \vone\bx^{t+1}}^2,
\end{split}
\end{align}
where the first inequality is based on the fact $\xi_{i,j}^t\in\{0,1,\dots,b\}$ that leads to
\begin{align}\label{eq:xi-bq-1}
    \left|\frac{\xi_{i,j}^t}{bq}-1\right|
\leq \max\left\{\left|\frac{b}{bq}-1\right|, 1 \right\}
\leq \left|\frac{1}{q}-1\right|+ 1 = \frac{1}{q};
\end{align}
the second inequality is based on Proposition \ref{prop:smooth}(a).
\item For the second term, we follow the derivation of equation (\ref{eq:U21}) to achieve
\begin{align}\label{eq:U22}
\begin{split}
  \dfrac{3}{mn}\sum_{i=1}^m\sum_{j=1}^{n}\left(\frac{\xi^t_{i,j}}{bq}-1\right)^2\Norm{\nabla f_{i,j}(\vx^t_i) - \nabla f_{i,j}(\bx^t)}^2 
\leq \frac{3n\bar L^2}{q^2}\Norm{\mX^t - \vone\bx^t}^2.
\end{split}
\end{align}
\item For the third term, 
the setting $[\xi^t_{1,1},\dots,\xi^t_{m,n}]^\top\sim {\rm Multinomial}\left(b,q\vone\right)$ with parameter $q=1/(mn)$ implies 
implies we have 
\begin{align*}
    \xi^t_{i,j}\sim{\rm Binomial}(b,q),
\end{align*} 
for all $i\in[m]$ and $j\in[n]$.
This leads to
\begin{align}\label{eq:U-expecation}
\begin{split}    
    & \BE_{\xi^t_{i,j}}\left[\dfrac{1}{mn}\sum_{i=1}^{m}\sum_{j=1}^{n}\dfrac{\xi^t_{i,j}}{bq}\big(\nabla f_{i,j}(\bx^{t+1}) - \nabla f_{i,j}(\bx^t)\big)\right]  \\
    = & \dfrac{1}{mn}\sum_{i=1}^{m}\sum_{j=1}^{n}\dfrac{\BE\left[\xi^t_{i,j}\right]}{bq}\big(\nabla f_{i,j}(\bx^{t+1}) - \nabla f_{i,j}(\bx^t)\big)  \\    
    = & \frac{1}{mn}\sum_{i=1}^{m}\sum_{i=1}^n\big(\nabla f_{i,j}(\bx^{t+1}) - \nabla f_{i,j}(\bx^t)\big),
\end{split}
\end{align}
where we use the fact $\BE[\xi_{i,j}^t]=bq$ for all $i\in[m]$ and $j\in[n]$.

The setting $[\xi^t_{1,1},\dots,\xi^t_{m,n}]^\top\sim {\rm Multinomial}\left(b,q\vone\right)$ with parameter $q=1/(mn)$ also implies the random vector
\begin{align}\label{eq:U-termA}
\begin{split}    
  &  \dfrac{1}{mn}\sum_{i=1}^{m}\sum_{j=1}^{n}\dfrac{\xi^t_{i,j}}{bq}\big(\nabla f_{i,j}(\bx^{t+1}) - \nabla f_{i,j}(\bx^t)\big) \\
= & \frac{1}{b}\sum_{i=1}^{m}\sum_{j=1}^{n}\xi^t_{i,j}\big(\nabla f_{i,j}(\bx^{t+1}) - \nabla f_{i,j}(\bx^t)\big)    
\end{split}
\end{align}
has the identical distribution to the random vector
\begin{align}\label{eq:U-termB}
\frac{1}{b}\sum_{k=1}^{b}\big(\nabla f_{i_k^t,j_k^t}(\bx^{t+1}) - \nabla f_{i_k^t,j_k^t}(\bx^t)\big),
\end{align}
where each pair $(i_k^t,j_k^t)$ is independently and uniformly sampled from the set
\begin{align*}
\big\{(i,j):i\in[m], j\in[n]\big\}.    
\end{align*}
In the view of equations (\ref{eq:U-expecation}), (\ref{eq:U-termA}), and (\ref{eq:U-termB}), we bound the variance of equation~(\ref{eq:U-termA}) as follows
\begin{align}\label{eq:U23}
\small\begin{split}    
  &  \BE_{\xi^t_{i,j}}\Norm{\dfrac{1}{mn}\sum_{i=1}^m\sum_{j=1}^{n}\dfrac{\xi^t_{i,j}}{bq}\big(\nabla f_{i,j}(\bx^{t+1}) - \nabla f_{i,j}(\bx^t)\big) - \dfrac{1}{mn}\sum_{i=1}^m\sum_{j=1}^{n}\big(\nabla f_{i,j}(\bx^{t+1}) - \nabla f_{i,j}(\bx^t))}^2 \\
&=   \BE_{i_k, j_k}\Norm{\frac{1}{b}\sum_{k=1}^{b}\big(\nabla f_{i_k,j_k}(\bx^{t+1}) - \nabla f_{i_k,j_k}(\bx^t)\big) - \dfrac{1}{mn}\sum_{i=1}^m\sum_{j=1}^{n}\big(\nabla f_{i,j}(\bx^{t+1}) - \nabla f_{i,j}(\bx^t))}^2  \\
&=   \frac{1}{b}\BE_{i_k, j_k}\Norm{\nabla f_{i_k,j_k}(\bx^{t+1}) - \nabla f_{i_k,j_k}(\bx^t) - \BE[\nabla f_{i_k,j_k}(\bx^{t+1}) - \nabla f_{i_k,j_k}(\bx^t)]}^2 \\
&\leq  \frac{1}{b}\BE_{i_k, j_k}\Norm{\nabla f_{i_k,j_k}(\bx^{t+1}) - \nabla f_{i_k,j_k}(\bx^t)}^2 \\
&\leq \frac{\bL^2}{b}\Norm{\bx^{t+1}-\bx^t}^2,
\end{split}
\end{align}
where the first equality is based on the distributions of equation (\ref{eq:U-termA}) and (\ref{eq:U-termB}) are identical; 
the second equality is based on the definition of $(i_k,j_k)$;
the first inequality is based on equations (\ref{eq:U-expecation}) and (\ref{eq:U-termA}); the second inequality is based on   Assumption \ref{asm:smooth-mean-squared}.
\end{itemize}

Combining the equations (\ref{eq:U2}), (\ref{eq:U21}), (\ref{eq:U22}), and (\ref{eq:U23}), we have
\begin{align}\label{eq:U2-final} 
\small\begin{split}    
& \BE\Norm{\dfrac{1}{mn}\sum_{i=1}^m\sum_{j=1}^{n}\dfrac{\xi^t_{i,j}}{bq}\big(\nabla f_{i,j}(\vx^{t+1}_i) - \nabla f_{i,j}(\vx^t_i)\big) - \frac{1}{mn}\sum_{i=1}^m\sum_{j=1}^n\big(\nabla f_{i,j}(\vx^{t+1}_i)-\nabla f_{i,j}(\vx^t_i)\big)}^2 \\
\leq &  \BE\left[\frac{3n\bar L^2}{q^2}\Norm{\mX^{t+1} - \vone\bx^{t+1}}^2 + \frac{3n\bar L^2}{q^2}\Norm{\mX^t - \vone\bx^t}^2 + \frac{3\bL^2}{b}\Norm{\bx^{t+1}-\bx^t}^2\right] \\
\leq &  \BE\left[\frac{6\rho^2n\bar L^2}{q^2}\big(\Norm{\mX^t - \vone\bx^t}^2 + \eta^2\Norm{\mS^{t} - \vone\bs^{t}}^2\big) + \frac{3n\bar L^2}{q^2}\Norm{\mX^t - \vone\bx^t}^2 + \frac{3\bL^2}{b}\Norm{\bx^{t+1}-\bx^t}^2\right] \\
\leq &  \BE\left[\frac{9n\bar L^2}{q^2}C^t + \frac{3\bL^2}{b}\Norm{\bx^{t+1}-\bx^t}^2\right],
\end{split}
\end{align}
where the second inequality is based on equation (\ref{recursive:x});
the third inequality is based on the Young's inequality;
the last inequality is based on the definition of $C^t$ and the fact~$\rho\leq 1$.

Combining equations (\ref{eq:U-initial}) and (\ref{eq:U2-final}),
we have
\begin{align*}
  \BE[U^{t+1}] 
\leq & (1-p)\BE[U^t] + \frac{9(1-p)n\bar L^2}{q^2}\BE[C^t] + \frac{3(1-p)\bL^2}{b}\BE\Norm{\bx^{t+1}-\bx^t}^2 \\
\leq & (1-p)\BE[U^t] + 9(1-p)m^2n^3\bar L^2\BE[C^t] + \frac{3(1-p)\bL^2}{b}\BE\Norm{\bx^{t+1}-\bx^t}^2,
\end{align*}
where the last step is based on the setting $q=1/(mn)$.
Therefore, we finish the proof.
\end{proof}

\begin{lemma}\label{lem:local}
Under the setting of Theorem~\ref{thm:main}, we have
\begin{align*}
\BE[V^{t+1}] 
\leq (1-p)\BE[V^t] + 36(1-p)m^2n^3\bL^2\BE[C^t] + 9(1-p)m^3n^3\bL^2\BE\Norm{\bx^{t+1}-\bx^t}^2,
\end{align*}
\end{lemma}
\begin{proof}
The setting of~$[\xi^t_{1,1},\dots,\xi^t_{m,n}]^\top\sim {\rm Multinomial}\left(b,q\vone\right)$ with $q=1/(mn)$ 
implies   
\begin{align*}
    \xi^t_{i,j}\sim{\rm Binomial}(b,q),
\end{align*} 
for all $i\in[m]$ and $j\in[n]$.
This leads to 
\begin{align}\label{eq:V-unbiased}
\begin{split}    
    & \BE_{\xi^t_{i,j}}\left[\dfrac{1}{n}\sum_{j=1}^{n}\dfrac{\xi^t_{i,j}}{bq}\big(\nabla f_{i,j}(\vx^{t+1}_i) - \nabla f_{i,j}(\vx^t_i)\big)\right]  \\
    = & \dfrac{1}{n}\sum_{j=1}^{n}\dfrac{\BE\left[\xi^t_{i,j}\right]}{bq}\big(\nabla f_{i,j}(\vx^{t+1}_i) - \nabla f_{i,j}(\vx^t_i)\big)  \\ 
    = & \frac{1}{n}\sum_{i=1}^n\big(\nabla f_{i,j}(\vx^{t+1}_i) - \nabla f_{i,j}(\vx^t_i)\big) 
    =  \nabla f_{i}(\vx^{t+1}_i) - \nabla f_{i}(\vx^t_i).
\end{split}
\end{align}
The update of $\vg^{t+1}_i$ (line \ref{line:update-g} of Algorithm \ref{alg:DEAREST}) means 
\begin{align}\label{eq:V-initial}
\begin{split}
& \BE\Norm{\vg^{t+1}_i-\nabla f_i(\vx^{t+1}_i)}^2 \\
\,\,=\,\, & p\BE\Norm{\nabla f_i(\vx^{t+1}_i)-\nabla f_i(\vx^{t+1}_i)}^2 \\
& + (1-p)\BE\Norm{\vg^t_i + \dfrac{1}{n}\sum_{j=1}^{n}\frac{\xi^t_{i,j}}{bq}\big(\nabla f_{i,j}(\vx^{t+1}_i) - \nabla f_{i,j}(\vx^t_i)\big) -\nabla f_i(\vx^{t+1}_i)}^2 \\
\overset{(\ref{eq:V-unbiased})} = & (1-p)\BE\Norm{\vg^t_i - \nabla f_i(\vx^t_i)}^2 \\
& + (1-p)\BE\Norm{\dfrac{1}{n}\sum_{j=1}^{n}\frac{\xi^t_{i,j}}{bq}\big(\nabla f_{i,j}(\vx^{t+1}_i) - \nabla f_{i,j}(\vx^t_i)\big) - \big(\nabla f_i(\vx^{t+1}_i) - \nabla f_i(\vx^t_i)\big)}^2, 
\end{split}
\end{align}
where the second equality uses the property of Martingale~\cite[Proposition~1]{fang2018spider} and equation (\ref{eq:V-unbiased}).

Similar to the derivation of equation (\ref{eq:U2}) in the proof of Lemma \ref{lem:global} , we can bound the last term in equation (\ref{eq:V-initial}) as 
\begin{align}\label{eq:V2}
\small\begin{split} 
& \Norm{\dfrac{1}{n}\sum_{j=1}^{n}\dfrac{\xi^t_{i,j}}{bq}\big(\nabla f_{i,j}(\vx^{t+1}_i) - \nabla f_{i,j}(\vx^t_i)\big) - \frac{1}{n}\sum_{j=1}^n\big(\nabla f_{i,j}(\vx^{t+1}_i)-\nabla f_{i,j}(\vx^t_i)\big)}^2 \\
\leq &  \dfrac{3}{n}\sum_{j=1}^{n}\left(\frac{\xi^t_{i,j}}{bq}-1\right)^2\Norm{\nabla f_{i,j}(\vx^{t+1}_i) - \nabla f_{i,j}(\bx^{t+1})}^2 \\
& + \dfrac{3}{n}\sum_{j=1}^{n}\left(\frac{\xi^t_{i,j}}{bq}-1\right)^2\Norm{\nabla f_{i,j}(\vx^t_i) - \nabla f_{i,j}(\bx^t)}^2 \\
& + 3\Norm{\dfrac{1}{n}\sum_{j=1}^{n}\dfrac{\xi^t_{i,j}}{bq}\big(\nabla f_{i,j}(\bx^{t+1}) - \nabla f_{i,j}(\bx^t)\big) - \dfrac{1}{n}\sum_{j=1}^{n}\big(\nabla f_{i,j}(\bx^{t+1}) - \nabla f_{i,j}(\bx^t))}^2.
\end{split}
\end{align}
Following the derivation of equations (\ref{eq:U21})-(\ref{eq:U22}) in the proof of Lemma \ref{lem:global}, we can bound the first two terms on the right-hand side of equation (\ref{eq:V2}) as 
\begin{align}\label{eq:V21}
\small\begin{split}    
  &  \dfrac{3}{n}\sum_{j=1}^{n}\left(\frac{\xi^t_{i,j}}{bq}-1\right)^2\Norm{\nabla f_{i,j}(\vx^{t+1}_i) - \nabla f_{i,j}(\bx^{t+1})}^2 \\
\overset{(\ref{eq:xi-bq-1})}{\leq} &  \dfrac{3}{n}\sum_{j=1}^{n}\frac{1}{q^2}\Norm{\nabla f_{i,j}(\vx^{t+1}_i) - \nabla f_{i,j}(\bx^{t+1})}^2  \\
\,\overset{(\ref{eq:smooth-individual1})}{\leq}\, &   \dfrac{3}{n} \sum_{j=1}^n  \frac{mn\bar L^2}{q^2}\Norm{\vx^{t+1}_i - \bx^{t+1}}^2 =   \frac{3mn\bar L^2}{q^2} \Norm{\vx^{t+1}_i - \bx^{t+1}}^2,
\end{split}
\end{align}
and 
\begin{align}\label{eq:V22}
\begin{split}    
  \dfrac{3}{n}\sum_{j=1}^{n}\left(\frac{\xi^t_{i,j}}{bq}-1\right)^2\Norm{\nabla f_{i,j}(\vx^t_i) - \nabla f_{i,j}(\bx^t)}^2 
\leq  \frac{3mn\bar L^2}{q^2} \Norm{\vx^t_i - \bx^t}^2.
\end{split}
\end{align}
For the third term on the right-hand side of equation (\ref{eq:V2}), we have
\begin{align}\label{eq:V23}
\begin{split} 
& 3\Norm{\dfrac{1}{n}\sum_{j=1}^{n}\dfrac{\xi^t_{i,j}}{bq}\big(\nabla f_{i,j}(\vx^{t+1}_i) - \nabla f_{i,j}(\vx^t_i)\big) - \frac{1}{n}\sum_{j=1}^n\big(\nabla f_{i,j}(\vx^{t+1}_i)-\nabla f_{i,j}(\vx^t_i)\big)}^2 \\
= & 3\Norm{\dfrac{1}{n}\sum_{j=1}^{n}\left(\dfrac{\xi^t_{i,j}}{bq}-1\right)\big(\nabla f_{i,j}(\vx^{t+1}_i) - \nabla f_{i,j}(\vx^t_i)\big)}^2 \\
\leq & \dfrac{3}{n}\sum_{j=1}^{n}\left(\dfrac{\xi^t_{i,j}}{bq}-1\right)^2\Norm{\nabla f_{i,j}(\vx^{t+1}_i) - \nabla f_{i,j}(\vx^t_i)}^2 \\
\overset{(\ref{eq:xi-bq-1})}{\leq} & \dfrac{3}{n}\sum_{j=1}^{n}\frac{1}{q^2}\Norm{\nabla f_{i,j}(\vx^{t+1}_i) - \nabla f_{i,j}(\vx^t_i)}^2 \\
\overset{(\ref{eq:smooth-individual1})}{\leq} & \dfrac{3}{n}\sum_{j=1}^{n}\frac{mn\bL^2}{q^2}\Norm{\vx^{t+1}_i - \vx^t_i}^2 
= \dfrac{3mn\bL^2}{q^2}\Norm{\vx^{t+1}_i - \vx^t_i}^2 \\
\leq & \dfrac{9mn\bL^2}{q^2}\left(\Norm{\vx^{t+1}_i - \bx^{t+1}}^2 + \Norm{\bx^{t+1} - \bx^t}^2 + \Norm{\bx^t - \vx^t_i}^2\right),
\end{split}
\end{align}
where the first inequality is based on the fact $\|\frac{1}{n}\sum_{j=1}^n  a_j\|^2 \leq \frac{1}{n} \sum_{i=1}^n \Norm{a_j}^2$ for all~$a_1,\dots,a_n\in\BR^d$; 
the second inequality is based on equation (\ref{eq:xi-bq-1}); 
the third inequality is based on Proposition~\ref{prop:smooth}(a); the last inequality is based on Young's inequality.

Combining equations (\ref{eq:V2}), (\ref{eq:V21}), (\ref{eq:V22}), and (\ref{eq:V23}), we achieve
\begin{align}\label{eq:V2-final}
\begin{split} 
& \Norm{\dfrac{1}{n}\sum_{j=1}^{n}\dfrac{\xi^t_{i,j}}{bq}\big(\nabla f_{i,j}(\vx^{t+1}_i) - \nabla f_{i,j}(\vx^t_i)\big) - \frac{1}{n}\sum_{j=1}^n\big(\nabla f_{i,j}(\vx^{t+1}_i)-\nabla f_{i,j}(\vx^t_i)\big)}^2 \\
\leq &  \frac{3mn\bar L^2}{q^2} \Big(\Norm{\vx^{t+1}_i - \bx^{t+1}}^2 +  \Norm{\vx^t_i - \bx^t}^2 \\
& \qquad\quad\quad + 3\Norm{\vx^{t+1}_i - \bx^{t+1}}^2 + 3\Norm{\bx^{t+1} - \bx^t}^2 + 3\Norm{\bx^t - \vx^t_i}^2\Big) \\
= &  \frac{3mn\bar L^2}{q^2} \left(4\Norm{\vx^{t+1}_i - \bx^{t+1}}^2 +  4\Norm{\vx^t_i - \bx^t}^2  + 3\Norm{\bx^{t+1} - \bx^t}^2\right).
\end{split}
\end{align}
Combining equations (\ref{eq:V-initial}) and (\ref{eq:V2-final}), we have
\begin{align}\label{eq:V-final-component}
\begin{split}
& \BE\Norm{\vg^{t+1}_i-\nabla f_i(\vx^{t+1}_i)}^2 \\
\leq & (1-p)\BE\Norm{\vg^t_i - \nabla f_i(\vx^t_i)}^2 \\
& +\frac{3(1-p)mn\bar L^2}{q^2}\BE\left[4\Norm{\vx^{t+1}_i - \bx^{t+1}}^2 +  4\Norm{\vx^t_i - \bx^t}^2 + 3\Norm{\bx^{t+1} - \bx^t}^2\right]
\end{split}
\end{align}
Taking the average on equation (\ref{eq:V-final-component}) over $i=1,\dots,m$, we obtain
\begin{align*}
\begin{split}
& \BE[V^{t+1}] \\
\,=\,& \frac{1}{m}\sum_{i=1}^m\BE\Norm{\vg^{t+1}_i-\nabla f_i(\vx^{t+1}_i)}^2 \\
\overset{\!(\ref{eq:V-final-component})\!}{\leq} & \frac{1-p}{m}\sum_{i=1}^m\BE\Norm{\vg^t_i - \nabla f_i(\vx^t_i)}^2 \\
& +\frac{3(1-p)n\bar L^2}{q^2}\sum_{i=1}^m\BE\left[4\Norm{\vx^{t+1}_i - \bx^{t+1}}^2 +  4\Norm{\vx^t_i - \bx^t}^2 + 3\Norm{\bx^{t+1}_i - \bx^t_i}^2\right] \\
\overset{\!(\ref{eq:def-V})\!}{\leq} & (1-p)\BE[V^t] + \dfrac{3(1-p)n\bL^2}{q^2}\sum_{i=1}^m\BE\left[4\Norm{\vx^{t+1}_i - \bx^{t+1}}^2 +  4\Norm{\vx^t_i - \bx^t}^2 + 3\Norm{\bx^{t+1} - \bx^t}^2\right] \\
\,=\, & (1-p)\BE[V^t] + \dfrac{3(1-p)n\bL^2}{q^2}\BE\Big[4\!\Norm{\vx^{t+1}-\vone\bx^{t+1}}^2 \!+\! 4\BE\Norm{\mX^t\!-\!\vone\bx^t}^2 \!+\! 3m\BE\Norm{\bx^{t+1}\!-\!\bx^t}^2\!\Big]  \\
\,\leq\, & (1-p)\BE[V^t] + 3(1-p)m^2n^3\bL^2\BE\Big[8C^t + 4C^t + 3m\BE\Norm{\bx^{t+1}-\bx^t}^2\Big]  \\
\,=\, & (1-p)\BE[V^t] + 36(1-p)m^2n^3\bL^2\BE[C^t] + 9(1-p)m^3n^3\bL^2\BE\Norm{\bx^{t+1}-\bx^t}^2,
\end{split}
\end{align*}
where the last inequality is based on equation (\ref{recursive:x}) and the definition of $C^t$. Therefore, we finish the proof.
\end{proof}

\subsection{The Proofs of Theorem \ref{thm:main}}\label{sec:thm:main}

We prove Theorem \ref{thm:main} by considering the cases of $\bar L<\sqrt{mn}L$ and $\bar L\geq \sqrt{mn}L$, respectively.
Note that \DEAREST~(Algorithm \ref{alg:DEAREST}) only apply variance reduction in the first case, and the algorithm always iterates with the exact local gradient in second ones. 

\begin{proof}
\textbf{Part I:} We first consider the case of
$\bar L<\sqrt{mn}L$.
We scale the results of Lemmas \ref{lem:consensus}--\ref{lem:local} as 
\begin{align*}
\begin{split}    
& \frac{2\eta}{p}\BE[U^{t+1}] 
\leq  \frac{2(1-p)\eta}{p}\BE[U^t] + \frac{18(1-p)m^2n^3\bar L^2\eta}{p}\BE[C^t] + \frac{6(1-p)\bL^2\eta}{bp}\BE\Norm{\bx^{t+1}-\bx^t}^2  \\[0.15cm]
& \frac{\eta}{m^3n^3bp}\BE[V^{t+1}] 
\leq  \frac{(1-p)\eta}{m^3n^3bp}\BE[V^t] + \frac{36(1-p)\bL^2 \eta}{mbp} \BE[C^t] + \frac{9(1-p)\bL^2\eta}{bp}\BE\Norm{\bx^{t+1}-\bx^t}^2 
\end{split}
\end{align*}
and
\begin{align*}
\frac{132m^2n^3\bar L^2\eta}{p}\BE[C^{t+1}]
\leq & \frac{264\rho^2(27m^3n^3\bL^2\eta^2+2)m^2n^3\bar L^2\eta}{p} \BE[C^t] 
 + 528\rho^2m^3n^3\bar L^2\eta^3 \BE[V^t]  \\
& + \frac{2376\rho^2m^6n^6\bar L^4\eta^3}{p}\BE\Norm{\bx^{t+1}-\bx^t}^2.
\end{align*}
Combining with Lemma \ref{lem:decrease-f}, we have
\begin{align}\label{eq:recursion-Phi-0}
\small\begin{split}
  & \BE[\Phi^{t+1}] \\
\overset{(\ref{eq:Lyapunov})}{=} & \BE\left[f(\bx^{t+1}) - f^* + \frac{2\eta}{p}U^{t+1} + \frac{\eta}{m^3n^3bp} V^{t+1} + \frac{132m^2n^3\bar L^2\eta}{p} C^{t+1}\right] \\
\,\,\leq\,\, & \BE\Big[f(\bx^t) - f^* - \frac{\eta}{2}\Norm{\nabla f(\bx^t)}^2  + \eta U^t +  n \bL^2\eta   C^t - \left(\frac{1}{2\eta}-\frac{L}{2}\right)\Norm{\bx^{t+1}-\bx^t}^2\Big]   \\
& + \frac{2(1-p)\eta}{p}\BE[U^t] + \frac{18(1-p)m^2n^3\bar L^2\eta}{p}\BE[C^t] + \frac{6(1-p)\bL^2\eta}{bp}\BE\Norm{\bx^{t+1}-\bx^t}^2  \\
& + \frac{(1-p)\eta}{m^3n^3bp}\BE[V^t] + \frac{36(1-p)\bL^2 \eta}{mbp} \BE[C^t] + \frac{9(1-p)\bL^2\eta}{bp}\BE\Norm{\bx^{t+1}-\bx^t}^2 \\
& + \frac{264\rho^2(27m^3n^3\bL^2\eta^2+2)m^2n^3\bar L^2\eta}{p} \BE[C^t] + 528\rho^2m^3n^3\bar L^2\eta^3 \BE[V^t] \\
& + \frac{2376\rho^2m^6n^6\bar L^4\eta^3}{p}\BE\Norm{\bx^{t+1}-\bx^t}^2\\
\,\,=\,\, & \BE\Big[f(\bx^t) - f^* - \frac{\eta}{2}\Norm{\nabla f(\bx^t)}^2\Big]  + \left(\eta + \frac{2(1-p)\eta}{p}\right)\BE[U^t]   \\
& + \left(\frac{(1-p)\eta}{m^3n^3bp} + 528\rho^2m^3n^3\bar L^2\eta^3\right) \BE[V^t] \\
& +  \left(n \bL^2\eta   + \frac{18(1-p)m^2n^3\bar L^2\eta}{p} + \frac{36(1-p)\bL^2 \eta}{mbp} + \frac{264\rho^2(27m^3n^3\bL^2\eta^2+2)m^2n^3\bar L^2\eta}{p}\right) \BE[C^t] \\
& - \left(\frac{1}{2\eta}-\frac{L}{2}
- \frac{15(1-p)\bL^2\eta}{bp}
- \frac{2376\rho^2m^6n^6\bar L^4\eta^3}{p}\right)\BE\Norm{\bx^{t+1}-\bx^t}^2 \\
\,\,\leq\,\, & \BE\Big[f(\bx^t) - f^* - \frac{\eta}{2}\Norm{\nabla f(\bx^t)}^2  + \frac{2\eta}{p}U^t   
 + \frac{\eta}{m^3n^3bp} V^t 
 +  \left(\frac{132m^2n^3\bar L^2\eta}{p} - \frac{m^2n^3\bar L^2\eta}{p}\right) C^t \Big] \\
 \overset{(\ref{eq:Lyapunov})}{=} & \BE\Big[\Phi^t - \frac{\eta}{2}\Norm{\nabla f(\bx^t)}^2  
  - \frac{m^2n^3\bar L^2\eta}{p} C^t\Big],  
\end{split},
\end{align}
where the first inequality is based on Lemma \ref{lem:decrease-f}--\ref{lem:local} and the second inequality is based on parameter settings shown in the statement of Theorem \ref{thm:main} and
\begin{align*}
  K= \left\lceil\frac{2+\sqrt{2}}{2\sqrt{\gamma}}\ln\left(14\max\left\{\frac{33m^{6.5}n^{6.5}\bL^3}{4L^3},\,\frac{33(27m^3n^3\bL^2+128L^2)}{608L^2}\right\}\right)\right\rceil.
\end{align*}
Summing over equation (\ref{eq:recursion-Phi-0}) with $t=0,\dots,T-1$, we have
\begin{align*}
\begin{split}
  \BE[\Phi^t] 
 \leq \BE\left[\Phi^0 - \frac{\eta}{2}\sum_{t=0}^{T-1}\Norm{\nabla f(\bx^t)}^2  
  - \frac{m^2n^3\bar L^2\eta}{p} C^t\right].
\end{split}
\end{align*}
This implies
\begin{align}\label{eq:recursion-Phi-1}
\begin{split}
   \BE\left[\frac{1}{T}\sum_{t=0}^{T-1}\Norm{\nabla f(\bx^t)}^2\right] 
\leq  \BE\left[\frac{2(\Phi^0 - \Phi^t)}{\eta T}  
  - \frac{2m^2n^3\bar L^2}{p T} \sum_{t=0}^{T-1}C^t\right].
\end{split}
\end{align}

For the first term on the right-hand side of equation (\ref{eq:recursion-Phi-1}), we have
\begin{align}\label{ieq:Phi0-T}
\begin{split}
& \Phi^0 - \Phi^t \\
\overset{(\ref{eq:Lyapunov})}{=} & f(\bx^0) - f^*  + \frac{2\eta}{p}U_{0} + \frac{\eta}{m^3n^3bp} V_{0} + \frac{132m^2n^3\bar L^2\eta}{p}C_{0} \\
 & - \left(f(\bx^t) - f^* + \frac{2\eta}{p}U^{t} + \frac{\eta}{m^3n^3bp} V^{t} + \frac{132m^2n^3\bar L^2\eta}{p}C^{t}\right)\\
\,\overset{(\ref{eq:lower-value})}{\leq}\, & f(\bx^0)  + \frac{2\eta}{p}U_{0} + \frac{\eta}{m^3n^3bp} V_{0} + \frac{132m^2n^3\bar L^2\eta}{p}C_{0}
 - f^* \\
\,\,=\,\, & f(\bx^0) - f^* + \frac{132m^2n^3\bar L^2\eta}{p} C^0 \\
\overset{(\ref{eq:def-C})}{\leq} & f(\bx^0) - f^* + \frac{132m^2n^3\bar L^2\eta^3}{p} \Norm{\mS^0 - \vone\bs^0}^2,
\end{split}
\end{align}
where the first inequality is based on the fact $U^t, V^t, C^t\geq 0$ and Assumption \ref{asm:lower-bounded}; 
the second equality is based on the fact $U^0=V^0=0$;
the second inequality is based on the definition of $C^t$.

For the second term on the right-hand side of equation (\ref{eq:recursion-Phi-1}), we have
\begin{align}\label{ieq:Phi0-C1}
\frac{2m^2n^3\bar L^2}{p T} \sum_{t=0}^{T-1}C^t
\overset{(\ref{eq:def-C})} \geq  \frac{2m^2n^3\bar L^2}{p T} \sum_{t=0}^{T-1}\Norm{\mX^t-\vone\bx^t}^2 \geq \frac{2L^2}{T} \sum_{t=0}^{T-1}\Norm{\mX^t-\vone\bx^t}^2,
\end{align}
where the first inequality is based on the definition of $C^t$; the second inequality is base on the setting $p\leq 1$ and the assumption $\bL \geq L$.

Combining equations (\ref{eq:recursion-Phi-1}), (\ref{ieq:Phi0-T}), and (\ref{ieq:Phi0-C1}), we have
\begin{align}\label{ieq:grad-bx}
\begin{split}    
 & \BE\left[\frac{1}{T}\sum_{t=0}^{T-1}\Norm{\nabla f(\bx^t)}^2\right]  \\
\overset{(\ref{eq:def-C})}{\leq} & \frac{2}{\eta T}\left(f(\bx^0) - f^* + \frac{132m^2n^3\bar L^2\eta^3}{p} \Norm{\mS^0 - \vone\bs^0}^2\right) - \frac{2L^2}{T} \sum_{t=0}^{T-1}\Norm{\mX^t-\vone\bx^t}.
\end{split}
\end{align}
Therefore, the output $\vx^{\rm out}_i$ satisfies 
\begin{align}\label{eq:grad-x-out}
\begin{split}
& \BE\Norm{\nabla f(\vx^{\rm out}_i)}^2 \\
\,=\, & \frac{1}{T}\sum_{t=0}^{T-1}\BE\Norm{\nabla f(\vx^t_i)}^2 \\
\,\leq\, & \frac{2}{T}\sum_{t=0}^{T-1} \BE\left[\Norm{\nabla f(\bx^t)}^2 + \Norm{\nabla f(\vx^t_i)-\nabla f(\bx^t)}^2\right] \\
\overset{(\ref{eq:smooth-global})}{\leq} & \frac{2}{T}\sum_{t=0}^{T-1} \BE\Norm{\nabla f(\bx^t)}^2  + \frac{2L^2}{T}\sum_{t=0}^{T-1}\BE\Norm{\vx^t_i-\bx^t}^2 \\
\,\leq\, & \frac{2}{T}\sum_{t=0}^{T-1}\BE\Norm{\nabla f(\bx^t)}^2  + \frac{2L^2}{T}\sum_{t=0}^{T-1}\BE\Norm{\mX^t-\vone\bx^t}^2 \\
\overset{(\ref{ieq:grad-bx})}{\leq} & \frac{4}{\eta T}\left(f(\bx^0) - f^* + \frac{132m^2n^3\bar L^2\eta^3}{p} \Norm{\mS^0 - \vone\bs^0}^2\right) - \frac{4L^2}{T} \sum_{t=0}^{T-1}\Norm{\mX^t-\vone\bx^t} \\
& + \frac{2L^2}{T}\sum_{t=0}^{T-1}\BE\Norm{\mX^t-\vone\bx^t}^2 \\
\,\leq\, & \frac{32}{33}\epsilon^2 + \frac{1}{33}\epsilon^2 = \epsilon^2,
\end{split}
\end{align}
for all $i\in[m]$ where the first inequality is based on Young's inequality; 
the second inequality is based on Assumption \ref{asm:smooth-global};
the third inequality is based on equation (\ref{ieq:grad-bx}); 
the last inequality is based on the parameter settings shown in the statement of Theorem~\ref{thm:main} and taking
\begin{align*}\small
\begin{split}    
  \hK=\left\lceil\frac{2+\sqrt{2}}{2\sqrt{\gamma}}  \ln \left(1+\frac{1089m^{2.5}n^{3.5}\bar L}{4L(33L(f(\bx^0)-f^*)+\epsilon^2)}\sum_{i=1}^m\Norm{\nabla f_i(\bx^0)-\nabla f(\bx^0)}^2\right)\right\rceil. 
\end{split}
\end{align*}
Finally, we apply Jensen's inequality to obtain
\begin{align}\label{eq:main-final}
\BE\Norm{\nabla f(\vx^{\rm out}_i)}\leq \sqrt{\BE\|\nabla f(\vx^{\rm out}_i)\|^2\,}\overset{(\ref{eq:grad-x-out})}{\leq}\epsilon.
\end{align}

\textbf{Part II:} We then consider the case of $\bar L\geq\sqrt{mn}L$.
Note that the setting $p=1$ leads to  Algorithm \ref{alg:DEAREST} always holds 
\begin{align*}
    \vg^t_i = \nabla f_i(\vx^t_i),
\end{align*}
which implies $U^t=V^t=0$ for all $t$ and the Lyapunov can be simplified to 
\begin{align}\label{eq:Lyapunov-2}
\Phi^t \triangleq & f(\bx^{t}) - f^* + 132m^2n^3\bar L^2\eta C^{t}.
\end{align}
Similar to the derivation of equation (\ref{eq:recursion-Phi-0}), we have
\begin{align}\label{eq:proof-main-last}\small
\begin{split}
  & \BE[\Phi^{t+1}] \\
\overset{(\ref{eq:Lyapunov-2})}{=} & \BE\left[f(\bx^{t+1}) - f^* + 132m^2n^3\bar L^2\eta C^{t+1}\right] \\
\,\,\leq\,\, & \BE\Big[f(\bx^t) - f^* - \frac{\eta}{2}\Norm{\nabla f(\bx^t)}^2  +  n \bL^2\eta   C^t - \left(\frac{1}{2\eta}-\frac{L}{2}\right)\Norm{\bx^{t+1}-\bx^t}^2   \\
& + 264\rho^2(27m^3n^3\bL^2\eta^2+2)m^2n^3\bar L^2\eta\BE[C^t] + 2376\rho^2m^6n^6\bar L^4\eta^3\BE\Norm{\bx^{t+1}-\bx^t}^2 \\
\,\,=\,\, & \BE\Big[f(\bx^t) - f^* - \frac{\eta}{2}\Norm{\nabla f(\bx^t)}^2\Big]  \\
& +  \left(n \bL^2\eta  + 264\rho^2(27m^3n^3\bL^2\eta^2+2)m^2n^3\bar L^2\eta\right) \BE[C^t] \\
& - \left(\frac{1}{2\eta} - \frac{L}{2}
- 2376\rho^2m^6n^6\bar L^4\eta^3\right)\BE\Norm{\bx^{t+1}-\bx^t}^2 \\
\,\,\leq\,\, & \BE\Big[f(\bx^t) - f^* - \frac{\eta}{2}\Norm{\nabla f(\bx^t)}^2  +  \left(132m^2n^3\bar L^2\eta - m^2n^3\bar L^2\eta\right) C^t \Big] \\
 \overset{(\ref{eq:Lyapunov-2})}{=} & \BE\Big[\Phi^t - \frac{\eta}{2}\Norm{\nabla f(\bx^t)}^2  
  - m^2n^3\bar L^2\eta C^t\Big], 
\end{split}
\end{align}
where the last inequality is based on the parameter settings shown in the statement of Theorem \ref{thm:main} and taking
\begin{align*}    
K= \left\lceil\frac{2+\sqrt{2}}{2\sqrt{\gamma}}\ln\left(14\max\left\{\frac{33(27m^3n^3\bL^2+128L^2)}{1040L^2},\,\frac{297m^{6.5}n^{6.5}\bL^3}{112L^3}\right\}\right)\right\rceil.
\end{align*}
We then follow the derivation of equations (\ref{eq:recursion-Phi-1})--(\ref{eq:main-final}) to achieve
\begin{align*}
\BE\Norm{\nabla f(\vx^{\rm out}_i)} \leq \epsilon.
\end{align*}
for all $i\in[m]$.
\end{proof}

\begin{remark}
It is worth noting that Theorem \ref{thm:main} guarantees that each agent can obtain an~$\epsilon$-stationary point $\vx_i^{\rm out}$ in expectation.
We achieve this result by retaining the negative term $-m^2n^3\bar L^2\eta p^{-1}C^t$ in the last line of equation (\ref{eq:proof-main-last}), which is helpful to cancel the terms related to consensus error in the output $\{\vx_i^{\rm out}\}_{i=1}^m$.
In contrast, the previous work only guarantee one of the agents can obtain an $\epsilon$-stationary point \cite{xin2022fast,li2022destress}, or the average of vectors from all agents is an $\epsilon$-stationary point \cite{zhan2022efficient,metelev2024decentralized,sun2020improving,yuan2022revisiting,lu2021optimal}.
\end{remark}

\subsection{The Proofs of Corollary \ref{cor:complexity-general}}\label{sec:cor:complexity-general}

We prove Corollary \ref{cor:complexity-general} by considering the communication rounds, the local incremental first-oracle calls, and the computation rounds, respectively.

\subsubsection{The Upper Bound on Communication Rounds}

We upper bound the communication rounds as follows
\begin{align*}
    \hat K + KT = \fO\left(\frac{\ln(mn\bL/(L\epsilon))}{\sqrt{\gamma}}+\frac{\ln(mn\bL/L)}{\sqrt{\gamma}}\cdot \frac{L}{\epsilon^2}\right) = \tilde\fO\left(\frac{L}{\sqrt{\gamma}\epsilon^2}\right),
\end{align*}
where we use the settings of $\hK$, $K$, and $T$ in Theorem \ref{alg:DEAREST}.

\subsubsection{The Upper Bound on LIFO Calls}

In the case of $\bL\leq\sqrt{mn}L$, we upper bound the number of LIFO calls as follows
\begin{align*}
 &   mn + (pn + (1-p)b)T \\
\leq & mn + (pmn + b)T \\
= & \fO\left(mn + \left(\frac{\bL}{\sqrt{mn}L}\cdot mn + \left\lceil\frac{\bL\sqrt{mn}}{L}\,\right\rceil \right)\frac{L}{\epsilon^2}\right) \\
= & \fO\left(mn + \frac{\sqrt{mn}\bL}{\epsilon^2}\right),
\end{align*}
where we use the settings of $p$, $n$, and $T$ in Theorem \ref{alg:DEAREST}.
In the case of $\bL>\sqrt{mn}L$, the number of LIFO calls can be upper bounded by $mnT = \fO(mn + mnL\epsilon^{-2})$.

Combining above two cases, the overall LIFO complexity of Algorithm \ref{alg:DEAREST} under the settings of Theorem~\ref{thm:main} is no more than
$\fO\big(mn+ \min\{mnL,\,\sqrt{mn}\bL\}\epsilon^{-2}\big)$.

\subsubsection{The Upper Bound on Computation Rounds}\label{sec:computation-rounds}

In the case of $\bL>\sqrt{mn}L$, the agents always perform the full gradient at each iteration, which means the number of computation round is no more than
\begin{align}\label{eq:computation-round-simple}
    n + Tn =  \fO\left(n + \frac{nL}{\epsilon^2}\right).
\end{align}

The remains in this subsection considers the case of $\bL\leq\sqrt{mn}L$.
The main issue is the case of $\zeta^t=0$, leading to
\begin{align}\label{eq:com-rounds-update}
    \vg^{t+1}_i = \vg^t_i + \dfrac{1}{n}\sum_{j=1}^{n}\frac{\xi^t_{i,j}}{bq}\big(\nabla f_{i,j}(\vx^{t+1}_i) - \nabla f_{i,j}(\vx^t_i)\big). 
\end{align}
The setting
\begin{align*} 
[\xi^t_{1,1},\dots,\xi^t_{m,n}]^\top\sim {\rm Multinomial}\left(b,q\vone\right)
\end{align*}
with $q=1/(mn)$ means the update (\ref{eq:com-rounds-update}) can be regarded as the procedure of sampling~$b$ index pairs $(i_1^t,j_1^t),\dots,(i_b^t,j_b^t)$ from the set
\begin{align*}
    \{(i,j): i\in[m], j\in[n]\}
\end{align*}
independently and uniformly with replacement and performing the computation
\begin{align*}
    \vg^{t+1}_i = \vg^t_i + \sum_{k=1}^{b}\frac{\tau_{i}^t(k)}{bq}\big(\nabla f_{i_k^t,j_k^t}(\vx^{t+1}_{i_k^t}) - \nabla f_{i_k^t,j_k^t}(\vx^t_{i_k^t})\big) 
\end{align*}
on each node, where we define
\begin{align}\label{eq:tauitk}
 \tau_{i}^t(k) \triangleq \BI(i_k=i)
\end{align}
for $i\in[m]$ and $k\in [b]$ to present whether the $k$-th index pair $(i_k^t,j_k^t)$ in iteration $t$ corresponds to the component function on node $i$. 

Since there are $mn$ individual functions in total and each node has $n$ individual functions, the uniform distribution of $(i_k^t, j_k^t)$ and their independence imply
\begin{align}\label{eq:tau-random}
\tau_{i}^t(k) \sim \text{Bernoulli}(nq)
\end{align}
and $\tau_i^t(1),\dots,\tau_i^t(b)$ are mutually independent.
Therefore, the procedure of Algorithm~\ref{alg:DEAREST} means the number of LIFO calls on node $i$ at the $t$-th iteration  is no more than 
\begin{align*}
    2\sum_{k=1}^b \tau_i^t(k).
\end{align*}
We denote
\begin{align}\label{eq:Yit}
    Y_i^t \triangleq \sum_{k=1}^b \tau_i^t(k),
\end{align}
which holds that $Y_i^t \sim \text{Binomial}(b, nq)$ because of the distribution (\ref{eq:tau-random}). 

To analyze the expected computation rounds in the $t$-th iteration, we only need to upper bound the quantity
\begin{align}\label{eq:sum_Y}
\BE\left[2\max_{i\in[m]} Y_i^t\right].    
\end{align}
In a recent work, \citet{liu2024decentralized} analyzed the computation rounds of Katyusha-type methods \cite{allen2017katyusha,qian2021svrg,kovalev2020don,song2024optimal} in decentralized convex optimization by establishing the upper bound of the quantity in the form of (\ref{eq:sum_Y}) with the mutually independent binomial variables $Y_1^t,\dots,Y_m^t$.
Although the random variables $Y_1^t,\dots,Y_m^t$ are not mutually independent in our setting, we can also follow the analysis of \citet{liu2024decentralized} to bound the quantity (\ref{eq:sum_Y}).
For the completeness, we present the details as follows.

\begin{lemma}[Theorem 4.1 of \citet{motwani1995randomized}]\label{lem:chernoff_bound}
    Suppose that the random variables $X_1,...,X_b$ are independent and each  $X_k$ is distributed to ${\rm Bernoulli}(p_k)$ for some $p_k\in[0,1]$. 
    We let $Z = \sum_{k=1}^b X_k$ and $\nu=\BE[Z]$. 
    Then for any $\delta > 0$, it holds 
    \begin{align}\label{eq:lem_chernoff_bound}
        \BP\left( Z\ge (1+\delta)\nu \right) \le \left( \frac{\exp(\delta)}{(1+\delta)^{1+\delta}} \right)^\nu.
    \end{align}
\end{lemma}

Based on Lemma \ref{lem:chernoff_bound}, we further achieve the following upper bound for the sum of Bernoulli variables with high probability. 

\begin{lemma}\label{lem:binomial_bound}
Under the settings and notations of Lemma \ref{lem:chernoff_bound},
For any constant $a \ge \exp(2)$, we have
    \begin{align}\label{eq:lem_binomial_bound}
        \BP\left( Z \ge 2{\rm e} \max\{ \BE[Z], (\ln a)^2 \}  \right) \le \frac{1}{a^2}.
    \end{align}
\end{lemma}
\begin{proof}
We denote $\nu=\BE[Z]=\sum_{k=1}^b p_k$. 
We consider the upper bounds for different cases of $\nu$ as follows:
\begin{itemize}
 \item[(a)] 
 In the case of $\nu \ge \ln a$, we apply Lemma~\ref{lem:chernoff_bound} with $\delta=2{\rm e}-1$ to achieve
 \begin{align*}
     & \BP(Z \ge 2{\rm e}\nu) \overset{(\ref{eq:lem_chernoff_bound})}{\le}  \left( \frac{\exp(2{\rm e}-1)}{(2{\rm e})^{2{\rm e}}} \right)^\nu 
     = \left( \frac{\exp(2{\rm e}-1)}{2^{2{\rm e}}\exp(2{\rm e})} \right)^\nu  \\
     = & \left( \frac{1}{2^{2{\rm e}}\rm e} \right)^\nu
     \leq \left( \frac{1}{2^{2{\rm e}}\rm e} \right)^{\ln a}
     \leq \frac{1}{a^2},
 \end{align*}
 where the last step can be verified by taking logarithm on both sides and the fact
 \begin{align*}
     2 \leq \ln\left(2^{2{\rm e}}\rm e\right) = 1 + 2{\rm e}\ln2.
 \end{align*}

\item[(b)]  If $1 \le \nu \le \ln a$, we have
\begin{align*}
    \BP\left( Z\ge 2{\rm e}(\ln a)^2 \right) 
    \le \BP\left( Z\ge (2{\rm e}\ln a)\nu \right).  
\end{align*}
According to Lemma~\ref{lem:chernoff_bound} with $\delta = 2{\rm e}\ln a-1$, we have
\begin{align*}
&    \BP\left( Z\ge (2{\rm e}\ln a)\nu \right) 
\overset{(\ref{eq:lem_chernoff_bound})}{\le} \left( \frac{\exp(2{\rm e}\ln a-1)}{(2{\rm e}\ln a)^{2{\rm e}\ln a}} \right)^\nu \\
= & \left( \frac{\exp(2{\rm e}\ln a-1)}{(2\ln a)^{2{\rm e}\ln a}\exp(2{\rm e}\ln a)} \right)^\nu  \\
= & \left( \frac{1}{(2\ln a)^{2{\rm e}\ln a}{\rm e}} \right)^\nu 
\leq \left( \frac{1}{4^{2{\rm e} \ln a}{\rm e}} \right)^\nu
\leq \frac{1}{a^2},
\end{align*}
where the second inequality is because $a \ge {\rm e}^2$ and the last inequality holds as~$\nu \ge 1$.
\item[(c)] If $\nu < 1$, we let $X_1'\sim{\rm Bernoulli}(p_1')$ with $p_1'=1-\nu + \BE[X_1]$ which is independent with $X_1,\dots,X_b$, and  denote $Z' = \sum_{j=2}^b X_j + X_1'$. It is clear that for any $c\geq 0$, it holds 
\begin{align}\label{eq:Z-Z'}
    \BP(Z \ge c) \le \BP(Z' \ge c).
\end{align}
We also can verify that
\begin{align}\label{eq:Z'-1}
\BE[Z']
= \BE\left[\sum_{j=2}^b X_j + X_1'\right]
= \sum_{j=2}^b p_j + 1 - \nu + p_1 = 1.
\end{align}
Therefore, applying Lemma~\ref{lem:chernoff_bound} on random variable $Z'$ with $\delta = 2{\rm e} \ln a -1$ leads to that
\begin{align*}
& \BP\left( Z\ge (2{\rm e}\ln a) \right) \\
\,\,=\,\, & \BP\left( Z\ge (2{\rm e}\ln a) \cdot \BE\left[Z'\right] \right) \\
\overset{(\ref{eq:Z-Z'})}{\le} & \BP\left( Z'\ge (2{\rm e}\ln a) \cdot \BE\left[Z'\right] \right)   \\
\overset{(\ref{eq:lem_chernoff_bound})}{\le}& \left( \frac{\exp(2{\rm e}\ln a-1)}{(2{\rm e}\ln a)^{2{\rm e}\ln a}} \right)^{\BE[Z']} \\
\overset{(\ref{eq:Z'-1})}{=}&  \frac{\exp(2{\rm e}\ln a-1)}{(2{\rm e}\ln a)^{2{\rm e}\ln a}} \\
\,\,=\,\, & \frac{\exp(2{\rm e}\ln a-1)}{(2\ln a)^{2{\rm e}\ln a}\exp(2{\rm e}\ln a)}   \\
\,\,=\,\, &  \frac{1}{(2\ln a)^{2{\rm e}\ln a}{\rm e}}  \\
\leq & \frac{1}{4^{2{\rm e} \ln a}{\rm e}} 
\leq \frac{1}{a^2},
\end{align*}
where the second last inequality is because $a \ge {\rm e}^2$.
\end{itemize}
Combining all three cases, we obtain 
\begin{align*}
    \BP\left( Z \ge 2{\rm e} \max\{ \BE[Z], (\ln a)^2 \}  \right) 
    = \BP\left( Z \ge \max\{ 2{\rm e} \nu, 2{\rm e} (\ln a)^2 \}  \right) \le \frac{1}{a^2}
\end{align*}
for all $a \ge {\rm e}^2$, which finishes the proof.
\end{proof}

Recall that it always holds $Y_i^t \le b$ for all $i\in[m]$, then we can apply Lemma \ref{lem:binomial_bound} to upper bound the expectation of $\max_{i\in [m]} Y_i^t$.

\begin{lemma}\label{lem:computation_per_iter}
    Following the definitions in equations (\ref{eq:tauitk}) and (\ref{eq:Yit}), we have
    \begin{align}\label{eq:lem_computation_per_iter}
        \BE\left[ \max_{i\in[m]} Y_i^t \right] \le 4{\rm e}\max\left\{bnq, (2+\ln mb)^2\right\} + 2.
    \end{align}
\end{lemma}
\begin{proof}
Notice that the definition of $Y_i^t$ in equation (\ref{eq:Yit}) means
it is a summation of Bernoulli random variables.
Thus, we can apply Lemma \ref{lem:binomial_bound} with $a =  mb\exp(2)$ and~$Z=Y_i^t$ for all $i\in[m]$ to achieve
\begin{align}\label{eq:prob-Y}
\begin{split}
        & \BP\left( Y_i^t \ge 2{\rm e} \max\left\{ \BE[Y_i^t], (\ln a)^2 \right\} \right) \\
        = &
        \BP\left( Y_i^t \ge 2{\rm e} \max\left\{ bnq, (2+\ln mb)^2 \right\} \right)  
        \overset{(\ref{eq:lem_binomial_bound})}{\le} \frac{1}{m^2b^2\exp(4)},
\end{split}        
\end{align}
where the first step is based on the fact $\BE[Y_i^t]=bnq$.
    
Then we apply Boole's inequality to achieve
    \begin{align}\label{eq:proof_lem_computation_per_iter}
    \begin{split}
        & \BP\left( \max_{i\in [m]} Y_i^t \ge 2{\rm e} \max\left\{bnq, (2+\ln mb)^2\right\}\right) \\
        \,\,=\,\, & \BP\left(\text{exists}~i\in[m]~\text{such that}~\space Y_i^t \ge 2{\rm e} \max\left\{bnq, (2+\ln mb)^2\right\}\right) \\
        \,\,\le\,\, & \sum_{i=1}^m \BP\left(Y_i^t \ge 2{\rm e} \max\left\{bnq, (2+\ln mb)^2\right\}\right) \\
        \overset{(\ref{eq:prob-Y})}{\le} & \sum_{i=1}^m \frac{1}{m^2b^2\exp(4)}
        = \frac{1}{mb^2\exp(4)}.
    \end{split}
    \end{align}
    
    Since we have $Y_i^t \le b$, the conditional expectation formula implies
    \begin{align*}
        & \BE\left[ \max_{i\in[m]} Y_i^t \right] \\
        \,\,=\,\, & \BP\left( \max_{i\in[m]} Y_i^t <  2{\rm e}\max\left\{bnq, (2+\ln mb)^2\right\} \right) \\
        & \quad \cdot \BE\left[ \max_{i\in[m]} Y_i^t \,\Big{|}\, \max_{i\in[m]} Y_i^t <  2{\rm e}\max\left\{bnq, (2+\ln mb)^2\right\} \right] \\
        &+
        \BP\left( \max_{i\in[m]} Y_i^t \ge  2{\rm e}\max\left\{bnq, (2+\ln mb)^2\right\} \right) \\
        & \quad\cdot \BE\left[ \max_{i\in[m]} Y_i^t \,\Big{|}\, \max_{i\in[m]} Y_i^t \ge  2{\rm e}\max\left\{bnq, (2+\ln mb)^2\right\} \right] \\
        \overset{(\ref{eq:proof_lem_computation_per_iter})}{\le}& \left( 1 - \frac{1}{mb^2\exp(4)} \right) \cdot 2{\rm e}\max\left\{bnq, (2+\ln mb)^2\right\} + \frac{1}{mb^2\exp(4)} \cdot b \\
        \,\,\leq\,\, & 2{\rm e}\max\left\{bnq, (2+\ln mb)^2\right\} + 1
    \end{align*}
    Hence, we have
    \begin{align*}
        2\BE\left[ \max_{i\in[m]} Y_i^t \right]\leq 4{\rm e}\max\left\{bnq, (2+\ln mb)^2\right\} + 2 = \fO\left( bnq + (\ln mb)^2 \right) ,
    \end{align*}
    which concludes the proof.
\end{proof}

In the case of $\zeta^t=0$, Lemma \ref{lem:computation_per_iter} means the expected number of computation rounds is no more than
    \begin{align*}
        2\BE\left[ \max_{i\in[m]} Y_i^t \right]\leq 4{\rm e}\max\left\{bnq, (2+\ln mb)^2\right\} + 2 = \fO\left( bnq + (\ln mb)^2 \right)
\end{align*}
In the case of $\zeta^t=1$, it is obviously that the number of computation rounds is~$n$.
Therefore, the overall expected number of computation rounds for Algorithm \ref{alg:DEAREST} in the case of $\bL\leq\sqrt{mn}L$ is no more than
\begin{align*}
& n + \left((1-p)\BE\left[2\sum_{t=1}^T \max_{i\in[m]} Y_i^t \right] + pn\right)T \\
\leq & \fO\left(n + (pn+(1-p)(bnq + (\ln mb)^2))T\right) \\
\leq & \fO\left(n + (pn + bnq + (\ln mb)^2)T\right) \\
\leq & \fO\left(n + \left(\frac{n\bL}{\sqrt{mn}L} + \frac{\bL\sqrt{mn}\,n}{mnL} + \left(\ln \frac{m\bL\sqrt{mn}}{L}\right)^2\right)L\epsilon^{-2}\right) \\
= & \tilde\fO\left(n + \left(\sqrt{\frac{n}{m}}\,\bL + L\right)\epsilon^{-2}\right), 
\end{align*}
where we use the parameter settings of $b$, $p$, and $T$ in Theorem \ref{thm:main}.
Combining the upper bound (\ref{eq:computation-round-simple}) in the case of $\bL\leq\sqrt{mn}L$, the expected number of computation rounds of our method can be upper bounded by
\begin{align*}
 & \min\left\{\fO\left(n + nL\epsilon^{-2}\right),\, \tilde\fO\left(n + \left(\sqrt{\frac{n}{m}}\,\bL + L\right)\epsilon^{-2}\right)\right\} \\
=& \tilde\fO\left( n+\left(L+\min\left\{nL, \sqrt{\frac{n}{m}}\,  \bar L\right\}\right)\epsilon^{-2}\right).
\end{align*}
Hence, we finish the proof of Corollary \ref{cor:complexity-general}. 

\section{The Lower Bounds for General Nonconvex Case}

This section provide the lower bounds for the communication rounds, the LIFO calls, and the computation rounds for finding the approximate first-order stationary point in   decentralized smooth nonconvex finite-sum optimization problem by the algorithm class shown in Definition \ref{dfn:LIFO}. 
Without the loss of generality, we always assume the  algorithm iterates with initial point $\bx^{0}=\vzero$ in this section. 
Otherwise, we can take functions $\{f_{i,j}(\vx-\vx^{0})\}_{i,j=1}^{m,n}$ into considerations.
Compared with most of existing lower bounds only consider one of smoothness parameters \cite{carmon2021lower,carmon2020lower,yuan2022revisiting,lu2021optimal,arjevani2023lower,hendrikx2021optimal,agarwal2015lower,woodworth2016tight}, we simultaneously address the global smoothness parameter~$L$ and the mean-squared smoothness parameter $\bL$ in each of our lower bounds.

For our later analysis, we introduce the nonconvex function~$f^{\rm nc}_T:\R^T\to\R$ provided by \citet{carmon2020lower} as follows
\begin{align}\label{eq:function-nc}
    f^{\rm nc}_T(\vx) := -\Psi(1)\Phi(x_1)+\sum_{k=2}^T \left(\Psi(-x_{k-1})\Phi(-x_k)-\Psi(x_{k-1})\Phi(x_k)\right),
\end{align}
where $\vx=[x_1,\dots,x_T]^\top$ and the component functions are defined as
\begin{align}\label{eq:lower-bound-phi-pi}
\begin{split}
 &   \Psi(x) \triangleq 
		\begin{cases}
			0, & x\leq 1/2, \\[0.2cm]
			\exp{\left(1-\dfrac{1}{(2x-1)^2}\right)}, & x>1/2,
		\end{cases} \\
& \text{and} \qquad 
	\Phi(x) := \sqrt{{\rm e}}\int_{-\infty}^x \exp\left(-\frac{1}{2} t^2\right)\,{\rm d} t.
\end{split}
\end{align}
We denote
\begin{align*}
    \prog{(\vx)}\triangleq \max\{i \geq 0 : |x_i|>0\},
\end{align*}
where $\vx=[x_1,\dots,x_T]^\top$ and we let $x_0=1$.
The functions $\Phi(\cdot)$ and $\Psi(\cdot)$ defined equation (\ref{eq:lower-bound-phi-pi}) have the following properties. \vskip0.1cm

\begin{lemma}[{\citep[Observation 2]{arjevani2023lower}}]\label{lem:properties_phi_psi} The function value and (second-order) derivatives of $\Phi$ and $\Psi$ defined in equation (\ref{eq:lower-bound-phi-pi}) satisfy 
\begin{align*}
    & ~~0\leq\Psi\leq{\rm e},\qquad 0\leq\Psi'\leq\sqrt{54/{\rm e}}, \qquad\lvert\Psi''\rvert\leq32.5, \\
    & 0\leq\Phi\leq\sqrt{2\pi {\rm e}}, \qquad
    0\leq\Phi'\leq\sqrt{\rm e}, \qquad \text{and} \qquad\lvert\Phi''\rvert\leq1.
\end{align*}    
\end{lemma}

We can decompose the function $f_T^{\rm nc}$ defined in equation (\ref{eq:function-nc}) as
\begin{align}\label{eq:nc1nc2-decompose}
    f_T^{\rm nc}(\vx) = f_T^{\rm nc,1}(\vx) + f_T^{\rm nc,2}(\vx),
\end{align}
where
\begin{equation}\label{eq:nc1nc2}
\begin{aligned}
    & f_T^{\rm nc,1}(\vx)=-\Psi(1)\Phi(x_1)+\sum\limits_{k~\text{is odd};\, 2\leq k\leq T} (\Psi(-x_{k-1})\Phi(-x_k)-\Psi(x_{k-1})\Phi(x_k)),\\
    & \text{and} \qquad f_T^{\rm nc,2}(\vx)=\sum\limits_{k~\text{is even};\, 2\leq k\leq T} (\Psi(-x_{k-1})\Phi(-x_k)-\Psi(x_{k-1})\Phi(x_k)).
\end{aligned}
\end{equation}
Lemma \ref{lem:properties_phi_psi} implies the following properties for functions $f_T^{\rm nc,1}$,  $f_T^{\rm nc,2}$, and $f_T^{\rm nc}$.

\begin{lemma}\label{lem:function-nc-decompose}
The functions $f_T^{\rm nc,1}$ and $f_T^{\rm nc,2}$ defined in equation (\ref{eq:nc1nc2}) is $l_0$-smooth with~$l_0=152$.
\end{lemma}
\begin{proof} 
    For all $\vx\in\BR^T$, we have
    \begin{align*}
        & \|\nabla^2 f_T^{{\rm nc},1}(\vx)\|_2 
        \leq \sqrt{\|\nabla^2 f_T^{{\rm nc},1}(\vx)\|_1 \|\nabla^2 f_T^{{\rm nc},1}(\vx)\|_{\infty}} = \|\nabla^2 f_T^{{\rm nc},1}(\vx)\|_{\infty} \\
        & = \max_{i\in[T]} \sum_{j=1}^T |\nabla_{i,j}^2 f_T^{{\rm nc},1}(\vx)| = \max_{i\in[T]} \sum_{j\in[T], |i-j|\leq 1}|\nabla_{i,j}^2 f_T^{{\rm nc},1}(\vx)| \\
        &\leq\max_{i\in[T]}|\nabla_{i,i}^2 f_T^{{\rm nc},1}(\vx)|
	+ \max_{i\in[T-1]}|\nabla_{i,i+1}^2 f_T^{{\rm nc},1}(\vx)| + 
	\max_{i\in\{2,\dots,T\}} |\nabla_{i-1,i}^2 f_T^{{\rm nc},1}(\vx)|\\
        &\leq\sup_{z\in\R} |\Phi''(z)|  \sup_{z\in\R} |\Psi(z)| + 
	\sup_{z\in\R} |\Phi(z)| \sup_{z\in\R} |\Psi''(z)| + 2 \sup_{z\in\R} 
	|\Phi'(z)| \sup_{z\in\R} |\Psi'(z)|\\
        &\leq152,
    \end{align*}
    where the first inequality is based on H\"older's inequality; 
    the first equality is due to the Hessian is symmetric;
    the second equality is based on the definition of the infinity norm;
    the third equality is due to the definition of $f_T^{{\rm nc},1}$ in equation (\ref{eq:nc1nc2});
    the second inequality is based on the chain rule;
    the last step is based on Lemma \ref{lem:properties_phi_psi}. Therefore, the function $f_T^{{\rm nc},1}$ is 152-smooth.
    Similarly, we can prove the function $f_T^{{\rm nc},2}$  is also $152$-smooth, which finish the proof.
\end{proof}

\begin{lemma}[{\cite[Lemma 2]{arjevani2023lower}}]\label{cor:f^nc}
The function $f^{\rm nc}_T$ defined in equation (\ref{eq:function-nc}) satisfies:
\begin{enumerate}[itemsep=3pt,topsep=2pt]
\item[(a)]  The function value of $f^{\rm nc}_T$ holds $f^{\rm nc}_T(\vzero)-\inf_{\vy\in\BR^T}f^{\rm nc}_T(\vy)\leq12T$.
\item[(b)]  The gradient of $f^{\rm nc}_T$ is $l_0$-Lipschitz continuous, where $l_0=152$.
\item[(c)]  For all $\vx=[x_1,\dots,x_T]\in\BR^T$, we have $\prog(\nabla f^{\rm nc}_T(\vx)) \leq \prog(\vx)+1$.
\item[(d)]  For all $\vx=[x_1,\dots,x_T]\in\BR^T$ such that ${\rm supp}(\vx)\subseteq\left\{1,2,\cdots,T-1\right\}$, we have~$\|\nabla f^{\rm nc}_T(\vx)\|>1$.
\end{enumerate}\vskip0.1cm
\end{lemma}

We also provide the following lemmas for our later analysis, extending the results of Lemmas~5.1 and 5.2 from \citet{zhou2019lower}. 

\begin{lemma}\label{lem:scale}
	Suppose that the function $g:\mathbb{R}^T\rightarrow\mathbb{R}$ is $\hat{L}$-smooth and has lower bound~$g^*=\inf_{\vy\in\BR^T} g(\vy)$, 
    then the function $\hat g(\vx)=\alpha g(\beta \vx)$ is $\alpha\beta^2\hat{L}$-smooth and satisfies 
    \begin{align*}
    \|\nabla \hat g(\vx)\|=\alpha\beta\|\nabla g(\beta\vx)\| \qquad\text{and}\qquad \hat g(\mathbf{0})-\hat g^*=\alpha (g(\mathbf{0})-g^*) 
    \end{align*}
    for all $\alpha,\beta>0$, where 
    $\hat g^*=\inf_{\vy\in\BR^T} \hat g(\vy)$.
    If we further suppose the function $g$ is~$\hat\mu$-PL, then the function $\hat g$ is $\alpha\beta^2\hat\mu$-PL.
\end{lemma}
\begin{proof}
     For all $\vx,\vy\in\BR^{mn}$, the smoothness of $g:\BR^m\to\BR$ implies
     \begin{align*}
             \Norm{\nabla\hat g(\vx)-\nabla\hat g(\vy)} 
             = \alpha\beta\Norm{\nabla g(\beta \vx)-\nabla g(\beta \vy)}
             \leq & \alpha\beta \hat{L}\Norm{\beta \vx-\beta  \vy} 
             \leq \alpha\beta^2 \hat{L}\Norm{\vx-\vy},
     \end{align*}
     which means $\hat g$ is $\alpha\beta^2 \hat{L}$-smooth.
     
     We can verify that $\hat g(\mathbf{0})=\alpha g(\mathbf{0})$ and $\hat g^*=\alpha g^*$, which means
     \begin{align*}
         \hat g(\mathbf{0})-\hat g^*=\alpha (g(\mathbf{0})-g^*).
     \end{align*}   
     
     For any $\vx\in\BR^{mn}$, the PL condition of $g:\BR^m\to\BR$ implies
     \begin{align*}
             \Norm{\nabla\hat g(\vx)}^2 
             = \alpha^2\beta^2\Norm{\nabla g(\beta\vx)}^2
             \geq  2\alpha^2\beta^2\hat{\mu}(g(\beta\vx)-g^*)
             = 2\alpha\beta^2\hat{\mu}(\hat g(\vx)-\hat g^*),
     \end{align*}
     which means $\hat g$ is $\alpha\beta^2\hat{\mu}$-PL.
     Hence, we finish the proof.
\end{proof}

\begin{lemma}\label{lem:orthogonal transformation}
Given a function \mbox{$g:\mathbb{R}^T\rightarrow\mathbb{R}$} that is $\hat{L}$-smooth, we define $f_i(\vx)=g(\mU^{(i)}\vx)$ for $\vx\in\BR^{mT}$ and $i\in[m]$, where $\mU^{\left(i\right)}=[\ve_{(i-1)T+1},\cdots,\ve_{iT}]^\top\in\BR^{T\times mT}$ and we denote $\ve_k=[0,\dots,1,\dots,0]^\top\in\BR^{mT}$ as the $k$-th standard basis vector in Euclidean space~$\BR^{mT}$,
then the functions $\{f_i:\mathbb{R}^{mT}\rightarrow\mathbb{R}\}_{i=1}^m$ are $\hat{L}/\sqrt{m}$-mean-squared smooth, 
and the function $f=\frac{1}{m}\sum_{i=1}^{m}f_i$ is $\hat{L}/m$-smooth and satisfies 
\begin{align*}
f(\mathbf{0})-\inf_{\vy\in\BR}f(\vy)=g(\mathbf{0})-\inf_{\vy\in\BR}g(\vy).
\end{align*}
If we further suppose the function $g$ is~$\hat\mu$-PL, then the function $f$ is~$\hat\mu/n$-PL.
\end{lemma}
\begin{proof}
For any $\vx,\vy\in\BR^{mn}$, the smoothness of $g:\BR^m\to\BR$ implies
\begin{align*}
& \Vert\nabla f_i(\vx)-\nabla f_i(\vy)\Vert \\
= & \Big\Vert(\mU^{(i)})^\top\nabla g(\mU^{(i)}\vx)-(\mU^{(i)})^\top\nabla g(\mU^{(i)}\vy)\Big\Vert \\
= & \big\Vert\nabla g(\mU^{(i)}\vx)-\nabla g(\mU^{(i)}\vy)\big\Vert \\
\leq & \hat{L}\big\Vert \mU^{(i)}(\vx-\vy)\big\Vert 
\leq \hat{L}\Vert \vx-\vy\Vert,
\end{align*}
and
\begin{align*}
 & \frac{1}{n}\sum_{i=1}^n\Vert\nabla f_i(\vx)-\nabla f_i(\vy)\Vert^2  \\
=& \frac{1}{n}\sum_{i=1}^{n}\Big\Vert(\mU^{(i)})^\top\nabla g(\mU^{(i)}\vx)-(\mU^{(i)})^\top\nabla g(\mU^{(i)}\vy)\Big\Vert^2 \\
=& \frac{1}{n}\sum_{i=1}^{n}\big\Vert\nabla g(\mU^{(i)}\vx)-\nabla g(\mU^{(i)}\vy)\big\Vert^2 \\
=& \frac{\hat{L}^2}{n}\sum_{i=1}^{n}\big\Vert \mU^{(i)}(\vx-\vy)\big\Vert^2 
\leq  \frac{\hat{L}^2}{n}\Vert \vx-\vy\Vert^2.
\end{align*}
This implies each $f_i$ is $\hat{L}$-smooth and $\{f_i\}_{i=1}^n$ are $\hat{L}/\sqrt{n}$-mean-squared smooth.

Consider the facts $f(\vzero)=g(\vzero)$ and
\begin{align}\label{eq:PL-finite-sum}
      \inf\limits_{\vx\in \mathbb{R}^{mn}}\sum_{i=1}^{n}g(\mU^{(i)}\vx)=\sum_{i=1}^{n}\inf\limits_{\vx\in \mathbb{R}^{mn}}g(\mU^{(i)}\vx)=n\inf\limits_{\vy\in \mathbb{R}^{m}}g(\vy),
\end{align}
then we have
$f(\mathbf{0})-\inf_{\vy\in\BR}f(\vy)=g(\mathbf{0})-\inf_{\vy\in\BR}g(\vy)$.

For all $\vx\in\BR^{mn}$, the PL condition of $g:\BR^m\to\BR$ implies
\begin{align*}
	& \Vert\nabla f(\vx)\Vert^2 = \frac{1}{n^{2}}\Big\Vert\sum_{i=1}^{n}\nabla f_i(\vx)\Big\Vert^2 \\
	\,\,\,=\,\,\, & \frac{1}{n^{2}}\Big\Vert\sum_{i=1}^{n}(\mU^{(i)})^\top\nabla g(\mU^{(i)}\vx)\Big\Vert^2 \\
	\,\,\,=\,\,\, & \frac{1}{n^{2}}\sum_{i=1}^{n}\big\Vert\nabla g(\mU^{(i)}\vx)\big\Vert^2\\
	\overset{\,\,(\ref{asm:PL})\,\,}{\geq} & \frac{2\hat{\mu}}{n^2}\sum_{i=1}^{n}\Big(g(\mU^{(i)}\vx)-\inf\limits_{\vx\in \mathbb{R}^{mn}}g(\mU^{(i)}\vx)\Big)\\
	\overset{(\ref{eq:PL-finite-sum})}{=} & \frac{2\hat{\mu}}{n}\big(f(\vx)-\inf_{\vy\in\BR}f(\vy)\big),
\end{align*}
which means $f$ is $\hat{\mu}/n$-PL. Hence, we finish the proof.
\end{proof}

\subsection{The Lower Bound on Communication Rounds}

The communication complexity of the decentralized optimization depends on the topology of the network.
We let $\fG=\{\fV,\fE\}$ be the graph associated with the network in the problem, where the node set $\fV=\{1,\dots,m\}$ corresponds to the $m$ agents and the edge set $\fE=\{(i,j):\text{agents $i$ and $j$ are connected}\}$ describes the connectivity of the agents.  
We use use the notation ${\rm dia}(\fG)$ to present the diameter of graph $\fG$.

We introduce the ring-lattice graph and its properties as follows.

\begin{definition}[Ring-Lattice Graph \cite{watts1998collective}]\label{def:graph}
    Given any positive integers $m,k$ such that $k$ is even and {$2\leq k < m-1$}, the $k$-regular ring-lattice graph over nodes $\{1,\cdots, m\}$, denoted by $\mathcal{R}_{m,k}$, is an undirected graph in which each node $i \in [m]$ is connected with the set of nodes $\{(i+\ell) \;\mathrm{mod}\;m: \ell\in\BZ, 1\leq |\ell|\leq k/2\}$ where  $(i+\ell) \;\mathrm{mod}\;m$ is defined as 
    \begin{align*}    
    (i+\ell) \;\mathrm{mod}\;m =\begin{cases}
    i+\ell &\text{if $1\leq  i+\ell\leq m$};\\
    i+\ell+m &\text{if $i+\ell<1$};\\
    i+\ell-m &\text{if $i+\ell> m$}.\\
    \end{cases}
    \end{align*}
    Additionally, we use notation $\mL_{m,k}\in\BR^{m\times m}$ to present the Laplacian matrix of the ring-lattice graph $\mathcal{R}_{m,k}$ and $\lambda_{m-1}(\mL_{m,k})$ to present its second smallest eigenvalue.
\end{definition}\vskip0.1cm

\begin{lemma}[{\citep[Lemma 1]{yuan2022revisiting}}]\label{lem:distance}
    For any two nodes $i,j\in[m]$ in the ring-lattice graph~$\fR_{m,k}$, the distance between nodes $i$ and $j$ is  
    \begin{align*}
        \left\lceil \frac{2\min\{j-i, i+m-j\} }{k}\right\rceil.    
    \end{align*}
    In particularly, the diameter of $\mathcal{R}_{m,k}$ is $\lceil {2\lfloor m/2\rfloor}/{k}\rceil=\Theta({m}/{k})$.
\end{lemma}

Recently, \citet{yuan2022revisiting} showed that for all fixed $m\geq 2$ and $\gamma\in[1-\cos(\pi/m),1]$, we can always establish a mixing matrix associated with some ring-lattice graph that has the spectral gap $\gamma$.

\begin{lemma}[{\citep[Theorem 1]{yuan2022revisiting}}]\label{lem:yuankun}
    For any $m\geq2$ and $\gamma\in[1-\cos({\pi}/{m}),1]$, there exists a ring-lattice graph $\mathcal{G}=\{\fV,\fE\}$ with an associated mixing matrix $\mW\in\BR^{m\times m}$ such that $\lvert\fV\rvert=m$, $1-\lambda_2(\mW)=\gamma$, and ${\rm dia(\fG)}=\Omega(1/\sqrt{\gamma})$.
\end{lemma}

Based on above results, we can construct the graph whose nodes contains two disjoint subsets with the specific  distance.

\begin{lemma}\label{lem:big gamma}
    For any $m\geq2$ and $\gamma\in[1-\cos(\pi/m),1]$, there exists a ring-lattice graph $\mathcal{G}=\{\fV,\fE\}$ with associated mixing matrix $\mW\in\BR^{m\times m}$ and $\fV_1,\fV_2\subseteq\fV$ 
    such that~$\lvert\fV\rvert=m$, $\lvert\fV_1\rvert\geq m/3$, $\lvert\fV_2\rvert\geq m/3$, $1-\lambda_2(\mW)=\gamma$, and the distance between $\fV_1$ and~$\fV_2$ satisfies ${\rm dist}(\fV_1,\fV_2)=\Omega(1/\sqrt{\gamma})$.
    Specifically, such graph $\fG$ and matrix $\mW$ can be constructed by
    \begin{align*}
        \mathcal{G}=\begin{cases}
            \mathcal{R}_{m,m-1},& \gamma\in[\max\{1-\cos({\pi}/{m}),1-\cos({\pi}/{9})\},1], \\ 
            \mathcal{R}_{m,k},& \text{otherwise}
        \end{cases}
    \end{align*}
    and
    \begin{align*}
        \mW=\begin{cases}
            \dfrac{\gamma}{m}\mathbf{1}\mathbf{1}^\top+(1-\gamma)\mI,& \gamma\in[\max\{1-\cos({\pi}/{m}),1-\cos({\pi}/{9})\},1],\\[0.1cm]
            \mI-\dfrac{\gamma}{\lambda_{m-1}(\mL_{m,k})}\mL_{m,k},& \text{otherwise},
        \end{cases}    \end{align*}
    where 
    \begin{align*}
    k=2\left\lceil m\sqrt{\frac{3\gamma}{\pi^2(1-\pi^2/12)}}\,\right\rceil.   
    \end{align*}    
\end{lemma}
\begin{proof}
Based on Lemma \ref{lem:yuankun}, there exists a ring-lattice graph $\mathcal{G}=\{\fV,\fE\}$ associated with the mixing matrix $\mW\in\BR^{m\times m}$ such that $\lvert\fV\rvert=m$, $1-\lambda_2(\mW)=\gamma$, and ${\rm dia(\fG)}=\Omega(1/\sqrt{\gamma})$.
Let $\fV_1=\{1,\dots,\lceil m/3\rceil\}$ and $\fV_2=\{\lfloor m/2\rfloor+1,\dots,\lfloor m/2\rfloor+\lceil m/3\rceil\}$, which means~$\lvert\fV_1\rvert\geq m/3$ and $\lvert\fV_2\rvert\geq m/3$.
Based on Lemma \ref{lem:distance}, the distance between node sets $\fV_1$ and $\fV_2$ satisfies
\begin{align*}
    {\rm dist}(\fV_1,\fV_2) &=\left\lceil \frac{2\min\limits_{i\in\fV_1,j\in\fV_2}\{j-i, i+m-j\} }{k}\right\rceil\\
    &\geq\left\lceil \frac{2\min\{\lfloor m/2\rfloor+1-\lceil m/3\rceil, 1+m-\lfloor m/2\rfloor-\lceil m/3\rceil\} }{k}\right\rceil\\
    &\geq\left\lceil \frac{2\min\{ m/2-1/2+1- m/3-2/3, 1+m-m/2- m/3-2/3\}}{k}\right\rceil\\
    &=\left\lceil \frac{{m-1}}{3k}\right\rceil
    =\Omega\left(\frac{m}{k}\right) =\Omega({\rm dia}(\fG))
    =\Omega\left(\frac{1}{\sqrt{\gamma}}\right).
\end{align*} 
\end{proof}

Applying Lemma \ref{lem:big gamma}, we provide the lower bound on communication rounds for our smooth finite-sum decentralized optimization problem. 

\begin{theorem}\label{thm:lower-communication}
    For any $L>0,\bar L>0,m\geq2,n>0,\Delta>0,\gamma\in{[1-\cos(\pi/m),1]}$ and~$\epsilon>0$ with $\bar L\geq L$ and $\epsilon=\fO(\sqrt{\Delta L}\,)$, there exist matrix $\mW\in\BR^{m\times m}$ with spectral gap $\gamma$ and $\bar L$-mean-squared smooth functions $\{f_{i,j}:\mathbb{R}^{d}\rightarrow\mathbb{R}\}_{i,j=1}^{m,n}$ with $d=\fO(\Delta L\epsilon^{-2})$ such that their average $f=\frac{1}{mn}\sum_{i=1}^{m}\sum_{j=1}^{n}f_{i,j}$ is $L$-smooth with $f(\mathbf{0})-f^*\leq\Delta$. 
    In order to find an $\epsilon$-stationary point of the function $f$, any LIFO algorithm needs at least~$\Omega(\Delta L\epsilon^{-2}/\sqrt{\gamma}\,)$ communication rounds.
\end{theorem}
\begin{proof}
According to Lemma \ref{lem:big gamma}, there exists a ring-lattice graph $\mathcal{G}=\{\fV,\fE\}$ with the associated mixing matrix $\mW\in\BR^{m\times m}$ and $\fV_1,\fV_2\subseteq\fV$ such that $\lvert\fV\rvert=m$, $1-\lambda_2(\mW)=\gamma$, $\lvert\fV_1\rvert\geq m/3$, $\lvert\fV_2\rvert\geq m/3$, 
and the distance between the sets $\fV_1$ and $\fV_2$ satisfies~${\rm dist}(\fV_1,\fV_2)=\Omega(1/\sqrt{\gamma})$.

We consider the nonconvex function $f^{\rm nc}_T$ defined in equation (\ref{eq:function-nc}) with 
\begin{align}\label{eq:comm-T}
T=\left\lfloor \frac{\Delta L}{36l_0\epsilon^2}\right\rfloor, \quad\text{where}~~\epsilon\leq\sqrt{\frac{\Delta L}{36l_0}},\qquad\text{and}\qquad l_0=152.    
\end{align}
According to Lemma \ref{lem:scale} with 
\begin{align}\label{eq:comm-alpha-beta}
g(\vx)=f^{\rm nc}_T(\vx), \qquad \alpha=\frac{3l_0\epsilon^2}{L}, \qquad \text{and} \qquad 
\beta=\frac{L}{3l_0\epsilon},    
\end{align}
we can conclude the function $\hat g(\vx)=\alpha f^{\rm nc}_T(\beta\vx)$ is $\alpha\beta^2l_0$-smooth and satisfies
\begin{align}\label{eq:comm-grad}
\|\nabla \hat g(\vx)\|=\alpha\beta\|\nabla f^{\rm nc}_T(\beta\vx)\|    
\end{align}
and
\begin{align}\label{eq:comm-lower-hat-g}
\hat g(\mathbf{0})-\inf_{\vy\in\BR^T}\hat g(\vy) = \alpha \Big(f^{\rm nc}_T(\mathbf{0})-\inf_{\vy\in\BR^T} f^{\rm nc}_T(\vy)\Big) \leq 12\alpha T, 
\end{align}
where the inequality is based on Lemma \ref{cor:f^nc}(a).

We then construct the hard instance based on the decomposition 
$f_T^{\rm nc}=f_T^{\rm nc,1} + f_T^{\rm nc,2}$ shown in equations (\ref{eq:nc1nc2-decompose}) and (\ref{eq:nc1nc2}).
Specifically, we take $f_{i,j}$ as 
\begin{align}\label{eq:comm-instance}
f_{i,j}(\vx)=
\begin{cases}
\dfrac{\alpha m}{\lvert\fV_1\rvert} f_T^{\rm nc,1}(\beta\vx),   &{i\in \fV_1}, \\[0.35cm]
\dfrac{\alpha m}{\lvert\fV_2\rvert} f_T^{\rm nc,2}(\beta\vx),   &{i\in \fV_2}, \\[0.3cm]
0,   &{\text{otherwise}},
\end{cases}
\end{align}
where $f_T^{\rm nc,1}$ and $f_T^{\rm nc,2}$ are defined in equation (\ref{eq:nc1nc2}). According to Lemma \ref{lem:function-nc-decompose}, Lemma~\ref{lem:scale}, and equation (\ref{eq:comm-instance}), each individual function $f_{i,j}$ is $\alpha\beta^2ml_0/\min\{\lvert\fV_1\rvert,\lvert\fV_2\rvert\}$-smooth.
Therefore, functions $\{f_{i,j}\}_{i,j=1}^{m,n}$ are $\alpha\beta^2ml_0/\min\{\lvert\fV_1\rvert,\lvert\fV_2\rvert\}$-mean-squared smooth.

Additionally, the definition of $f_{i,j}$ in equation (\ref{eq:comm-instance}) means
\begin{align}\label{eq:comm-fun}
f(\vx)=\frac{1}{mn}\sum_{i=1}^m\sum_{j=1}^nf_{i,j}(\vx)=\hat g(\vx),
\end{align}
which is $\alpha\beta^2l_0$-smooth.

The parameter settings in equations (\ref{eq:comm-T}) and (\ref{eq:comm-alpha-beta}) imply
\begin{align*}
	\alpha\beta^2l_0=\frac{L}{3}\leq L, \quad \frac{\alpha\beta^2ml_0}{\min\{\lvert\fV_1\rvert,\lvert\fV_2\rvert\}} = \frac{mL}{3\min\{\lvert\fV_1\rvert,\lvert\fV_2\rvert\}}\leq L \leq \bL, \quad\text{and}\quad  
    12\alpha T\leq\Delta.
\end{align*}
Therefore, the function $f=\frac{1}{mn}\sum_{i=1}^m\sum_{j=1}^nf_{i,j}$ is $L$-smooth and satisfies $f(\mathbf{0})-f^*\leq\Delta$ and the function set $\{f_{i,j}\}$ is $\bar L$-mean-squared smooth, which satisfies our requirements.

Now we show that any LIFO algorithm require at least 
\begin{align}
(T-1){\rm dist}(\fV_1,\fV_2)=\Omega\left(\frac{L \Delta}{\sqrt{\gamma}\epsilon^2}\right)     
\end{align}
communication rounds to to achieve an $\epsilon$-stationary point of $f(\cdot)$.

According to Corollary \ref{cor:f^nc}, equations (\ref{eq:comm-alpha-beta}), (\ref{eq:comm-grad}) and (\ref{eq:comm-fun}), we have
\begin{align*}
   \|\nabla f(\vx)\| >\alpha\beta=\epsilon.
\end{align*}
for all $\vx=[x_1,\dots,x_T]^\top\in\BR^d$ such that $\vx_T=0$.

We let
\begin{align}\label{dfn:nnz}
\begin{split}    
    & {\rm nnz}(s,i) \triangleq
     \max\big\{t\in\{0,1,\dots,T\}: \\
     & \qquad\qquad\qquad\qquad\text{there exists $\vy=[y_1,\dots,y_T]^\top\in\mathcal{M}_i^s$ such that $y_t\neq0$}\big\},
\end{split}
\end{align}
where we denote $y_0=1$.

Consider the definitions of $f_{i,j}$, $f_T^{\rm nc, 1}$ and $f_T^{\rm nc, 2}$ in equations (\ref{eq:nc1nc2})
and (\ref{eq:comm-instance}).
Notice that the fact $\Psi(0)=\Psi'(0)=0$ implies the partial derivatives of 
\begin{align*}
\Psi(-x_{k-1})\Phi(-x_k)-\Psi(x_{k-1})\Phi(-x_k)     
\end{align*}
with respect to $x_{k-1}$ or $x_k$ are zero when $x_{k-1}=x_k=0$.
This means the computation of any LIFO algorithm holds:
\begin{enumerate}
    \item If $i\in\fV_1$ and ${\rm nnz}(s,i)$ is even, one step of local computation can increase at most one dimension for memory of node~$i$.
    \item If $i\in\fV_2$ and ${\rm nnz}(s,i)$ is odd, one step of local computation can increase at most one dimension for memory of node $i$. 
    \item Otherwise, one step of local computation cannot increase the dimension for memory of node~$i$.
\end{enumerate}
In summary, we have
\begin{align}\label{eq:nnz1}
    {\rm nnz}(s+1,i) \leq \begin{cases}
        {\rm nnz}(s,i)+1,  & \text{if}~~i\in\fV_1,\  {\rm nnz}(s,i)\equiv 0~({\rm mod}~2), \\
        {\rm nnz}(s,i)+1, & \text{if}~~i\in\fV_2,\  {\rm nnz}(s,i)\equiv 1~({\rm mod}~2), \\
        {\rm nnz}(s,i), &\text{otherwise}.
    \end{cases}
\end{align}

We consider the cost to reach the second coordinate from the initial status that~$\fM_i^0=\{\vzero\}$ for all $i\in[m]$.
According to equation (\ref{eq:nnz}), we need to let a node in $\fV_2$ reach the first coordinate, which requires at least one local computation step on some node in $\fV_1$ first. 
Then, according to definitions of LIFO algorithm (Definition~\ref{dfn:LIFO}), one must perform at least ${\rm dist}(\fV_1,\fV_2)$ local communication rounds for a node in $\fV_2$ to receive the information of the first coordinate from some node in $\fV_1$. 
After above rounds, we can perform at least 1 computation on nodes in $\fV_2$ to reach the second coordinate.
In summary, to reach the second coordinate requires at least ${\rm dist}(\fV_1,\fV_2)$ local communication rounds.
Similarly, to reach the $k$-th coordinate, any LIFO algorithm must perform at least~$(k-1){\rm dist}(\fV_1,\fV_2)$ local communication rounds.
Therefore, to achieve the $\epsilon$-stationary point,  we require $x_T\neq 0$ that needs at least 
\begin{align*}
(T-1){\rm dist}(\fV_1,\fV_2)=\Omega\left(\frac{L\Delta}{\sqrt{\gamma}\epsilon^{2}}\right)    
\end{align*}
communications rounds.
\end{proof}

\begin{remark}
The results of Corollary \ref{cor:complexity-general} and Theorem \ref{thm:lower-communication} show that the complexity on communication rounds of \DEAREST~is near-optimal for all $\gamma\in{[1-\cos(\pi/m),1]}$ and~$\bL\geq L$.
As pointed by \citet{yuan2022revisiting}, such result is more general than the lower bounds established by the linear graph \cite{hendrikx2021optimal,scaman2018optimal,schmidt2017minimizing,lu2021optimal} that only work in the case of~$\gamma=1-\cos(\pi/m)$, 
since $\cos(\pi/m)$ approaches to 1 for large $m$.
More importantly, our results indicate the communication complexity in distributed nonconvex optimization  essentially depends on the global smoothness parameter $L$, rather than the mean-squared or the local smoothness parameters which used in existing work \cite{sun2020improving,xin2022fast,zhan2022efficient,li2022destress,metelev2024decentralized,lu2021optimal,yuan2022revisiting}.
\end{remark}

\subsection{The Lower Bound on LIFO calls}

It is natural that the lower bound on the LIFO calls for decentralized finite-sum optimization problem~(\ref{prob:main}) with local functions (\ref{func:local}) can be established by considering the IFO complexity of the finite-sum optimization problem on single machine with~$N=mn$ individual functions.
Unlike existing lower bounds on IFO complexity $\bL$~\cite{fang2018spider,zhou2019lower,li2021page,lan2018optimal,agarwal2015lower,woodworth2016tight} only consider the mean-squared smoothness parameter, our constructions address both~$L$ and $\bL$, which is formally presented as follows.

\begin{theorem}\label{thm:LFO}
        For all $\bar L>0$, $L>0$, $N>0$, $\Delta>0$ and $\epsilon>0$ with $\bar L\geq L$ and $\epsilon=\fO((\min\{L,\bar L/\sqrt{N}\}\Delta)^{1/2})$, there exist $\bar L$-mean-squared smooth functions \mbox{$\{\tf_i:\BR^d\to\BR\}_{i=1}^N$} such that the function $\tf=\frac{1}{N}\sum_{i=1}^{N}\tf_i$ is $L$-smooth and satisfies that $\tf(\vzero)-\inf_{\vy\in\BR^d}\tf(\vy)\leq \Delta$. In order to find an $\epsilon$-stationary point of $\tf$, any IFO algorithm (defined in Appendix \ref{appendix:IFO}) needs at least $\Omega\big(N+ \min\{NL,\sqrt{N}\bL\}\Delta\epsilon^{-2}\big)$ IFO calls.
\end{theorem}
\begin{proof}
We consider the nonconvex function $f^{\rm nc}_T$ defined in equation (\ref{eq:function-nc}) with 
\begin{align}\label{eq:ifo-T}
T=\left\lfloor \frac{\min\{L,\bar L/\sqrt{N}\}\Delta}{24l_0\epsilon^2}\right\rfloor, \quad \text{where}~~\epsilon\leq \sqrt{\frac{\min\{L,\bar L/\sqrt{N}\}\Delta}{24l_0}}~~\text{and}~~l_0=152.    
\end{align}  
According to Lemma \ref{lem:scale} with 
\begin{align}\label{eq:ifo-alpha-beta}
g(\vx)=f^{\rm nc}_T(\vx), \quad \alpha = \frac{2l_0\epsilon^2}{\min\big\{L,\bar L/\sqrt{N}\big\}},\quad \text{and} \quad 
\beta=\frac{\min\big\{L\sqrt{N},\bar L\big\}}{\sqrt{2}l_0\epsilon},    
\end{align}
we can conclude the function $\hat g(\vx)=\alpha f^{\rm nc}_T(\beta\vx)$ is $\alpha\beta^2l_0$-smooth and holds
\begin{align}\label{eq:lower-grad-hat-g}
    \|\nabla \hat g(\vx)\|=\alpha\beta\|\nabla f^{\rm nc}_T(\beta\vx)\| 
\end{align}
and
\begin{align}\label{eq:lower-hat-g}
\hat g(\mathbf{0})-\inf_{\vy\in\BR^T}\hat g(\vy) = \alpha \Big(f^{\rm nc}_T(\mathbf{0})-\inf_{\vy\in\BR^T} f^{\rm nc}_T(\vy)\Big) \leq 12\alpha T, 
\end{align}
where the inequality is based on Lemma \ref{cor:f^nc}(a).

We construct the hard instance according to Lemma \ref{lem:orthogonal transformation} with $g(\vx)=\hat g(\vx)$, which results the functions
\begin{align}\label{eq:ifo-instance-1}
     \tf_i(\vx)=\alpha f^{\rm nc}_T(\beta \mU^{(i)}\vx) 
     \qquad\text{and}\qquad
     \tf(\vx)=\frac{1}{N}\sum_{i=1}^{N}\tf_i(\vx) 
\end{align}
such that $\{\tf_i\}_{i=1}^N$ is $\alpha\beta^2l_0/\sqrt{N}$-mean-squared smooth and $\tf$ is $\alpha\beta^2l_0/N$-smooth with 
\begin{align*}
\tf(\vx^0)-\tf^*=\hat g(\mathbf{0})-\inf_{\vy\in\BR^T}\hat g(\vy)
\overset{(\ref{eq:lower-hat-g})}{\leq} 12\alpha T,     
\end{align*}
where $\tf^*=\inf_{\vy\in\BR^T}\tf(\vy)$.

The settings of $\alpha$, $\beta$ and $T$ in equations (\ref{eq:ifo-T}) and (\ref{eq:ifo-alpha-beta}) imply
\begin{align*}
	\frac{\alpha\beta^2 l_0}{\sqrt{N}}\overset{(\ref{eq:ifo-alpha-beta})}{\leq}\bar L, \qquad \frac{\alpha\beta^2l_0}{N}\overset{(\ref{eq:ifo-alpha-beta})}{\leq} L, \qquad\text{and}\qquad  
    12\alpha T \overset{(\ref{eq:ifo-T})}{\leq}\Delta.
\end{align*}
Therefore, the function set 
$\{\tf_i\}_{i=1}^N$ is $\bar L$-average smooth and the function $\tf$ is $L$-smooth with $\tf(\vx^0)-\tf^*\leq\Delta$, which satisfies our requirements.

Now we show that any IFO algorithm require at least 
$NT/2$ IFO calls to achieve an $\epsilon$-stationary point $\hat\vx$ of the function $f$.
According to equation (\ref{eq:lower-grad-hat-g}), Lemma \ref{cor:f^nc}(d) and Lemma \ref{lem:scale}, for all~$\vx\in\BR^{NT}$ with ${\rm supp}(\mU^{(j)}\vx)\subseteq\left\{1,2,\cdots,T-1\right\}$, we have
\begin{align}\label{eq:ifo-grad-upper}
   \|\nabla f_j(\vx)\| = \big\|\alpha\beta(\mU^{(j)})^\top \nabla f_T^{\rm nc}(\beta\mU^{(j)}\vx)\big\| = \alpha\beta\big\|\nabla f_T^{\rm nc}(\beta\mU^{(j)}\vx)\big\| > \alpha\beta.
\end{align} 

We consider the vector $\vx\in\BR^{NT}$ which is achieved by an IFO algorithm with at most~$\lfloor NT/2 \rfloor$ IFO calls.
The zero-chain property shown in Lemma \ref{cor:f^nc}(c) implies such vector $\vx$ has at most~$\lfloor NT/2 \rfloor$ non-zero entries.
We partition $\vx\in\BR^{NT}$ into $N$ vectors $\vy^{(1)},\dots,\vy^{(N)}\in\BR^{T}$ such that $\vy^{(j)}=\mU^{(j)}\vx\in\BR^{T}$,
then there at least $N/2$ vectors in $\{\vy^{(j)}\}_{j=1}^N$ such that each of them has at least one zero entry.
The zero-chain property (Lemma \ref{cor:f^nc}(c)) means
there exists index set $\fI\subseteq [N]$ with $|\fI|\geq N/2$ such that each~$j\in\fI$ satisfies ${\rm supp}(\vy^{(j)})\subseteq\left\{1,2,\cdots,T-1\right\}$.
Therefore, we have
\begin{align*}
& \|\nabla f(\vx)\|=
\frac{1}{N}\sqrt{\sum_{i=1}^N\|\nabla f_{i}(\vx)\|^2} \\
\geq & \frac{1}{N}\sqrt{\sum_{j\in\fI}\|\nabla f_j(\vx)\|^2}
\overset{(\ref{eq:ifo-grad-upper})}{>}\frac{1}{N}\sqrt{\frac{N}{2}}\,\alpha\beta=\epsilon.
\end{align*}
Hence, any IFO algorithm requires at least 
\begin{align*}
 \left\lfloor \frac{NT}{2} \right\rfloor+1=\Omega\left(\frac{\min\{NL,\sqrt{N}\bar L\}\Delta}{\epsilon^2}\right)  
\end{align*}
IFO calls to achieve an $\epsilon$-stationary point of function $\tf(\cdot)$ defined in equation (\ref{eq:ifo-instance-1}).

Then, we consider the lower bound dominated by $N$ and construct the another hard instance.
We assume that $\epsilon\leq\sqrt{L\Delta/2}$.
We follow the functions provided by \citet{li2021page}, i.e., define $\tf_i:\BR^d\to\BR$ and $\tf:\BR^d\to\BR$ as
\begin{align}\label{eq:ifo-instance-2}
    \tf_i(\vx)=c\langle \vv_i,\vx\rangle+\frac{L}{2}\|\vx\|^2
    \qquad\text{and}\qquad
     \tf(\vx)=\frac{1}{N}\sum_{i=1}^{N} \tf_i(\vx).
\end{align}
for all $i\in[N]$, where $\vv_i=\ve_i$, $c=2\epsilon\sqrt{N}$ and $d=N$.
We can verify that for all~$\vx,\vy\in\BR^d$, it holds
\begin{align*}
    \|\nabla \tf(\vx)-\nabla \tf(\vy)\| = \Norm{\left(\frac{c}{N}\sum_{i=1}^{N}\vv_i+L\vx\right)-\left(\frac{c}{N}\sum_{i=1}^{N}\vv_i+L\vy\right)} = L\|\vx-\vy\|
\end{align*}
and
\begin{align*}
    & \frac{1}{N}\sum_{i=1}^{N}\|\nabla \tf_i(\vx) - \tf_i(\vy)\|^2 \\
    = & \frac{1}{N}\sum_{i=1}^{N}\|(c\vv_i+L\vx)-(c\vv_i+L\vy)\|^2 \\
    =  & L^2\|\vx-\vy\|^2 \leq \bar L^2\|\vx-\vy\|^2.
\end{align*}
Therefore, the function $\tf$ is $L$-smooth and the functions $\{\tf_i\}_{i=1}^N$ are $\bar L$-average smooth. 
Since we suppose the IFO algorithm start with the initial point $\vx^0=\vzero$, it holds
\begin{align*}
     \tf(\vx^0) - \inf_{\vy\in\BR^d} \tf(\vy) =0-\left(\frac{c}{N}\sum_{i=1}^{N}\langle\vv_i,\vx^*\rangle+\frac{L}{2}\|\vx^*\|^2\right)=\frac{c^2d}{2LN^2}\leq\Delta,
\end{align*}
where $\vx^*=-\frac{c}{LN}\sum_{i=1}^{N}\vv_i$ is the minimizer of $\tf$.
Following Theorem 2 of \citet{li2021page}, we can show that any IFO algorithm requires at least IFO calls of
$\Omega(N)$ to achieve an $\epsilon$-stationary point of the function $\hat f$ defined in equation (\ref{eq:ifo-instance-2}).

Combining results of above two hard instances, we achieve the lower bound on the IFO complexity of
\begin{align*}
    \Omega\left(N+\frac{\min\{NL,\sqrt{N}\bar L\Delta\}}{\epsilon^2}\right).
\end{align*}
\end{proof}

\begin{remark}
Applying Theorem \ref{thm:LFO} with $N=mn$ means the LIFO complexity shown in Corollary \ref{cor:complexity-general} is optimal.
\end{remark}

\subsection{The Lower Bound on Computation Rounds}

As pointed in Remark \ref{remark:not=proportion}, the communication rounds are not always proportion to the LIFO calls.
Therefore, we detailed analyze this lower bound as follows. 

\begin{theorem}\label{thm:lower-computation}
        For any $\bar L>0$, $L>0$, $m>0$, $n>0$, $\Delta>0$, and $\epsilon>0$ with $\bL\geq L$ and $\epsilon=\fO\big(\big(\min\{L,\bar L/\sqrt{N}\}\Delta\big)^{1/2}\big)$, there exist $\bar L$-mean-squared smooth functions $\{f_{i,j}:\BR^d\to\BR\}$ such that the function $f=\frac{1}{mn}\sum_{i=1}^{m}\sum_{j=1}^{n}f_{i,j}$ is $L$-smooth with $f(\vzero)-f^*\leq \Delta$. In order to find an $\epsilon$-stationary point of $f$, any LIFO algorithm needs at least $\Omega(n+(L+\min\{nL, \sqrt{n/m}\bar L\})\epsilon^{-2})$ computation rounds.
\end{theorem}
\begin{proof}
    We first show the lower bound in the view of linear speed-up.
    According to Theorem \ref{thm:LFO} with $N=mn$ and $\epsilon=\fO((\min\{L,\bar L/\sqrt{N}\}\Delta)^{1/2})$, 
    there exist $\bar L$-mean-squared smooth functions $\{f_{i,j}\}_{i=1,j=1}^{m,n}$ such that the function $f=\frac{1}{mn}\sum_{i=1}^{m}\sum_{j=1}^{n}f_{i,j}$ is $L$-smooth and satisfies $f(\vx^0)-f^*\leq \Delta$. 
    In order to find an $\epsilon$-stationary point of problem $f$, any IFO algorithm needs at least 
    \begin{align}\label{eq:IFO-computation}
        \Omega\left(mn+\frac{\min\{mnL,\sqrt{mn}\bar L\}}{\epsilon^2}\right).
    \end{align}
    IFO calls. 
    For the distributed setting with $m$ agents, any LIFO algorithm can perform at most $m$ LIFO calls in per computation rounds.
    Therefore, the IFO lower bound in equation (\ref{eq:IFO-computation}) directly implies the lower bound on the computation rounds of 
    \begin{align}\label{eq:IFO-computation2}
        \Omega\left(\frac{mn}{m}+\frac{\min\{mnL,\sqrt{mn}\bar L\}}{m\epsilon^2}\right) 
        = \Omega\left(n+\frac{\min\{nL,\sqrt{n/m}\bar L\}}{\epsilon^2}\right),
    \end{align}
    which nearly matches the upper bound (\ref{eq:complexity-compuation-rounds}) in Corollary \ref{cor:complexity-general} when $\bL/L<\sqrt{m/n}$.
    
The remain of the proof only needs to provide a hard instance with the lower bound of $\Omega(L\epsilon^{-2}\Delta)$ when $\bL/L<\sqrt{m/n}$.
We consider the nonconvex function $f^{\rm nc}_T$ defined in equation (\ref{eq:function-nc}) with 
\begin{align}\label{eq:ifo-T2}
T=\left\lfloor \frac{L\Delta}{12l_0\epsilon^2}\right\rfloor, \quad \text{where}~~\epsilon\leq \sqrt{\frac{L\Delta}{12l_0}}~~\text{and}~~l_0=152.    
\end{align}
According to Lemma \ref{lem:scale} with 
\begin{align}\label{eq:ifo-alpha-beta2}
g(\vx)=f^{\rm nc}_T(\vx), \qquad \alpha = \frac{l_0\epsilon^2}{L} \qquad \text{and} \qquad 
\beta=\frac{L}{l_0\epsilon},    
\end{align}
we can conclude the function $\hat g(\vx)=\alpha f^{\rm nc}_T(\beta\vx)$ is $\alpha\beta^2l_0$-smooth and satisfies
\begin{align}\label{eq:lower-grad-hat-g2}
\|\nabla \hat g(\vx)\|=\alpha\beta\|\nabla f^{\rm nc}_T(\beta\vx)\|    
\end{align}
and
\begin{align}\label{eq:lower-hat-g2}
\hat g(\mathbf{0})-\inf_{\vy\in\BR^T}\hat g(\vy) = \alpha \left(f^{\rm nc}_T(\mathbf{0})-\inf_{\vy\in\BR^T} f^{\rm nc}_T(\vy)\right) \leq 12\alpha T, 
\end{align}
where the inequality is based on Lemma \ref{cor:f^nc}(a).

We construct the hard instance as
\begin{align}\label{eq:difo-instance-2}
     f_{i,j}(\vx)=\hat g(\vx)
     \qquad\text{and}\qquad
     f(\vx)=\frac{1}{mn}\sum_{i=1}^{m}\sum_{j=1}^{n}f_{i,j}(\vx) 
\end{align}
such that $\{f_{i,j}\}_{i=1,j=1}^{m,n}$ is $\alpha\beta^2l_0$-mean-squared smooth and $f$ is $\alpha\beta^2l_0$-smooth with 
\begin{align*}
f(\vx^0)-f^*=\hat g(\mathbf{0})-\inf_{\vy\in\BR^T}\hat g(\vy)
\overset{(\ref{eq:lower-hat-g2})}{\leq} 12\alpha T.     
\end{align*}
The setting of $\alpha$, $\beta$ and $T$ implies
\begin{align*}
	\alpha\beta^2 l_0\overset{(\ref{eq:ifo-alpha-beta2})}{=} L \qquad\text{and}\qquad  
    12\alpha T \overset{(\ref{eq:ifo-T2})}{\leq}\Delta.
\end{align*}
Therefore, the function set 
$\{f_{i,j}\}_{i=1,j=1}^{m,n}$ is $L$-average smooth and the function $f$ is $L$-smooth with $f(\vx^0)-f^*\leq\Delta$, which satisfies our requirements.

Now we show that any LIFO algorithm require at least 
$T-1$ computation rounds to achieve an $\epsilon$-stationary point of the function $f(\cdot)$ at each node.
According to equations (\ref{eq:lower-grad-hat-g2})--(\ref{eq:difo-instance-2}),
Lemmas \ref{cor:f^nc}(d) and \ref{lem:scale}, for all vector $\vx_i\in\BR^{T}$ that satisfies ${\rm supp}(\vx_i)\subseteq\left\{1,2,\cdots,T-1\right\}$, we have
\begin{align}\label{eq:ifo-grad-upper2}
   \|\nabla f(\vx_i)\| = \big\|\alpha\beta \nabla f_T^{\rm nc}(\beta\vx_i)\big\| > \alpha\beta.
\end{align}

We consider the vectors $\vx_1,\dots,\vx_m\in\BR^{T}$ which are achieved by an LIFO algorithm with at most $T-1$ computation rounds.
Lemma \ref{cor:f^nc}(c) implies such $\vx_i$ has at most $T-1$ non-zero entries, that is ${\rm supp}(\vx_i)\subseteq\left\{1,2,\cdots,T-1\right\}$ for all~$i\in[m]$.
Therefore, we have
\begin{align*}
\|\nabla f(\vx_i)\| > \alpha\beta=&\epsilon.
\end{align*}
Hence, any LIFO algorithm requires at least 
\begin{align}\label{eq:IFO-computation3}
 T=\Omega\left(\frac{L\Delta}{\epsilon^2}\right)  
\end{align}
computation rounds to achieve an $\epsilon$-stationary point of function $f$. Combing results of~(\ref{eq:IFO-computation2}) and (\ref{eq:IFO-computation3}), we finish the proof.
\end{proof}

\begin{remark}
In a very recent work, \citet{metelev2024decentralized} claimed the lower bounds on the communication rounds and the computation rounds with respect to local smoothness parameters $L_\ell$ and $\bL_\ell$ (see Assumptions \ref{asm:smooth-local} and \ref{asm:smooth-local-mean-squared}) respectively. 
However, their analysis essentially only obtains the lower bound on the computation rounds of~$\Omega(nL_\ell\Delta\epsilon^{-2})$ \cite[at the end of page 45]{metelev2024decentralized}, and the desired lower bound of $\Omega(n+\sqrt{n}\bL_\ell\epsilon^{-2})$ shown in their statement \cite[Corollary 4.6]{metelev2024decentralized} implicitly requires the additional assumption of~$\bL_\ell/L_\ell\leq \fO(\sqrt{n}\,)$.
Recall that the discussion in Remark~\ref{prop:smooth} shows the ratio $\bL_\ell/L_\ell$ (consider $\bL/L$ when $m=1$) may be larger than  $\Omega(\sqrt{n}\,)$, which is not included in the analysis of \citet{metelev2024decentralized}.
In contrast, all of our theorems for lower bounds are valid for any~$L>0$ and $\bL>0$ such that $\bL\geq L$, which is general since it always holds for the tight $L$ and $\bL$ (see Remark  \ref{prop:smooth}).
\end{remark}

\section{The Results under the PL Condition}\label{sec:PL}

The PL condition (Assumption \ref{asm:PL}) suggests the function value gap is dominated by the square of gradient norm, which avoids the hardness of finding global solution in general smooth nonconvex optimization \cite{nemirovskij1983problem} and leads to the gradient descent method linearly converging to the global minimum without the convexity~\cite{karimi2016linear,polyak1963gradient,lojasiewicz1963topological}.
It covers a lot of popular applications, such as deep neural networks~\cite{liu2022loss,allen2019convergence,zeng2018global}, reinforcement learning~\cite{fazel2018global,agarwal2021theory,mei2020global,yuan2022general}, optimal control~\cite{bu2019lqr,fatkhullin2021optimizing} and matrix recovery~\cite{hardt2016identity,li2018algorithmic,bi2022local}.

For the single machine scenario, \citet{reddi2016stochastic,lei2017non} considered the finite-sum optimization under the PL condition by proposing SVRG-type methods \cite{johnson2013accelerating}, which achieves the $\epsilon$-suboptimal solution with at most $\fO((N+N^{2/3}\bkappa)\ln(1/\epsilon))$ IFO calls, where $N$ is the number of individual functions and $\bkappa$ is defined in equation~(\ref{def:kappa}).
Later, \citet{zhou2019faster,wang2019spiderboost,li2021page} improved the IFO upper bound to~$\fO((N+\sqrt{N}\bkappa)\ln(1/\epsilon))$ by stochastic recursive gradient estimator.
Additionally, \citet{yue2023lower} established the tight lower bound on the exact first-order oracle complexity of~$\Omega(\kappa\ln(1/\epsilon))$ for minimizing the PL function, where $\kappa$ is defined in equation~(\ref{def:kappa}).

In this section, we show the proposed \DEAREST~(Algorithm \ref{alg:DEAREST}) also achieves the near optimal complexity bounds for the decentralized smooth finite-sum problem under the PL condition.
We present the complexity analysis for \DEAREST~under the PL condition in Section~\ref{sec:PL-upper} and provide the corresponding lower bounds in Section~\ref{sec:PL-lower}, which extends and improves the results in our conference paper \cite{bai2024complexity} in the following aspects:
\begin{enumerate}[itemsep=3pt,topsep=2pt]
    \item[(a)] This manuscript considers the finite-sum structure (\ref{func:local}) in local functions, while our conference paper \cite{bai2024complexity} only considers the oracle of accessing the exact local gradient which can be regarded as the special case of $n=1$.
    \item[(b)] This manuscript distinguishes the difference between $L$ and $\bL$ in the algorithm design and analysis, while our conference paper \cite{bai2024complexity} only addresses the mean-squared smoothness parameter $\bL$.
    \item [(c)] This manuscript  proves the algorithm can find $\vx_i^T$ such that $\BE[f(\vx^T_i)-f^*] \leq \epsilon$ for every agent, while our conference paper \cite{bai2024complexity} only shows  $\BE[f(\vx^{\rm out})-f^*] \leq \epsilon$, where $\vx^{\rm out}$ is uniformly sampled from vectors $\{\vx_i^T\}_{i=1}^m$ on all $m$ agents.
\end{enumerate}

We summarize the theoretical results in this section and our conference paper in Tables \ref{table:pl-1} and \ref{table:pl-2}. To the best of our knowledge, no previous work specifically studies the complexity of decentralized optimization problem under the PL condition.

\begin{table}[t]
\caption{We present the communication rounds and the computation rounds for finding $\epsilon$-suboptimal solutions under the PL condition, where $\Delta=f(\bx^0)-f^*$ and $\vx^0\in\BR^d$ is the initial point of the algorithm.}\label{table:pl-1}
\begin{tabular}{ccccc}
    \hline
    Methods & \#Communication & \#Computation  &  Reference
    \\\hline\hline\addlinespace
    \DGDGT & $\fO\left(\dfrac{\kappa\ln(\Delta/\epsilon)}{\sqrt{\gamma}}\right)$ 
    & $\fO\left(n\kappa\ln\bigg(\dfrac{\Delta}{\epsilon}\bigg)\bigg)\right)$  
    & \citet{bai2024complexity}   \\\addlinespace
    \DRONE & $\fO\left(\dfrac{\bkappa\ln(\Delta/\epsilon)}{\sqrt{\gamma}}\right)$ 
    & $\fO\bigg(n\bkappa\ln\bigg(\dfrac{\Delta}{\epsilon}\bigg)\bigg)$ 
    & \citet{bai2024complexity}  \\\addlinespace 
    \DEAREST 
    & $\tilde\fO\left(\dfrac{\kappa\ln(\Delta/\epsilon)}{\sqrt{\gamma}}\right)$ 
    & $\tilde\fO\bigg(\bigg(n + \kappa + \min\bigg\{n\kappa, \sqrt{\dfrac{n}{m}}\,\bar\kappa\bigg\}\bigg) \ln\bigg(\dfrac{\Delta}{\epsilon}\bigg)\bigg)$ 
    & Corollary \ref{cor:PL-1}  \\    
    \addlinespace \hline  \addlinespace
    \!Lower Bounds
    & $\Omega\left(\dfrac{\kappa\ln(\Delta/\epsilon)}{\sqrt{\gamma}}\right)$ 
    & $\Omega\bigg(n+\bigg(\kappa+\min\bigg\{n\kappa, \sqrt{\dfrac{n}{m}}\,  \bkappa\bigg\}\bigg)\ln\bigg(\dfrac{\Delta}{\epsilon}\bigg)\bigg)$ 
    & This work  \\  
    \addlinespace \hline 
\end{tabular} 
\end{table}

\begin{table}[t]
\caption{We present the number of LIFO calls for finding $\epsilon$-suboptimal solutions under the PL condition, where $\Delta=f(\bx^0)-f^*$ and $\vx^0\in\BR^d$ is the initial point of the algorithm.}\label{table:pl-2}
\begin{tabular}{ccccc}
    \hline
    Methods & \#LIFO &   Reference
    \\\hline\hline\addlinespace
    \DGDGT 
    & $\fO\left(mn\kappa\ln\left(\dfrac{\Delta}{\epsilon}\right)\right)$ 
    & \citet{bai2024complexity}  \\ \addlinespace   
    \DRONE 
    & $\fO\left(\big(mn+\sqrt{m}n\bkappa\big)\ln\left(\dfrac{\Delta}{\epsilon}\right)\right)$ 
    & \citet{bai2024complexity}  \\ \addlinespace 
    \DEAREST 
    & $\fO\left(\left(mn + \min\left\{mn\kappa, \sqrt{mn}\,\bar\kappa\right\}\right)\ln\left(\dfrac{\Delta}{\epsilon}\right)\right)$ & Corollary \ref{cor:PL-1}  \\\addlinespace \hline  \addlinespace
    Lower Bounds
    & $\Omega\left(mn + \min\left\{mn\kappa, \sqrt{mn}\,\bar\kappa\right\}\ln\left(\dfrac{\Delta}{\epsilon}\right)\right)$
    & This work  \\  
    \addlinespace \hline 
\end{tabular} 
\end{table}

\subsection{The Complexity Analysis under the PL Condition}\label{sec:PL-upper}

The complexity analysis for \DEAREST~(Algorithm \ref{alg:DEAREST}) under the PL condition also distinguish the between~$L$ and $\bL$ ($\kappa$ and $\bkappa$). 
We formally present the results as follows.

\begin{theorem}\label{thm:PL-1}
Under Assumptions \ref{asm:lower-bounded}--\ref{asm:smooth-mean-squared} and \ref{asm:PL}--\ref{asm:W} with $\mu\!<\!L\!\leq\!\bar L$, we run Algorithm \ref{alg:DEAREST} with
\begin{align*}
\small\begin{split}    
    & \eta = \frac{1}{8L}, \qquad
    b = \left\lceil\frac{\bL\min\left\{\bkappa, \sqrt{mn}\,\right\}}{L}\right\rceil, \qquad
    p = \min \left\{\frac{\bL}{L}\max\left\{\frac{1}{\bkappa}, \frac{1}{\sqrt{mn}}\right\},1\right\}, 
     \\
    & T = \left\lceil16\kappa\ln\left(\frac{2\Delta}{\epsilon}+1\right)\right\rceil,~~ K=\fO\left(\frac{\ln(mn\kappa\bL)}{\sqrt{\gamma}}\right),
    ~~\text{and}~~
    \hK=\fO\left(\frac{\ln(mn\bL/(L\epsilon))}{\sqrt{\gamma}}\right),
\end{split}    
\end{align*}
where $\Delta=f(\bx^0)-f^*$. 
Then the output satisfies $\BE[f(\vx^T_i)-f^*] \leq \epsilon$ for all $i\in[m]$.
\end{theorem}

\begin{proof}
\textbf{Part I:} We first consider the case of $\bL<\sqrt{mn}L$. 
Following the derivation of equation (\ref{eq:recursion-Phi-0}) in the analysis of general case and the PL condition (\ref{eq:PL-condition}), we have
\begin{align}\label{eq:PL-decay}
\small\begin{split}
  & \BE[\Phi^{t+1}] \\
\leq & \BE\left[f(\bx^t) - f^* - \frac{\eta}{2}\Norm{\nabla f(\bx^t)}^2\right]  + \left(\eta + \frac{2(1-p)\eta}{p}\right)\BE[U^t]   \\
& + \left(\frac{(1-p)\eta}{m^3n^3bp} + 528\rho^2m^3n^3\bar L^2\eta^3\right) \BE[V^t] \\
& +  \left(n \bL^2\eta   + \frac{18(1-p)m^2n^3\bar L^2\eta}{p} + \frac{36(1-p)\bL^2 \eta}{mbp} + \frac{264\rho^2(27m^3n^3\bL^2\eta^2+2)m^2n^3\bar L^2\eta}{p}\right) \BE[C^t] \\
& - \left(\frac{1}{2\eta}-\frac{L}{2}
- \frac{15(1-p)\bL^2\eta}{bp}
- \frac{2376\rho^2m^6n^6\bar L^4\eta^3}{p}\right)\BE\Norm{\bx^{t+1}-\bx^t}^2 \\
\leq & \BE\left[\left(1-\frac{\eta\mu}{2}\right)(f(\bx^t) - f^*) - \frac{\eta}{4}\Norm{\nabla f(\bx^t)}^2\right]  + \left(\eta + \frac{2(1-p)\eta}{p}\right)\BE[U^t]   \\
& + \left(\frac{(1-8\eta\mu)\eta}{m^3n^3bp} + 528\rho^2m^3n^3\bar L^2\eta^3\right) \BE[V^t] \\
& +  \left(n \bL^2\eta   + \frac{18(1-8\eta\mu)m^2n^3\bar L^2\eta}{p} + \frac{36(1-8\eta\mu)\bL^2 \eta}{mbp} + \frac{264\rho^2(27m^3n^3\bL^2\eta^2+2)m^2n^3\bar L^2\eta}{p}\right) \BE[C^t] \\
& - \left(\frac{1}{2\eta}-\frac{L}{2}
- \frac{15(1-p)\bL^2\eta}{bp}
- \frac{2376\rho^2m^6n^6\bar L^4\eta^3}{p}\right)\BE\Norm{\bx^{t+1}-\bx^t}^2 \\
\leq &  \BE\left[\left(1-\frac{\mu\eta}{2}\right)\Phi^t - \frac{\eta}{4}\Norm{\nabla f(\bx^{t})}^2  - \frac{L^2\eta}{4}\Norm{\bx^{t+1}-\bx^t}^2\right]
\leq  \BE\left[\left(1-\frac{\eta\mu}{2}\right)\Phi^t \right],
\end{split}
\end{align}
where the second inequality uses the fact $p\geq 8\eta\mu$ from the settings of $\eta=1/(8L)$ and $p = (\bL/L)\max\left\{\mu/\bL, 1/\sqrt{mn}\right\}$; the last step is based on the parameter settings in the statement of Theorem \ref{thm:PL-1} and taking
\begin{align*}  
\small\begin{split}    
 K= \left\lceil\frac{2+\sqrt{2}}{2\sqrt{\gamma}}\ln\left(14\max\left\{\frac{44m^6n^6(\bL^2+L^2)\bL^2}{5L^3\mu},\,\frac{33(27m^3n^3\bL^2+128L^2)}{2L(308L+183\mu)},\,\frac{297m^{6.5}n^{6.5}\bar L^3}{51L^3}\right\}\right)\right\rceil.
\end{split} 
\end{align*}
We split the function value gap at the $i$-th node as
\begin{align}\label{eq:PL-gap}
  \BE[f(\vx^T_i) - f^*]  
= \BE[f(\bx^T) - f^*]  + \BE[f(\vx^T_i) - f(\bx^T)].
\end{align}
For the first term on the right-hand side of equation (\ref{eq:PL-gap}), we have
\begin{align}\label{eq:PL-gap-1}
\small\begin{split}    
& \BE[f(\bx^T) - f^*] \\
\overset{~(\ref{eq:Lyapunov})~}\leq & \BE[\Phi^T-132\bL^2\eta C^T] \\
\overset{(\ref{eq:PL-decay})}\leq & \BE\left[\left(1-\frac{\mu\eta}{2}\right)\Phi^{T-1} - \frac{\eta}{4}\Norm{\nabla f(\bx^{T-1})}^2 - \frac{L^2\eta}{4}\Norm{\bx^{T}-\bx^{T-1}}^2 - 132\bL^2\eta C^T\right].
\end{split}
\end{align}
For the second term on the right-hand side of equation (\ref{eq:PL-gap}), we have
\begin{align}\label{eq:PL-gap-2}
\begin{split}   
 & \BE[f(\vx^T_i) - f(\bx^T)]  \\
\leq & \BE\left[\inner{\nabla f(\bx^T)}{\vx^T_i - \bx^T} + \frac{L}{2}\Norm{\vx^T_i - \bx^T}^2\right]  \\
\leq &  \BE\left[\frac{\eta}{8}\Norm{\nabla f(\bx^T)}^2 + \frac{2}{\eta}\Norm{\vx^T_i - \bx^T}^2 + \frac{L}{2}\Norm{\vx^T_i - \bx^T}^2\right]  \\
\leq &  \BE\left[\frac{\eta}{4}\Norm{\nabla f(\bx^{T-1})}^2 + \frac{\eta}{4}\Norm{\nabla f(\bx^T)-\nabla f(\bx^{T-1})}^2 + \left(\frac{2}{\eta} + \frac{L}{2}\right)C_T\right]  \\
\leq &  \BE\left[\frac{\eta}{4}\Norm{\nabla f(\bx^{T-1})}^2 + \frac{L^2\eta}{4}\Norm{\bx^T-\bx^{T-1}}^2 + \left(\frac{2}{\eta} + \frac{L}{2}\right)C_T\right] 
\end{split}
\end{align}
where the first and the last inequalities are based on Assumption \ref{asm:smooth-global}; 
the second inequality is based on the Cauchy--Schwarz inequality;
the third inequality is based on the Young's inequality and the definition of $C_T$ as equation (\ref{eq:def-C}).

Combining equations (\ref{eq:PL-gap}), (\ref{eq:PL-gap-1}) and (\ref{eq:PL-gap-2}), we have
\begin{align*}  
 & \BE[f(\vx^T_i) - f^*]   \\
\leq & \BE\left[\left(1-\frac{\mu\eta}{2}\right)\Phi^{T-1}   - \left(132\bL^2\eta - \frac{2}{\eta} - \frac{L}{2}\right)C_T\right] \\
\leq & \BE\left[\left(1-\frac{\mu\eta}{2}\right)^T\Phi^0\right], 
\end{align*}
where the last step is based on the setting of $\eta=1/(8L)$, the assumption $\bL\geq L$, and equation (\ref{eq:PL-decay}).

We also have
\begin{align}\label{ieq:Phi0-PL}
\begin{split}
 \Phi^0  
\overset{(\ref{eq:Lyapunov})}{=} & f(\bx^0)  - f^*  + \frac{2\eta}{p}U_{0} + \frac{\eta}{m^3n^3bp} V_{0} + \frac{132m^2n^3\bar L^2\eta}{p}C_{0} \\
\,\,=\,\, & f(\bx^0) - f^* + \frac{132m^2n^3\bar L^2\eta}{p} C^0 \\
\overset{(\ref{eq:def-C})}{\leq} & f(\bx^0) - f^* + \frac{132m^2n^3\bar L^2\eta^3}{p} \Norm{\mS^0 - \vone\bs^0}^2 \\
\,\,\leq\,\, & 2(f(\bx^0) - f^*)+\epsilon,
\end{split}
\end{align}
where the last step is based on the setting
\begin{align*}
\small\begin{split}
\hK= \left\lceil\frac{2+\sqrt{2}}{2\sqrt{\gamma}}\ln\left(1+\frac{231m^2n^3\bar L^2}{64L^3} \max\left\{\bkappa, {\sqrt{mn}}\right\}  \sum_{i=1}^m\frac{\Norm{\nabla f_i(\bx^0)-\nabla f(\bx^0)}^2}{f(\bx^0) - f^*+ \epsilon}\right)\right\rceil.
\end{split}
\end{align*}
Therefore, we can achieve
\begin{align}\label{eq:PL-convergence-end}
    \BE[f(\vx^T_i) - f^*] \leq \epsilon
\end{align}
for all $i\in[m]$ by taking
\begin{align*}
    T = \left\lceil\frac{16L}{\mu}\ln\left(\frac{2(f(\bx^0) - f^*)}{\epsilon}\right)\right\rceil.
\end{align*}

\textbf{Part II}: We then consider the case of $\bL\geq\sqrt{mn}L$.
Note that the setting $p=1$ leads to  Algorithm \ref{alg:DEAREST} always holds 
\begin{align*}
    \vg^t_i = \nabla f_i(\vx^t_i),
\end{align*}
which implies $U^t=V^t=0$ for all $t$ and the Lyapunov function can be simplified to 
\begin{align*}
\Phi^t \triangleq & f(\bx^{t}) - f^* + 132m^2n^3\bar L^2\eta C^{t}.
\end{align*}
Following the derivation of equations (\ref{eq:recursion-Phi-0}) and (\ref{eq:PL-decay}), we have
\begin{align}
\begin{split}
  & \BE[\Phi^{t+1}] \\
= & \BE\left[f(\bx^{t+1}) - f^* + 132m^2n^3\bar L^2\eta C^{t+1}\right]  \\
\leq & \BE[(1-\eta\mu)(f(\bx^t) - f^*)] +  \left(n \bL^2\eta  + 264\rho^2(27m^3n^3\bL^2\eta^2+2)m^2n^3\bar L^2\eta\right) \BE[C^t] \\
& - \left(\frac{1}{2\eta}-\frac{L}{2}
- 2376\rho^2m^6n^6\bar L^4\eta^3\right)\BE\Norm{\bx^{t+1}-\bx^t}^2 \\
\leq & \BE\left[(1-\eta\mu)\Phi^t - (1-\eta\mu)m^2n^3\bar L^2\eta C^t\right]. 
\end{split}
\end{align}
We then follow the derivation of equations (\ref{eq:PL-decay})-(\ref{eq:PL-convergence-end}) to achieve
\begin{align}
\BE[f(\vx^T_i)-f^*] \leq \epsilon.
\end{align}   
for all $i\in[m]$.
\end{proof}

Similar to the analysis of Corollary \ref{cor:PL-1}, we can obtain the upper complexity bounds for \DEAREST~(Algorithm \ref{alg:DEAREST}) under the PL condition as follows. \vskip0.1cm

\begin{corollary}\label{cor:PL-1}    
Under the assumptions and settings of Theorem \ref{thm:PL-1}, Algorithm~\ref{alg:DEAREST} can find an $\epsilon$-suboptimal solution at every agent with the communication rounds of 
\begin{align*}
    \tilde\fO\left(\frac{\kappa}{\sqrt{\gamma}}\ln\left(\frac{\Delta}{\epsilon}\right)\right),
\end{align*}
expected LIFO complexity of
\begin{align*}
\fO\left(\left(mn + \min\left\{mn\kappa, \sqrt{mn}\bar\kappa\right\}\right)\ln\left(\frac{\Delta}{\epsilon}\right)\right),     
\end{align*}
and the expected computation rounds of
\begin{align*}
\tilde\fO\left(\left(n+\kappa+\min\left\{n\kappa, \sqrt{\frac{n}{m}}\bar \kappa\right\}\right)\ln\left(\frac{\Delta}{\epsilon}\right)\right).
\end{align*}
\end{corollary}
\begin{proof}  
The communication rounds can be upper bounded by
\begin{align*}
    \hat K + KT = \fO\left(\frac{\ln(mn\bL/(L\epsilon))}{\sqrt{\gamma}} + \frac{\ln(mn\kappa\bL)}{\sqrt{\gamma}}\cdot \kappa\ln\left(\frac{1}{\epsilon}\right)\right) = \tilde\fO\left(\frac{\kappa\ln(1/\epsilon)}{\sqrt{\gamma}}\right),
\end{align*}
where we use the settings of $T$, $K$, and $\hat K$ in Theorem \ref{thm:PL-1}.

We then consider the LIFO complexity.
In the case of $\bL>\sqrt{mn}L$, the number of LIFO calls can be upper bounded by
\begin{align*}
    mnT = \fO\left(mn\kappa\ln\left(\frac{1}{\epsilon}\right)\right),
\end{align*}
where we use the setting of $T$ in Theorem \ref{thm:PL-1}.
In the case of $\bL\leq\sqrt{mn}L$, the number of LIFO calls can be upper bounded by
\begin{align*}
 &   mn + (pn + (1-p)b)T \leq mn + (pmn + b)T \\
= & \fO\left(mn + \left(\frac{\bL}{L}\max\left\{\frac{1}{\bkappa}, \frac{1}{\sqrt{mn}}\right\}\cdot mn + \left\lceil\frac{\bL}{L}\min\left\{\bkappa, \sqrt{mn}\,\right\}\right\rceil \right)\cdot\kappa\ln\left(\frac{1}{\epsilon}\right)\right) \\
= & \fO\left(\left(mn + \sqrt{mn}\,\bar\kappa\right)\ln\left(\frac{1}{\epsilon}\right)\right),
\end{align*}
where we use the settings of $b$, $p$, and $T$ in Theorem \ref{thm:PL-1}.
Therefore, the overall LIFO complexity is no more than
\begin{align*}
\fO\left(\left(mn + \min\left\{mn\kappa, \sqrt{mn}\bar\kappa\right\}\right)\ln\left(\frac{1}{\epsilon}\right)\right),     
\end{align*}

We finally consider the computation rounds.
In the case of $\bL>\sqrt{mn}L$, the number of computation round is no more than
\begin{align}\label{eq:computation-round-simple-PL}
    n + nT = \fO\left(n + \frac{L}{\mu}\ln\left(\frac{1}{\epsilon}\right)\cdot n\right) \leq \fO\left(n\kappa \ln\left(\frac{1}{\epsilon}\right)\right).
\end{align}
In the case of~$\bL\leq\sqrt{mn}L$, we following the analysis in Section \ref{sec:computation-rounds} which implies
the overall expected number of computation rounds is no more than
\begin{align*}
\small\begin{split}
& n + \left((1-p)\BE\left[2\sum_{t=1}^T \max_{i\in[m]} Y_i^t \right] + pn\right)T \\
\leq & \fO\left(n + (pn+(1-p)(bnq + (\ln mb)^2))T\right) \\
\leq & \fO\left(n + (pn + bnq + (\ln mb)^2)T\right) \\
\leq & \fO\left(n + \left(\frac{\bL}{L}\max\left\{\frac{1}{\bkappa}, \frac{1}{\sqrt{mn}}\right\}\cdot n + \frac{\bL}{L}\min\bigg\{\bkappa, \sqrt{mn}\bigg\}\cdot\frac{n}{mn} + \left(\min\left\{\bkappa, \sqrt{mn}\ln \frac{m\bL}{L}\right\}\right)^2\right)T\right) \\
\leq & \tilde\fO\left(\left(n+\kappa+\sqrt{\frac{n}{m}}\,  \bar \kappa\right)\ln\left(\frac{1}{\epsilon}\right)\right), 
\end{split}
\end{align*}
where we use $q=1/(mn)$ and the settings of $b$, $p$, and $T$ in Theorem \ref{thm:PL-1}.
Combining the upper bound (\ref{eq:computation-round-simple-PL}) in the case of $\bL\leq\sqrt{mn}L$, the expected number of computation rounds of our method can be upper bounded by
\begin{align*}
& \min\left\{\fO\left(n\kappa \ln\left(\frac{1}{\epsilon}\right)\right), \tilde\fO\left(\left(n+\kappa+\sqrt{\frac{n}{m}}\,  \bar \kappa\right)\ln\left(\frac{1}{\epsilon}\right)\right) \right\} \\
= & \tilde\fO\left(\left(n+\kappa+\min\left\{n\kappa, \sqrt{\frac{n}{m}}\,  \bar \kappa\right\}\right)\ln\left(\frac{1}{\epsilon}\right)\right).
\end{align*}
\end{proof}

\subsection{The Lower Bounds under the PL Condition}\label{sec:PL-lower}

This section provides the lower bounds for decentralized finite-sum optimization under the PL condition. 
Compared with the construction of \citet{yue2023lower}, we additionally consider the communication in the network, the finite-sum structure in objective,  and the difference between global and mean-squared smoothness.
We present the results of communication

We introduce the functions $\psi_\theta:\BR\to\BR$, $q_{T,t}:\BR^{Tt}\to\BR$ and $g_{T,t}:\BR^{Tt}\to\BR$ provided by \citet{yue2023lower}, that is
\begin{equation}\label{eq:g}
\begin{aligned}   
& \psi_{\theta}(x)
= \begin{cases}
	\frac{1}{2}x^{2}, 
	& x\leq\frac{31}{32}\theta, \\[0.3em]
	\frac{1}{2}x^{2}-16(x-\frac{31}{32}\theta)^{2},& {\frac{31}{32}\theta<x\leq \theta,}\\[0.3em]
	\frac{1}{2}x^{2}-\frac{1}{32}\theta^2+16(x-\frac{33}{32}\theta)^{2}, &      
	{\theta<x\leq \frac{33}{32}\theta,}\\[0.3em]
	\frac{1}{2}x^{2}-\frac{1}{32}\theta^2,      
	& {x>\frac{33}{32}\theta,}
\end{cases} \\[0.15cm]
& q_{T,t}(\vx)=\frac{1}{2}\sum_{i=0}^{t-1}\Big(\Big(\frac{7}{8}x_{iT} -x_{iT+1}\Big)^2+\sum_{j=1}^{T-1}(x_{iT+j+1}-x_{iT+j})^2\Big), \\
& \text{and}\quad g_{T,t}(\vx)=q_{T,t}(\vb-\vx)+\sum_{i=1}^{Tt}\psi_{b_i}(b_i-x_i), 
\end{aligned}
\end{equation}
where we define $x_0=0$ and $\vb\in \mathbb{R}^{Tt}$ with $b_{kT+\tau}=({7}/{8})^k$ for~$k\in \{0\}\cup[T-1]$ and~$\tau\in[T]$.
We can verify that 
\begin{align*}
g_{T,t}^*\triangleq \inf_{\vy\in\BR^{Tt}}g_{T,t}(\vy)=0.    
\end{align*}
The following lemma shows the function $g_{T,t}$ holds the zero-chain property \cite{nesterov2018lectures,carmon2020first} and describes its smoothness, PL parameter and optimal function value gap, which results the tight lower bound of full-batch first-order methods \citep{yue2023lower}.
\begin{lemma}[{\citep[Section 4]{yue2023lower}}]\label{cor:g} 
The function $g_{T,t}:\BR^{Tt}\to\BR$ holds that:
\begin{enumerate}[itemsep=3pt,topsep=2pt]
\item[(a)]
For any $\vx\in\BR^{Tt}$, it holds $\prog(\nabla g_{T,t}(\vx)) \leq \prog(\vx)+1$.
\item[(b)] The function $g_{T,t}$ is 37-smooth.
\item[(c)] The function $g_{T,t}$ is $1/(aT)$-PL with $a=19708$.  
\item[(d)] The function $g_{T,t}$ satisfies that $g_{T,t}(\mathbf{0})-g_{T,t}^*\leq3T$. 
\item[(e)] For any  $\delta<0.01$, $t=2\lfloor\log_{{8}/{7}}{2}/{(3\delta)}\rfloor$, and $\vx\in\BR^{Tt}$ satisfying that \\ ${\rm supp}(\vx)\subseteq\left\{1,2,\cdots,Tt/2\right\}$, it holds that $g_{T,t}(\vx)-g_{T,t}^*> 3T\delta$. 
\end{enumerate}
\end{lemma}
We can decompose the function $g_{T,t}:\BR^{Tt}\to\BR$ defined in equation (\ref{eq:g}) by introducing the functions $q_1:\BR^{Tt}\to\BR$, $q_2:\BR^{Tt}\to\BR$ and $r:\BR^{Tt}\to\BR$ as 
\begin{align}
&	q_1(\vx)=\frac{1}{2}\sum_{i=1}^{Tt/2}(x_{2i-1}-x_{2i})^2, \label{eq:q1} \\
&	q_2(\vx)=\frac{1}{2}\sum_{i=0}^{t-1}\left[\Big(\frac{7}{8}x_{iT}-x_{iT+1}\Big)^2+\sum_{j=iT/2+1}^{(i+1)T/2-1} (x_{2j}-x_{2j+1})^2\right], \label{eq:q2} \\
& r(\vx)=\sum_{i=1}^{Tt}\psi_{b_i}(b_i-x_i), \label{eq:r}
\end{align}
where we suppose $T$ is even and let $x_0=0$.
Then we can verify that the function~$g_{T,t}(\cdot)$ can be written as
\begin{align*}
    g_{T,t}(\vx)=q_1(\vb-\vx)+q_2(\vb-\vx)+r(\vx).
\end{align*}

\begin{lemma}\label{lem:function-pl-decompose}
The functions $q_1$ and $q_2$ are 2-smooth and the function $r$ is 33-smooth. 
\end{lemma}
\begin{proof}
    The smoothness of $r$ can be achieved from the proof of Lemma 4 provided by \citet{yue2023lower}.
    Since the functions $q_1$ and $q_2$ are quadratic and hold $q_1(\vx)\leq\vx^\top\vx$ and~$q_2(\vx)\leq\vx^\top\vx$ for all $\vx\in\BR^{Tt}$, they are $2$-smooth.  
\end{proof}

We now provide the lower bounds on the communication rounds under the PL condition in the following theorem.

\begin{theorem}
    For any $L>0$, $\bar L>0$, $\mu>0$, $m\geq2$, $n>0$, $\Delta>0$, $\gamma\in[1-\cos(\pi/m),1]$, and $\epsilon>0$ with $\bL\geq L\geq78a\mu$, $\epsilon<0.01\Delta$, and $a=19708$, there exist matrix $\mW\in\BR^{m\times m}$ with $1-\lambda_2(\mW)=\gamma$ and $\bar L$-mean-squared smooth functions $\{f_{i,j}:\mathbb{R}^{d}\rightarrow\mathbb{R}\}$ with $d=\fO(\Delta L\epsilon^{-2})$ such that $f=\frac{1}{mn}\sum_{i=1}^{m}\sum_{j=1}^{n}f_{i,j}$ is $L$-smooth and $\mu$-PL with $f(\mathbf{0})-\inf_{\vy\in\BR^d}f(\vy)\leq\Delta$. 
    In order to find an $\epsilon$-suboptimal solution of problem $\min_{\vx\in\BR^d}f(\vx)$, any LIFO algorithm needs at least $\Omega\big(\kappa/\sqrt{\gamma}\ln(1/\epsilon)\big)$ communication rounds.
\end{theorem}

\begin{proof}
According to Lemma \ref{lem:big gamma}, there exists a ring-lattice graph $\mathcal{G}=\{\fV,\fE\}$ with the associated mixing matrix $\mW\in\BR^{m\times m}$ and $\fV_1,\fV_2\subseteq\fV$ such that $\lvert\fV\rvert=m$, $1-\lambda_2(\mW)=\gamma$, $\lvert\fV_1\rvert\geq m/3$, $\lvert\fV_2\rvert\geq m/3$ and the distance between $\fV_1$ and $\fV_2$ satisfies ${\rm dist}(\fV_1,\fV_2)=\Omega(1/\sqrt{\gamma})$.

We consider the function $g_{T,t}$ defined in equation (\ref{eq:g}) with 
\begin{align}\label{eq:comm-T-pl}
T=2\left\lfloor\frac{\kappa}{78a}\right\rfloor\quad\text{and}\quad t=2\left\lfloor \log_{\frac{8}{7}}\frac{2\Delta}{3\epsilon}\right\rfloor,\quad\text{where}~~L\geq78a\mu\quad\text{and}\quad \epsilon<0.01\Delta.    
\end{align}

According to Lemma \ref{lem:scale} with 
\begin{align}\label{eq:comm-alpha-beta-pl}
g(\vx)=g_{T,t}(\vx), \qquad \alpha=\frac{\Delta}{3T} \qquad \text{and} \qquad 
\beta=\sqrt{\frac{LT}{13\Delta}},    
\end{align}
we can conclude the function $\hat g(\vx)=\alpha g_{T,t}(\beta\vx)$ is $37\alpha\beta^2$-smooth, and it holds 
\begin{align}\label{eq:comm-lower-hat-g-pl}
\hat g(\mathbf{0})-\inf_{\vy\in\BR^T}\hat g(\vy) = \alpha \left(g_{T,t}(\mathbf{0})-\inf_{\vy\in\BR^T} g_{T,t}(\vy)\right) \leq 3\alpha T, 
\end{align}
where the inequality is based on Lemma \ref{cor:g}(d).

We then construct the hard instance based on the decomposition 
\begin{align*}   
g_{T,t}(\vx)=q_1(\vb-\vx)+q_2(\vb-\vx)+r(\vx), 
\end{align*}
where $q_1$, $q_2$ and $r$ are defined in equations (\ref{eq:q1}), (\ref{eq:q2}), and (\ref{eq:r}).
Specifically, we take $f_{i,j}$ as 
\begin{align}\label{eq:comm-instance-pl}
f_{i,j}(\vx)=
\begin{cases}
\alpha r(\beta\vx)+\dfrac{\alpha m}{\vert \fV_1\vert}q_1(\vb-\beta\vx),   &{i\in \fV_1}, \\[0.4cm]
\alpha r(\beta\vx)+\dfrac{\alpha m}{\vert \fV_2\vert}q_2(\vb-\beta\vx), &{i\in \fV_2}, \\[0.4cm]
\alpha r(\beta\vx),   &{\text{otherwise}}.
\end{cases}
\end{align}
According to Lemma \ref{lem:scale}, Lemma \ref{lem:function-pl-decompose} and equation (\ref{eq:comm-instance-pl}), each component function $f_{i,j}$ is $\alpha\beta^2\big(33+2m/\min\{\lvert\fV_1\rvert,\lvert\fV_2\rvert\}\big)$-smooth.
Therefore, the function set $\{f_{i,j}\}_{i,j=1}^{m,n}$ is $\alpha\beta^2\big(33+2m/\min\{\lvert\fV_1\rvert,\lvert\fV_2\rvert\}\big)$-mean-squared smooth.

The definition of $f_{i,j}$ in equation (\ref{eq:comm-instance-pl}) also means
\begin{align*}
f(\vx)=\frac{1}{mn}\sum_{i=1}^m\sum_{j=1}^nf_{i,j}(\vx)=\hat g(\vx).
\end{align*}
The settings of $\alpha$, $\beta$, $T$, $t$, $\lvert\fV_1\rvert$ and $\lvert\fV_2\rvert$ in equations (\ref{eq:comm-T-pl}) and (\ref{eq:comm-lower-hat-g-pl}) imply
\begin{align*}
	& 37\alpha\beta^2=\frac{37L}{39}\leq L, \qquad \frac{\alpha\beta^2}{aT}\geq\mu, \qquad 
    3\alpha T\leq\Delta, \\
    &  \text{and} \qquad
    \alpha\beta^2\left(33+\frac{2m}{\min\{\lvert\fV_1\rvert,\lvert\fV_2\rvert\}}\right) \leq L \leq \bL.    
\end{align*}
Therefore, the global objective $f=\frac{1}{mn}\sum_{i=1}^m\sum_{j=1}^nf_{i,j}$ is $L$-smooth, $\mu$-PL and satisfies $f(\mathbf{0})-f^*\leq\Delta$ and the function sets $\{f_{i,j}\}_{i=1,j}^{m,n}$ are $\bar L$-mean-squared smooth, which satisfies our requirements.

Now we show that any LIFO algorithm require at least 
\begin{align}
\frac{Tt}{2}{\rm dist}(\fV_1,\fV_2)=\Omega\left(\frac{\kappa}{\sqrt{\gamma}}\ln(\frac{1}{\epsilon})\right)    
\end{align}
communication rounds to to achieve an $\epsilon$-suboptimal solution $\hat\vx$ of $f(\cdot)$.

We take $\delta=\epsilon/\Delta<0.01$. According to Lemma \ref{cor:g}(e), we have
\begin{align*}
   \|\nabla f(\vx)\| >3\alpha T\delta=\epsilon.
\end{align*}
for all $\vx=[x_1,\dots,x_{Tt}]^\top\in\BR^d$ such that $\prog(\vx)\leq Tt/2$.
We let 
\begin{align}\label{dfn:nnz1}
\begin{split}    
    & \widehat{\rm nnz}(s,i) \triangleq
     \max\big\{i\in\{0,1,\dots,Tt\}: \\
    & \qquad\qquad\qquad\qquad \text{there exists $\vy=[y_1,\dots,y_{Tt}]^\top\in\mathcal{M}_i^s$ such that $y_i\neq0$}\big\},
\end{split}    
\end{align}
where we denote $y_0=1$.

Consider the definitions of $f_{i,j}$, $q_1$, $q_2$ and $r$ in equations (\ref{eq:q1}), (\ref{eq:q2}), (\ref{eq:r})
and~(\ref{eq:comm-instance-pl}).
Notice that the fact $\psi_{b_i}'(b_i)=0$ implies the partial derivative of $r(\vx)$ with respect to~$x_k$ is zero when $x_k=0$.
This means the computation of any LIFO algorithm holds:
\begin{enumerate}
    \item If $i\in\fV_1$ and $\widehat{\rm nnz}(s,i)$ is odd, one step of local computation can increase at most one dimension for memory of node~$i$.
    \item If $i\in\fV_2$ and $\widehat{\rm nnz}(s,i)$ is even, one step of local computation can increase at most one dimension for memory of node $i$. 
    \item Otherwise, one step of local computation cannot increase the dimension for memory of node~$i$.
\end{enumerate}
In summary, we have
\begin{align}\label{eq:nnz}
    \widehat{\rm nnz}(s+1,i) \leq \begin{cases}
        \widehat{\rm nnz}(s,i)+1,  & \text{if}~~i\in\fV_1,\  \widehat{\rm nnz}(s,i)\equiv 1 \pmod{2}, \\
        \widehat{\rm nnz}(s,i)+1, & \text{if}~~i\in\fV_2,\  \widehat{\rm nnz}(s,i)\equiv 0 \pmod{2}, \\
        \widehat{\rm nnz}(s,i), &\text{otherwise}.
    \end{cases}
\end{align}

We consider the cost to reach the second coordinate from the initial status that~$\fM_i^0=\{\vzero\}$ for all $i\in[m]$.
According to equation (\ref{eq:nnz}), we need to let a node in~$\fV_1$ reach the first coordinate, which requires at least one local computation step on some node in $\fV_2$ first. 
According to the definition of LIFO algorithm (Definition~\ref{dfn:LIFO}), one must perform at least ${\rm dist}(\fV_1,\fV_2)$ communication rounds for a node in $\fV_1$ to receive the information of the first coordinate from some node in $\fV_2$. 
After above rounds, we can perform at least 1 computation on nodes in $\fV_1$ to reach the second coordinate.
In summary, to reach the second coordinate requires at least~${\rm dist}(\fV_1,\fV_2)$ communication rounds.

Similarly, to reach the $k$-th coordinate, any LIFO algorithm must perform at least~$(k-1){\rm dist}(\fV_1,\fV_2)$ communication rounds.
Thus, to attain the condition 
\begin{align*}
\prog(\vx)>\frac{Tt}{2},     
\end{align*}
one needs at least 
\begin{align*}
\frac{Tt}{2}{\rm dist}(\fV_1,\fV_2)=\Omega\left(\frac{\kappa}{\sqrt{\gamma}}\ln\left(\frac{1}{\epsilon}\right)\right) 
\end{align*}
communications rounds.
\end{proof}

We then provide the lower bound on the IFO complexity for the finite-sum optimization problem with $N$ individual functions under the PL condition, which directly implies the corresponding lower bound on the LIFO complexity in decentralized setting by taking $N=mn$.

\begin{theorem}\label{thm:LFO-pl}
        For any $\bar L>0$, $L>0$, $\mu>0$, $N>0$, $\Delta>0$, and $\epsilon>0$ with $\epsilon<0.005\Delta$ and $\bL\geq L>\mu$, 
        there exists $\bar L$-mean-squared smooth functions $\{\tf_i:\BR^d\to\BR\}_{i=1}^N$ such that the function $\tf=\frac{1}{N}\sum_{i=1}^{N}\tf_i$ is $L$-smooth, $\mu$-PL and holds $\tf(\vx^0)-\inf_{\vy\in\BR^d}\tf(\vy)\leq \Delta$. 
In order to find an $\epsilon$-suboptimal solution of problem $\min_{\vx\in\BR^d}\tf(\vx)$, any IFO algorithm (defined in Appendix \ref{appendix:IFO}) needs at least $\Omega\big(N+\min\{N\kappa ,\sqrt{N}\bar\kappa\}\ln(1/\epsilon)\big)$ IFO calls.
\end{theorem}
\begin{proof}
We first consider the case of $\min\{\bar\kappa,\kappa\sqrt{N}\}\geq37a\sqrt{N}$.
We apply the function~$g_{T,t}$ defined in equation (\ref{eq:g}) with 
\begin{align}\label{eq:ifo-T-pl}
T=\left\lfloor\frac{\min\{\bar\kappa,\kappa\sqrt{N}\}}{37a\sqrt{N}}\right\rfloor\quad\text{and}\quad t=2\left\lfloor \log_{\frac{8}{7}}\frac{\Delta}{3\epsilon}\right\rfloor,\quad\text{where}~~\epsilon<0.005\Delta.    
\end{align} 
According to Lemma \ref{lem:scale} with 
\begin{align}\label{eq:ifo-alpha-beta-pl}
g(\vx)=g_{T,t}, \qquad \alpha = \frac{\Delta}{3T} \qquad \text{and} \qquad 
\beta=\sqrt{\frac{3\min\{\sqrt{N}\bL,NL\}T}{37\Delta}},    
\end{align}
we can conclude the function $\hat g(\vx)=\alpha g_{T,t}(\beta\vx)$ is $37\alpha\beta^2$-smooth, and it holds 
\begin{align}\label{eq:lower-hat-g-pl}
\hat g(\mathbf{0})-\inf_{\vy\in\BR^T}\hat g(\vy) = \alpha \left(g_{T,t}(\mathbf{0})-\inf_{\vy\in\BR^T} g_{T,t}(\vy)\right) \leq 3\alpha T, 
\end{align}
where the inequality is based on Lemma \ref{cor:g}(d).

We construct the hard instance according to Lemma \ref{lem:orthogonal transformation} with $g(\vx)=\hat g(\vx)$, which results the functions
\begin{align}\label{eq:ifo-instance-1-pl}
     \tf_i(\vx)=\alpha g_{T,t}(\beta \mU^{(i)}\vx) 
     \qquad\text{and}\qquad
     \tf(\vx)=\frac{1}{N}\sum_{i=1}^{N}\tf_i(\vx) 
\end{align}
such that the functions $\{\tf_i\}_{i=1}^N$ are $37\alpha\beta^2/\sqrt{N}$-mean-squared smooth and the function $\tf$ is $37\alpha\beta^2/N$-smooth and $\alpha\beta^2/(aNT)$-PL with 
\begin{align*}
\tf(\vx^0)-\tf^*=\hat g(\mathbf{0})-\inf_{\vy\in\BR^T}\hat g(\vy)
\overset{(\ref{eq:lower-hat-g-pl})}{\leq} 3\alpha T,     
\end{align*}
where $\tf^*=\inf_{\vy\in\BR^T}\tf(\vy)$.

The settings of $\alpha$, $\beta$, $T$ and $t$ in equations (\ref{eq:ifo-T-pl}) and (\ref{eq:ifo-alpha-beta-pl}) imply
\begin{align*}
	\frac{37\alpha\beta^2}{\sqrt{N}}\overset{(\ref{eq:ifo-alpha-beta-pl})}{\leq}\bar L, \qquad \frac{37\alpha\beta^2}{N}\overset{(\ref{eq:ifo-alpha-beta-pl})}{\leq} L, \qquad\frac{\alpha\beta^2}{anT}\overset{(\ref{eq:ifo-T-pl})}{\geq}\mu\qquad\text{and}\qquad  
    3\alpha T \overset{(\ref{eq:ifo-alpha-beta-pl})}{\leq}\Delta.
\end{align*}
Therefore, the function set 
$\{\tf_i\}_{i=1}^N$ is $\bar L$-average smooth and the function $f$ is $L$-smooth and $\mu$-PL with $\tf(\vx^0)-\tf^*\leq\Delta$, which satisfies our requirements.

Now we show that any IFO algorithm require at least 
$\lfloor NTt/4\rfloor+1$ IFO calls to achieve an $\epsilon$-suboptimal solution $\hat\vx$ of the function $f$.
We take $\delta=2\epsilon/\Delta$. According to Lemmas \ref{lem:scale} and \ref{cor:g}(e), for all $\vx\in\BR^{NTt}$ with ${\rm supp}(\mU^{(j)}\vx)\subseteq\left\{1,2,\cdots,Tt/2\right\}$, we have
\begin{align}\label{eq:ifo-grad-upper-pl}
    f_i(\vx)- f_i^*= \alpha g_{T,t}(\beta \mU^{(i)}\vx)-\alpha g_{T,t}^*>3\alpha T\delta=2\epsilon.
\end{align} 
We consider the vector $\vx\in\BR^{NTt}$ which is achieved by an IFO algorithm with at most $\lfloor NTt/4\rfloor$ IFO calls.
Lemma \ref{cor:g}(a) implies such vector $\vx$ has at most $\lfloor NTt/4 \rfloor$ non-zero entries.
We partition $\vx\in\BR^{NTt}$ into $N$ vectors $\vy^{(1)},\dots,\vy^{(N)}\in\BR^{Tt}$ such that $\vy^{(j)}=\mU^{(j)}\vx\in\BR^{Tt}$,
then there at least $\lceil N/2\rceil$ vectors in $\{\vy^{(j)}\}_{j=1}^N$ such that each of them has at least $Tt/2$ zero entries.
Lemma \ref{cor:g}(a) means
there exists index set $\fI\subseteq [N]$ with $|\fI|\geq \lceil N/2\rceil$ such that each~$j\in\fI$ satisfies ${\rm supp}(\vy^{(j)})\subseteq\left\{1,2,\cdots,Tt/2\right\}$.
Therefore, we have
\begin{align*}
 f(\vx)- f^* 
= & \frac{1}{N}\sum_{i=1}^N f_i(\vx)- \alpha g_{T,t}^* \\
\geq & \frac{1}{N}\sum_{i\in\fI} (f_i(\vx)- \alpha g_{T,t}^*) \\
>& \frac{1}{N}\cdot\left\lceil \frac{N}{2} \right\rceil \cdot 2\epsilon
\geq\epsilon.
\end{align*}
Hence, any IFO algorithm requires at least 
\begin{align*}
 \left\lfloor \frac{NTt}{4}\right\rfloor+1=\Omega\left(\min\{\kappa N,\bar\kappa\sqrt{N}\}\ln\left(\frac{1}{\epsilon}\right)\right)  
\end{align*}
IFO calls to achieve an $\epsilon$-suboptimal solution of function $\tf(\cdot)$ defined in equation (\ref{eq:ifo-instance-1-pl}).

We then consider the case $\min\{\bar\kappa,\kappa\sqrt{N}\}<37a\sqrt{N}$ and construct the another hard instance.
We assume that $\epsilon<\Delta/2$.
We follow the functions provided by \citet{li2021page}, i.e., define $\tf_i:\BR^d\to\BR$ and $\tf:\BR^d\to\BR$ as
\begin{align}\label{eq:ifo-instance-2-pl}
    \tf_i(\vx)=c\langle \vv_i,\vx\rangle+\frac{L}{2}\|\vx\|^2
    \qquad\text{and}\qquad
     \tf(\vx)=\frac{1}{N}\sum_{i=1}^{N} \tf_i(\vx).
\end{align}
for all $i\in[N]$, where $c=\sqrt{L\Delta}$, $d=2N^2$ and 
\begin{align*}
\vv_i=\Big[\BI\Big(\Big\lceil\frac{1}{2N}\Big\rceil=i\Big),\BI\Big(\Big\lceil\frac{2}{2N}\Big\rceil=i\Big)\cdots,\BI\Big(\Big\lceil\frac{2N^2}{2N}\Big\rceil=i\Big)\Big]^\top\in\mathbb{R}^{d}.    
\end{align*}
We can verify that $\nabla^2f(\vx)=L\mI\succeq\mu\mI$ for any $\vx\in\BR^d$, which means the function $\tf$ is $L$-smooth and $\mu$-strongly convex, also $\mu$-PL
We also have
\begin{align*}
f^* 
= & \frac{1}{N}\sum_{i=1}^{N}\left(c\inner{\mathbf{u}_i}{\vx^*} +\frac{L}{2}\Vert \vx^*\Vert^2\right) \\
= & \frac{c}{N}\sum_{i=1}^{N} \inner{\mathbf{u}_i}{\vx^*}  +\frac{L}{2}\Vert \vx^*\Vert^2 \\
=& -\frac{c^2}{2LN^2}\bigg\Vert\sum_{i=1}^{N}\mathbf{u}_i\bigg\Vert^2
= -\frac{c^2}{L},
\end{align*}
where $\vx^*=-({c}/{NL})\vone$ is the minimizer of $f$.
Then the optimal function value gap holds 
\begin{align*}
f(\vx^0)-f^* = 0 - f^* =\frac{c^2}{L} = \Delta.
\end{align*}

We consider any IFO algorithm with initial point $\vzero$ the $t$-th IFO calls, which holds
\begin{align*}
    \vx^t \in {\rm span}\big(\{\nabla f_{i_0}(\vx^0),\dots,\nabla f_{i_{t-1}}(\vx^{t-1})\}\big) = {\rm span}\big(\{\vu_{i_0},\dots,\vu_{i_{t-1}}\}\big),
\end{align*}
where $i_\tau\in[N]$ is the index of individual which is accessed at the $\tau$-th IFO calls.
Since each $\vu_{i_\tau}$ has $2N$ nonzero entries, any vector $\vx\in\BR^d$ achieved by at most $N/2$ IFO calls has at least 
\begin{align*}
    d - \frac{N}{2}\cdot 2N = 2N^2 - N^2 = N^2
\end{align*}
zero entries.
Let $\fI_0=\{j\in[2n^2]:x_j=0\}$, then we have $|\fI|\geq N^2$.
Based on the construction of $f_i$ and $\vu_i$, we have
\begin{align*}
& f(\vx)-f^* \\
= & \frac{1}{N}\sum_{i=1}^{N} \left(c\inner{\vu_i}{\vx} + \frac{L}{2}\norm{\vx}\right) - \left(-\frac{c^2}{L}\right) \\
= & \sum_{j=1}^{2N^2}\left(\frac{c}{N}x_j+\frac{L}{2}x_j^2+\frac{c^2}{2LN^2}\right) \\
= & \sum_{j\in\fI_0}\left(\frac{c}{N}x_j+\frac{L}{2}x_j^2+\frac{c^2}{2LN^2}\right) + \sum_{j\not\in\fI_0}\left(\frac{c}{N}x_j+\frac{L}{2}x_j^2+\frac{c^2}{2LN^2}\right) \\
\geq & N^2\cdot\frac{c^2}{2LN^2} + \sum_{j\not\in\fI_0}\left(x_j+\frac{c}{NL}\right)^2 \\
\geq & \frac{\Delta}{2}>\epsilon,
\end{align*}
Hence, achieving an $\epsilon$-suboptimal solution requires at least $N/2+1=\Omega(N)$ IFO calls.

Combining the results in above two hard instances, we achieve the lower bound on IFO complexity of
\begin{align*}
    \Omega\left(N+\min\{ N\kappa,\sqrt{N}\bar\kappa\}\ln\left(\frac{1}{\epsilon}\right)\right).
\end{align*}
\end{proof}

Finally, we show the lower bound on the communication rounds under the PL condition.

\begin{theorem}
        For any $\bar L>0$, $L>0$, $\mu>0$, $m>0$, $n>0$, $\Delta>0$, and $\epsilon>0$ with $\epsilon<0.005\Delta$, $\bL\geq L\geq37a\mu$, and $a=19708$, there exists $\bar L$-mean-squared smooth functions $\{f_{i,j}:\BR^d\to\BR\}$ such that the function $f=\frac{1}{mn}\sum_{i=1}^{m}\sum_{j=1}^{n}f_{i,j}$ is $L$-smooth and $\mu$-PL with $f(\vx^0)-\inf_{\vy\in\BR^d}f(\vy)\leq \Delta$. 
        In order to find an $\epsilon$-suboptimal solution of problem $\min_{\vx\in\BR^d}f(\vx)$, any LIFO algorithm needs at least the computation rounds of~$\Omega\big(n+\big(\kappa+\min\{n\kappa, \sqrt{n/m}\bar \kappa\}\big)\ln(1/\epsilon)\big)$.
\end{theorem}
\begin{proof}
    We first show the lower bound in the view of linear speed-up.
    According to Theorem \ref{thm:LFO-pl} with $N=mn$ and $\epsilon<0.005\Delta$, there exist 
    $\bar L$-mean-squared smooth function set $\{f_{i,j}\}_{i=1,j=1}^{m,n}$ such that the function $f(\cdot)=\frac{1}{N}\sum_{i=1}^{N}f_i(\cdot)$ is $L$-smooth, $\mu$-PL and satisfies $f(\vx^0)-f^*\leq \Delta$. 
    In order to find an $\epsilon$-suboptimal solution of problem $\min_{\vx\in\BR^d}f(\vx)$, any IFO algorithm needs at least 
    \begin{align}\label{eq:IFO-computation-pl}
        \Omega\left(mn+\min\{ mn\kappa,\sqrt{mn}\bar\kappa\}\ln\left(\frac{1}{\epsilon}\right)\right)
    \end{align}
    IFO calls. 
    For the distributed setting with $m$ nodes, any DIFO algorithm can perform at most $m$ LIFO calls in per computation round.
    Therefore, we achieve the lower complexity bound of on the computation rounds of 
    \begin{align}\label{eq:IFO-computation2-pl}
    \small\begin{split}
    \Omega\left(\frac{mn}{m}+\frac{\min\{ mn\kappa,\sqrt{mn}\bar\kappa\}}{m}\ln\left(\frac{1}{\epsilon}\right)\right)
    =  \Omega\left(n+\min\left\{n\kappa ,\sqrt{\frac{n}{m}}\bar\kappa\right\}\ln\left(\frac{1}{\epsilon}\right)\right).
    \end{split}
    \end{align}

    Now show the lower bound $\Omega(\kappa\ln(1/\epsilon))$ on the computation rounds. It is worth noting that it is necessary since the case of $\bar L/L<\sqrt{m/n}$ leads to
    \begin{align*}
        \kappa>\min\{\kappa n,\bar\kappa\sqrt{n/m}\}.
\end{align*}
We consider the PL function $g_{T,t}$ defined in equation (\ref{eq:g}) with 
\begin{align}\label{eq:ifo-T2-pl}
T=\left\lfloor \frac{\kappa}{37a}\right\rfloor, \quad t=2\left\lfloor \log_{\frac{8}{7}}\frac{2\Delta}{3\epsilon}\right\rfloor, \quad \text{where}~~\epsilon<0.01\Delta~~\text{and}~~a=19708.    
\end{align}
According to Lemma \ref{lem:scale} with 
\begin{align}\label{eq:ifo-alpha-beta2-pl}
g(\vx)=g_{T,t}(\vx), \qquad \alpha = \frac{\Delta}{3T} \qquad \text{and} \qquad 
\beta=\sqrt{\frac{3LT}{37\Delta}},    
\end{align}
we can conclude the function $\hat g(\vx)=\alpha g_{T,t}(\beta\vx)$ is $37\alpha\beta^2$-smooth and $\alpha\beta^2/(aT)$-PL, $\hat g(\vx)-\inf_{\vy\in\BR^T}\hat g(\vy) = \alpha \left(g_{T,t}(\vx)-\inf_{\vy\in\BR^T} g_{T,t}(\vy)\right)$, and it holds 
\begin{align}\label{eq:lower-hat-g2-pl}
\hat g(\mathbf{0})-\inf_{\vy\in\BR^T}\hat g(\vy) = \alpha \left(g_{T,t}(\mathbf{0})-\inf_{\vy\in\BR^T} g_{T,t}(\vy)\right) \leq 3\alpha T, 
\end{align}
where the inequality is based on Lemma \ref{cor:g}(d).
We construct the hard instance as
\begin{align}\label{eq:difo-instance-2-PL}
     f_{i,j}(\vx)=\hat g(\vx)
     \qquad\text{and}\qquad
     f(\vx)=\frac{1}{mn}\sum_{i=1}^{m}\sum_{j=1}^{n}f_{i,j}(\vx) 
\end{align}
such that $\{f_{i,j}\}_{i=1,j=1}^{m,n}$ is $37\alpha\beta^2$-mean-squared smooth and $f$ is $37\alpha\beta^2$-smooth with 
\begin{align*}
f(\vx^0)-f^*=\hat g(\mathbf{0})-\inf_{\vy\in\BR^T}\hat g(\vy)
\overset{(\ref{eq:lower-hat-g2-pl})}{\leq} 3\alpha T.     
\end{align*}
The settings of $\alpha$, $\beta$, $T$ and $t$ in equations  (\ref{eq:ifo-T2-pl}) and (\ref{eq:ifo-alpha-beta2-pl})  imply
\begin{align*}
	37\alpha\beta^2\overset{(\ref{eq:ifo-alpha-beta2-pl})}{=} L\leq\bar L,\qquad\frac{\alpha\beta^2}{aT}\overset{(\ref{eq:ifo-T2-pl})}{\geq}\mu \qquad\text{and}\qquad  
    3\alpha T \overset{(\ref{eq:ifo-alpha-beta2-pl})}{\leq}\Delta.
\end{align*}
Therefore, the global objective $f=\frac{1}{mn}\sum_{i=1}^m\sum_{j=1}^nf_{i,j}$ is $L$-smooth, $\mu$-PL and satisfies $f(\mathbf{0})-f^*\leq\Delta$, and the function set $\{f_{i,j}\}$ is $\bar L$-mean-squared smooth, which satisfies our requirements.

We then show that any DIFO algorithm require at least 
$Tt/2+1$ computation rounds to achieve an $\epsilon$-suboptimal solution of the problem $\min_{\vx\in\BR^d}f(\vx)$.
We take~$\delta=\epsilon/\Delta$.
According to Lemma \ref{lem:scale} and \ref{cor:g}(e), for all $\vx_i\in\BR^{Tt}$ with ${\rm supp}(\vx_i)\subseteq\left\{1,2,\cdots,Tt/2\right\}$, we have
\begin{align}\label{eq:ifo-grad-upper2-pl}
   f(\vx_i)-f^*= \alpha g_{T,t}(\beta\vx_i)-\alpha g_{T,t}^*>3\alpha T\delta=\epsilon.
\end{align}
We consider the vectors $\vx_1,\dots,\vx_m\in\BR^{Tt}$ which are achieved by an DIFO algorithm with at most $Tt/2$ computation rounds.
Lemma \ref{cor:g}(a) implies such vector $\vx_i$ has at most $Tt/2$ non-zero entries, that is ${\rm supp}(\vx_i)\subseteq\left\{1,2,\cdots,Tt/2\right\}$ for all $i\in[m]$.
Therefore, we have
\begin{align*}
f(\vx_i)-f^*>\epsilon.
\end{align*}
Hence, any LIFO algorithm requires at least 
\begin{align*}
 \frac{Tt}{2}+1=\Omega\left(\kappa\ln\left(\frac{1}{\epsilon}\right)\right)  
\end{align*}
computation rounds to achieve an $\epsilon$-suboptimal solution of function $f(\cdot)$ defined in equation (\ref{eq:difo-instance-2-PL}) at local nodes.

Combining above two hard instance, we obtain the lower bound on the communication rounds of
\begin{align*}
 \Omega\left(n+\left(\kappa + \min\left\{n\kappa ,\sqrt{\frac{n}{m}}\bar\kappa\right\}\ln\left(\frac{1}{\epsilon}\right)\right)\right).
\end{align*}
\end{proof}

\section{Numerical Experiments}

In this section, we validate our theory by conducting the numerical experiments on the following problems:
\begin{enumerate}[itemsep=3pt,topsep=2pt]
    \item[(a)] Hard Instances: We consider the hard instance for the general nonconvex case defined in equation (\ref{eq:ifo-instance-1}) with $T=10^2$ and $\beta=10^3$ and the hard instance for the PL case defined in equation (\ref{eq:comm-alpha-beta-pl}) with $T=100$, $t = 5$, and $\beta = 10^3$.    
    \item[(b)] Linear Regression: We consider the (regularized) linear regression, which corresponds to the individual function 
    \begin{align*}
        f_{i,j}(\vx) = \frac{1}{2}(\va_{i,j}^\top\vx-b_{i,j})^2 + \lambda r_\alpha(\vx),
    \end{align*}
    where 
    \begin{align}\label{eq:regularizer}
        r_{\alpha}(\vx) =   \sum_{k=1}^d \frac{\alpha x_k^2}{1+\alpha x_k^2}
    \end{align}
    is the nonconvex regularizer \cite{antoniadis2011penalized} with $\alpha=10$, $\va_{i,j}\in\BR^d$ is the feature of the~$j$-th sample on the agent $i$, $b_{i,j}\in\BR$ is the corresponding label, and $\lambda\geq 0$ is hyperparameter.
    We set $\lambda = 10^{-8}$ for the general nonconvex case and $\lambda=0$ for the PL case.     
    \item[(c)] Logistic Regression: We consider the (regularized) logistic regression, which corresponds to the individual function 
    \begin{align*}
        f_{i,j}(\vx) = \ln (1 + \exp(-b_{ij} \va_{ij}^\top \vx)) + \lambda r_{\alpha} (\vx), 
    \end{align*}
    where $r_\alpha$ is defined in equation (\ref{eq:regularizer}) with $\alpha=10$, $\va_{i,j}\in\BR^d$ is the feature of the~$j$-th sample on the agent $i$, $b_{i,j}\in\{-1,1\}$ is the corresponding label, and $\lambda\geq 0$ is hyperparameter.
    We also set $\lambda = 10^{-8}$ for the general nonconvex case and $\lambda=0$ for the PL case.
\end{enumerate}

We include the experiments on problems of linear regression and logistic regression on datasets ``Wikipedia Math Essential''  ($mn=1,
056$, $d=730$) \citep{rozemberczki2021pytorch} and ``RCV1.binary'' \citep{lewis2004rcv1, CC01a} ($N=20,096$, $d=47,236$), respectively.
We perform all experiments on the ring graph with $16$ agents (i.e., $m=16$).

For the general nonconvex case, we compare \DEAREST~(Algorithm \ref{alg:DEAREST}) with
the classical Decentralized Stochastic Gradient Descent (\texttt{DSGD}) \citep{nedic2009distributed,lian2017can}
and the state-of-the-art variance reduced method \texttt{DESTRESS}~\citep{li2022destress}.
For the PL case, we compare \DEAREST~(Algorithm \ref{alg:DEAREST}) with  \texttt{DSGD}~\citep{nedic2009distributed,lian2017can} and \texttt{DRONE}~\cite{bai2024complexity}. 
We tune the stepsize~$\eta$ from $\{ 10^1, 10^0, 10^{-1}, 10^{-2}, 10^{-3}, 10^{-4} \}$ for all methods and the mini-batch size $b$ from $\{ 2^4, 2^5, 2^6, 2^7\}$ for \texttt{DSGD}, \DESTRESS, and \DEAREST. 
For \DEAREST, we simply set the probability for full gradient computation as $p=b/(mn)$ according to the theoretical setting in Theorem \ref{thm:main}. 
For \texttt{DRONE}, we set the mini-batch size of agents be $4$ and tune probability for the full participation from $\{0.1, 0.3, 0.5, 0.9\}$.

\begin{figure}[t]
    \centering
    \begin{tabular}{ c c c}
    \includegraphics[scale=0.24]{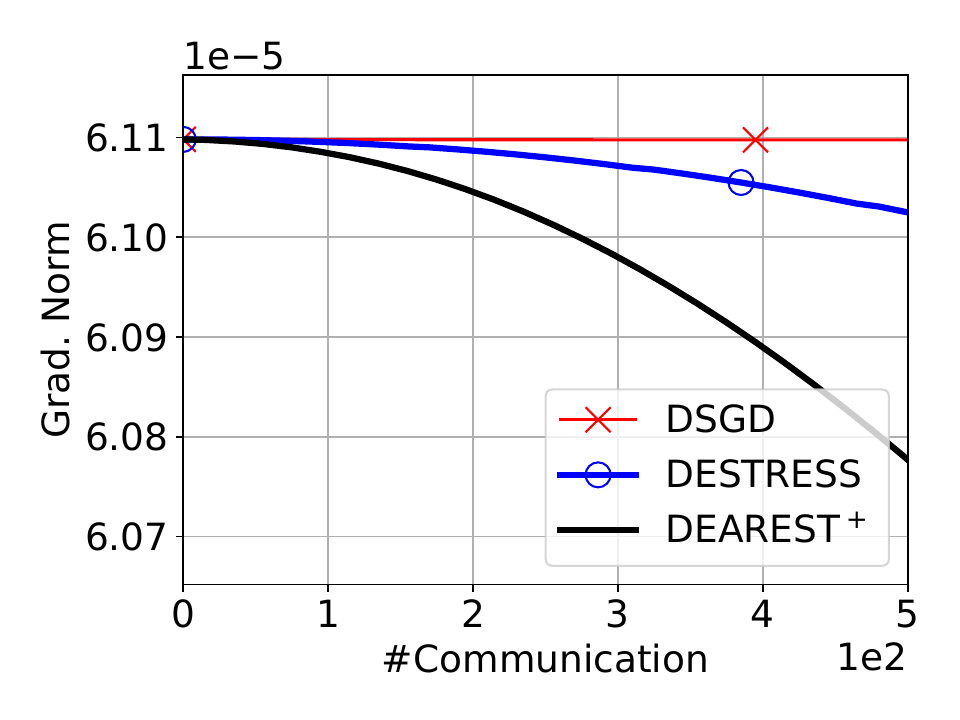} & 
    \includegraphics[scale=0.24]{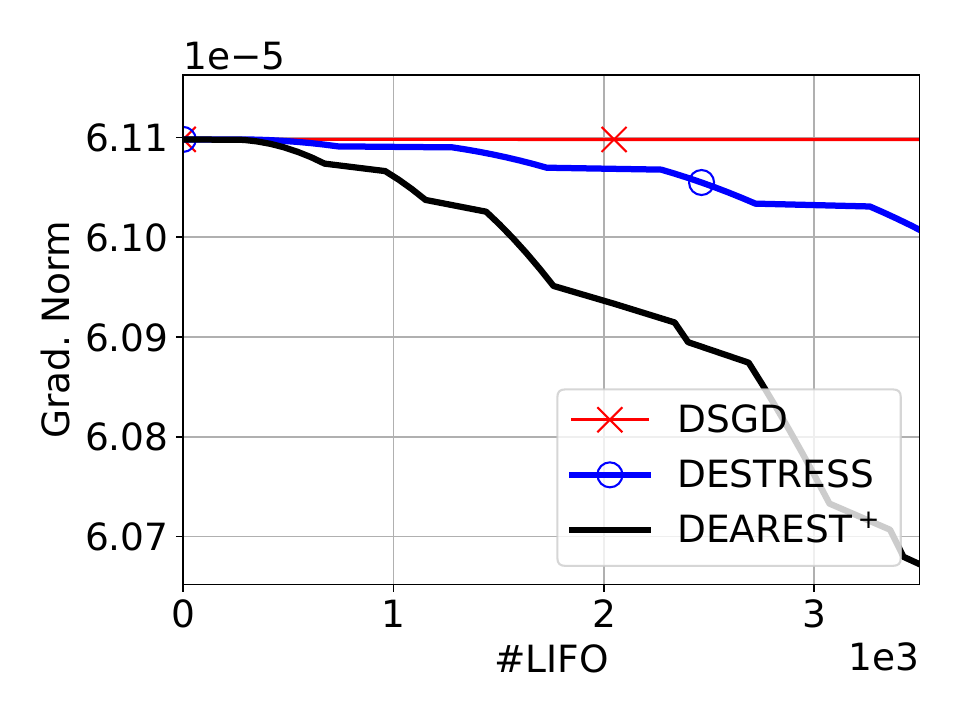}  & \includegraphics[scale=0.24]{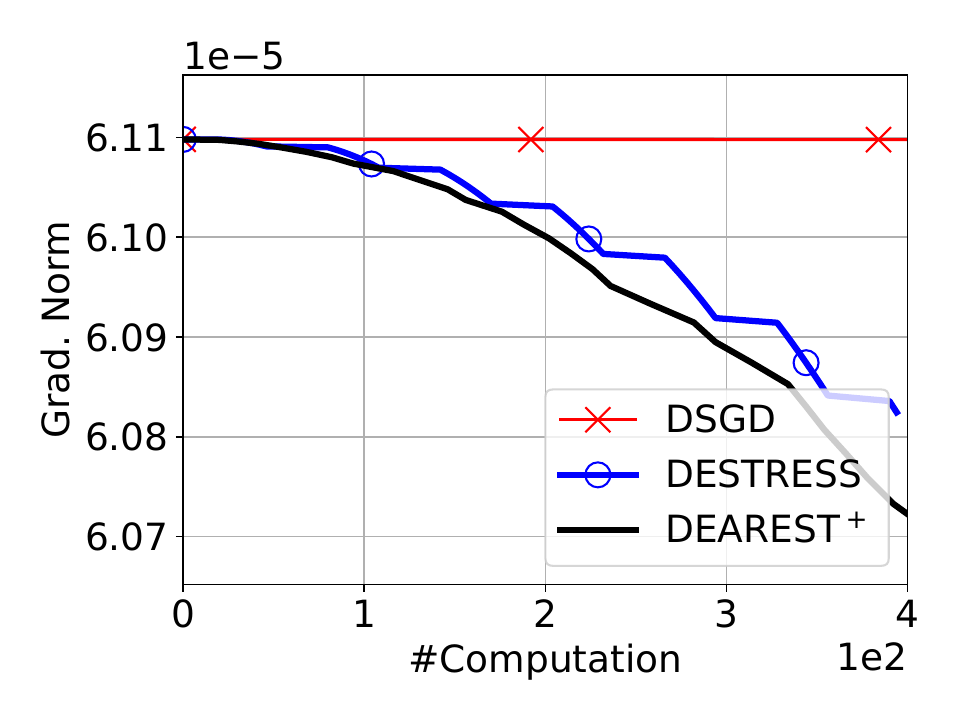}  \\
    \footnotesize (a) Communication &
    \footnotesize (b) LIFO  & 
    \footnotesize (c)  Computation
    \end{tabular} \vskip0.1cm
    \caption{Results on the hard instance (\ref{eq:ifo-instance-1}) in the lower bound for the general nonconvex case.
    }
    \label{fig:nonconvex-lb}
\end{figure}

\begin{figure}
    \centering
    \begin{tabular}{ c c c}
    \includegraphics[scale=0.24]{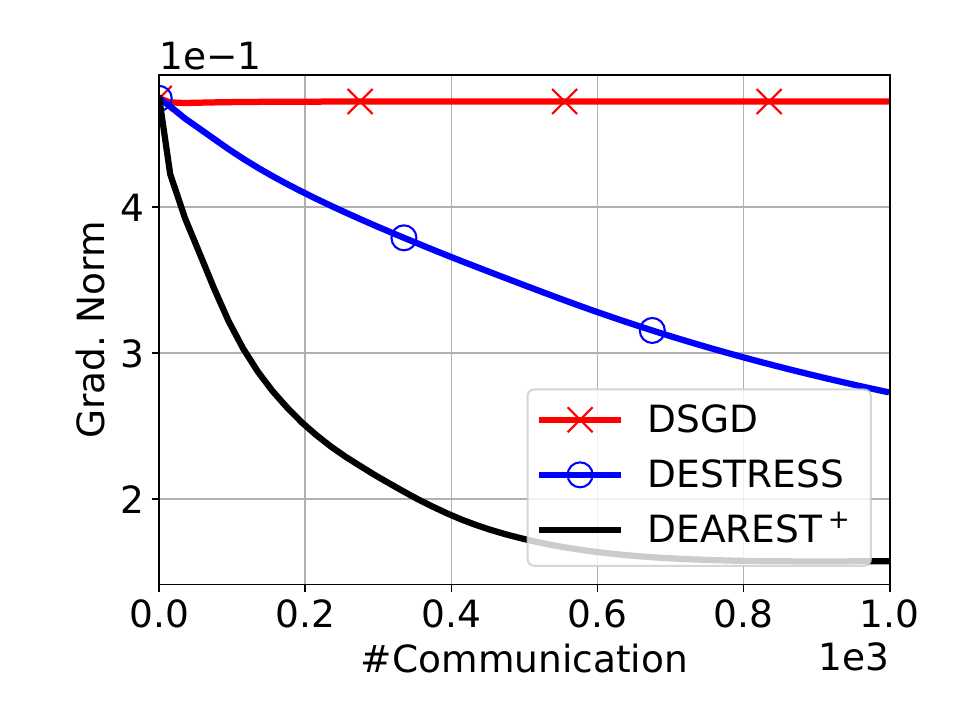} & 
    \includegraphics[scale=0.24]{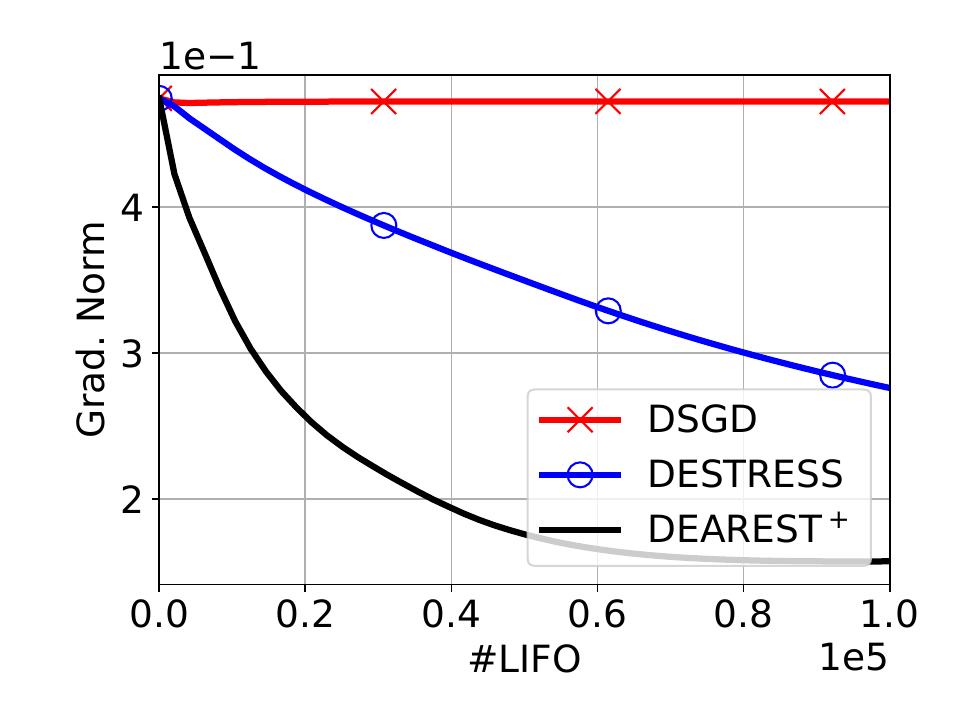}  & \includegraphics[scale=0.24]{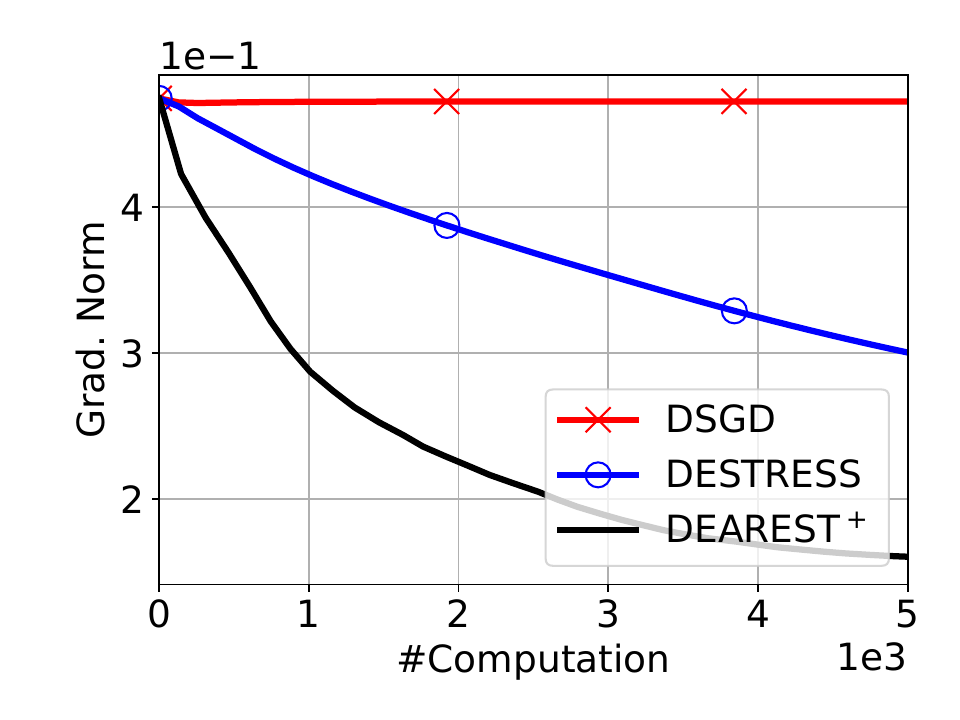}  \\
    \footnotesize (a) Communication &
    \footnotesize (b) LIFO  & 
    \footnotesize (c)  Computation
    \end{tabular} \vskip0.1cm
    \caption{Results on the regularized linear regression (the general nonconvex case).
    }
    \label{fig:nonconvex-linear}
\end{figure}

\begin{figure}
    \centering
    \begin{tabular}{ c c c}
    \includegraphics[scale=0.24]{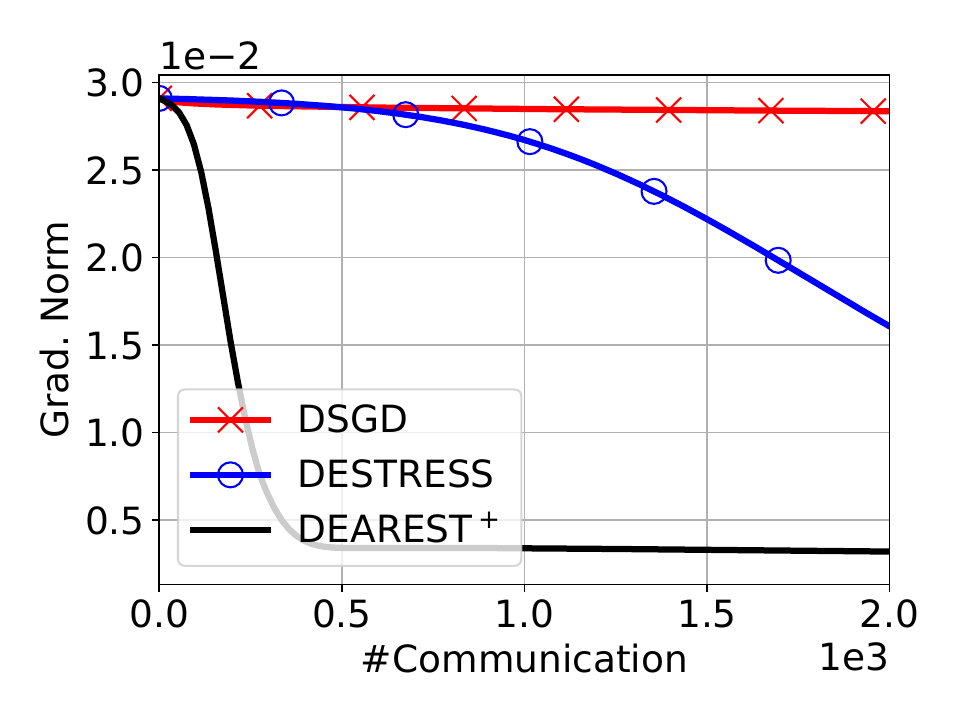} & 
    \includegraphics[scale=0.24]{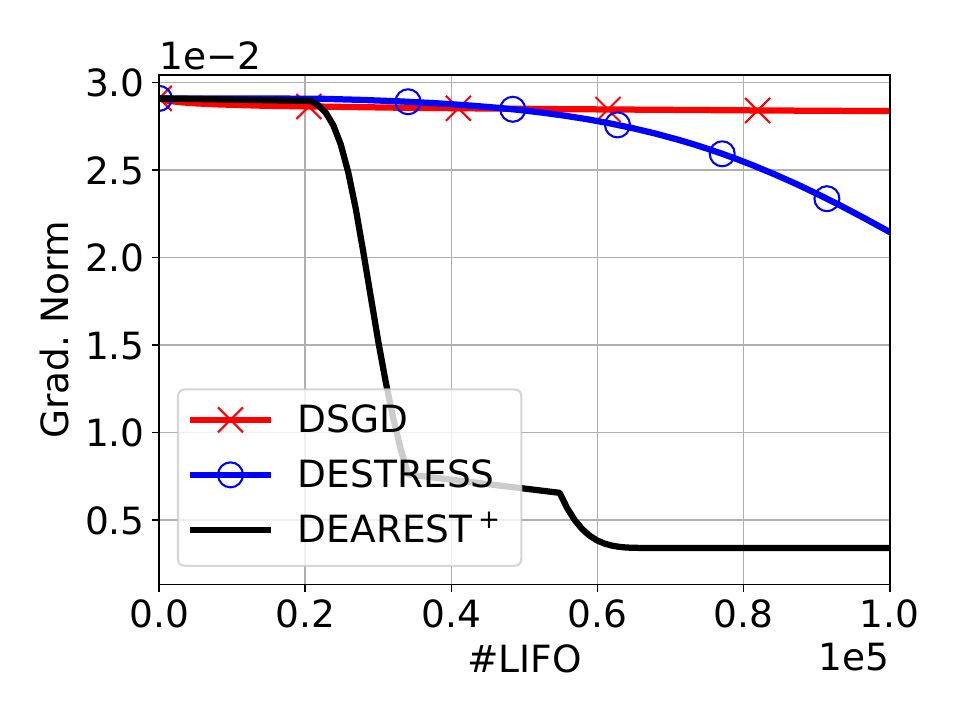}  & \includegraphics[scale=0.24]{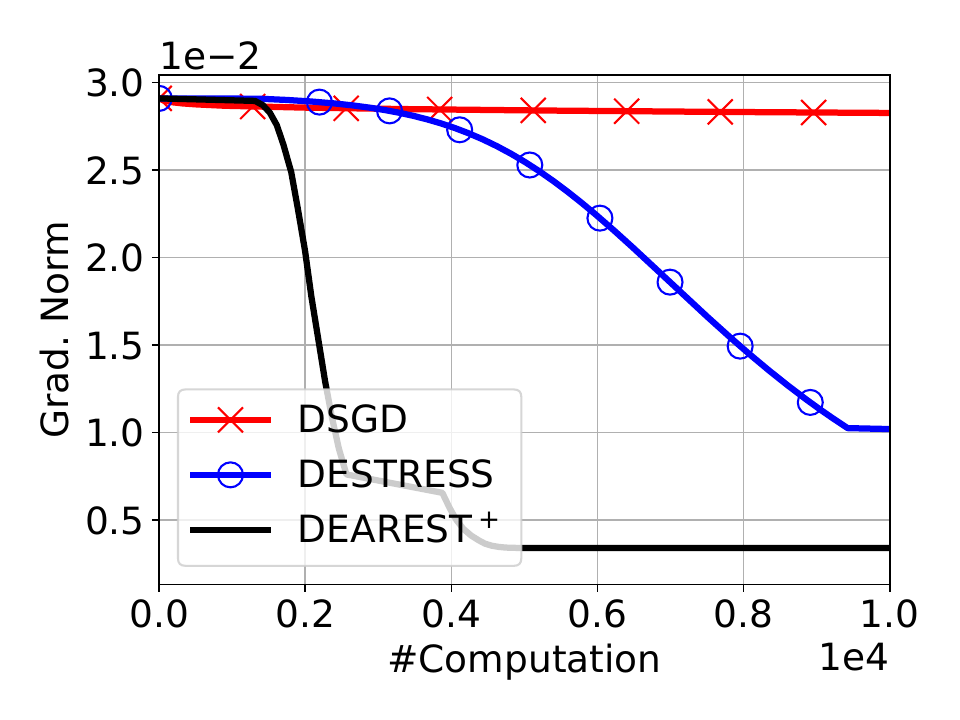}  \\
    \footnotesize (a) Communication &
    \footnotesize (b) LIFO  & 
    \footnotesize (c)  Computation
    \end{tabular} \vskip0.1cm
    \caption{Experiment results on the regularized logistic regression (the general nonconvex case).
    }
    \label{fig:nonconvex-logit}
\end{figure}

We present the experimental results of the communication rounds, the LIFO calls, and the computation rounds against the gradient norm for the general nonconvex problems in Figures~\ref{fig:nonconvex-lb}--\ref{fig:nonconvex-logit}.
We can observe that the proposed \DEAREST~outperforms the  baselines in terms of all three metrics.
We also present the experimental results of the communication rounds, the LIFO calls, and the computation rounds against the objective function value for the PL problems in Figures~\ref{fig:pl-lb}--\ref{fig:pl-logit}.
We can observe that the proposed \DEAREST~outperforms the  baselines in terms of the LIFO calls and the computation rounds. 
Additionally, \DRONE~has the comparable communication rounds to \DEAREST.
This is because of the iteration scheme of \DRONE~can be regarded as the special case of \DEAREST~with $n=1$.
Although the analysis of \DRONE~in our conference paper~\cite{bai2024complexity} only considers the mean-squared smoothness, this method with the appropriate parameter setting (Theorem \ref{thm:PL-1} with $n=1$) can achieve the communication complexity depends on the global smoothness dependency.

\begin{figure}
    \centering
    \begin{tabular}{c c c}
    \includegraphics[scale=0.24]{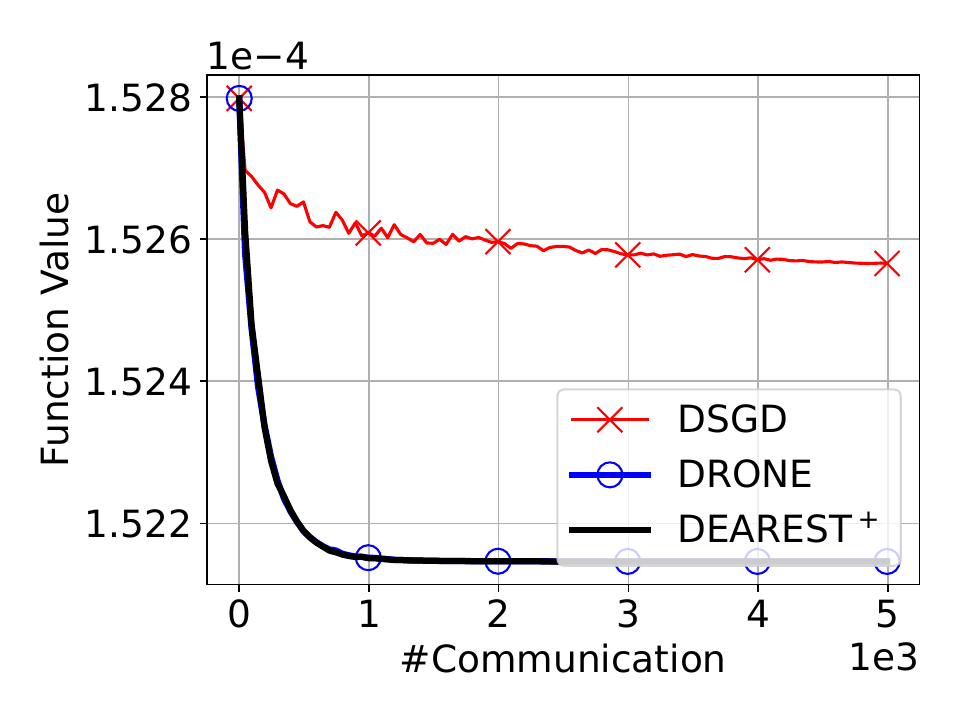} & 
    \includegraphics[scale=0.24]{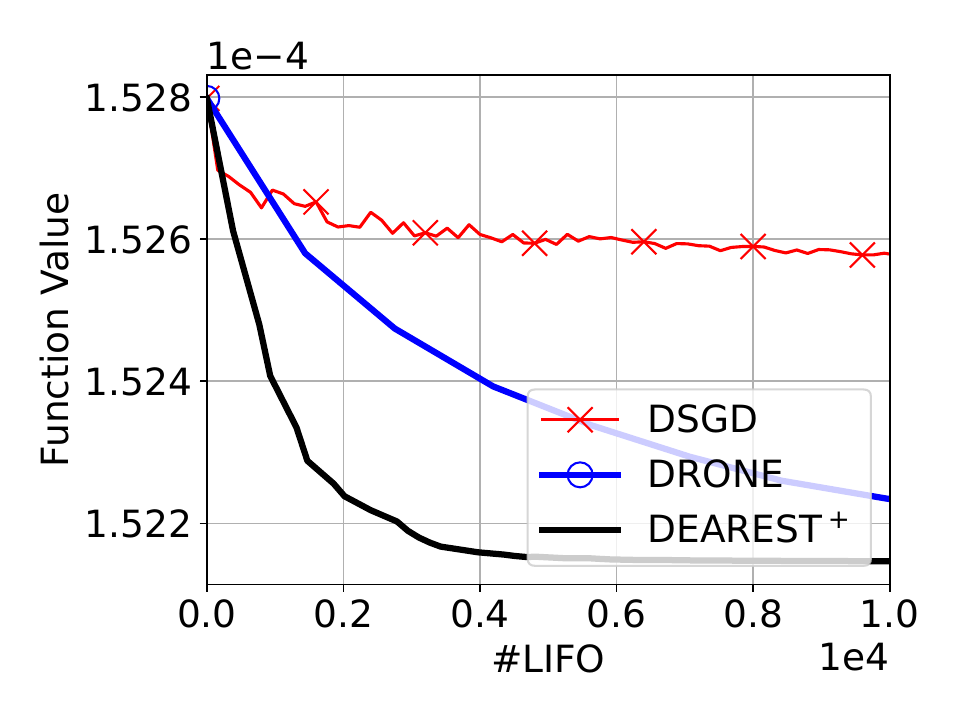}  & \includegraphics[scale=0.24]{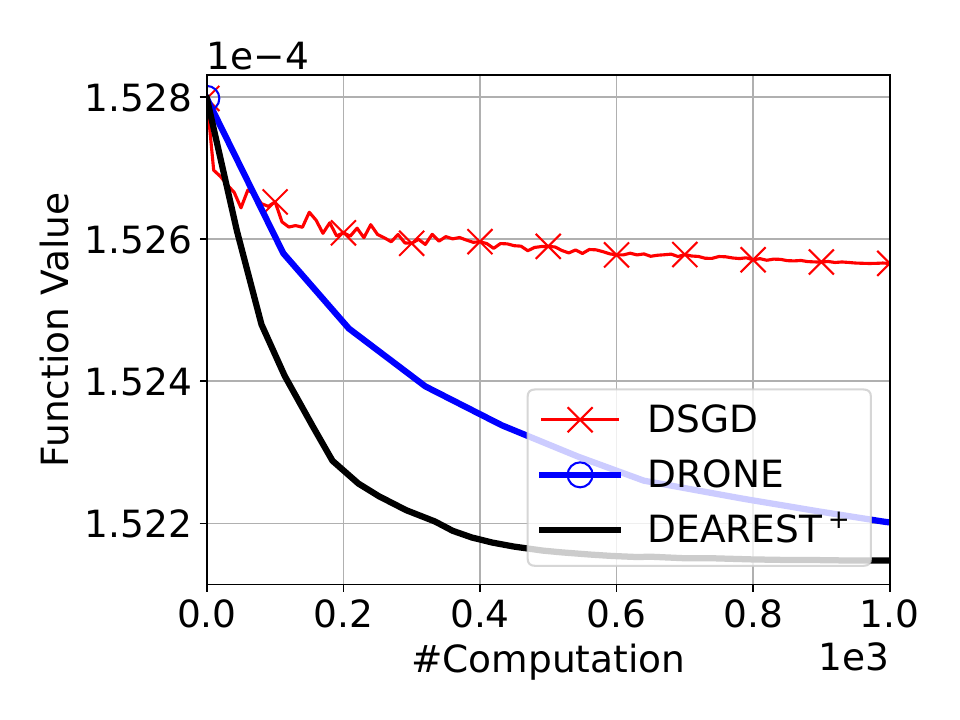}  \\
    \footnotesize (a) Communication &
    \footnotesize (b) LIFO  & 
    \footnotesize (c)  Computation
    \end{tabular} \vskip0.1cm
    \caption{Results on the hard instance (\ref{eq:comm-alpha-beta-pl}) in the lower bound for the PL case.
    }
    \label{fig:pl-lb}
\end{figure}

\begin{figure}
    \centering
    \begin{tabular}{c c c}
    \includegraphics[scale=0.24]{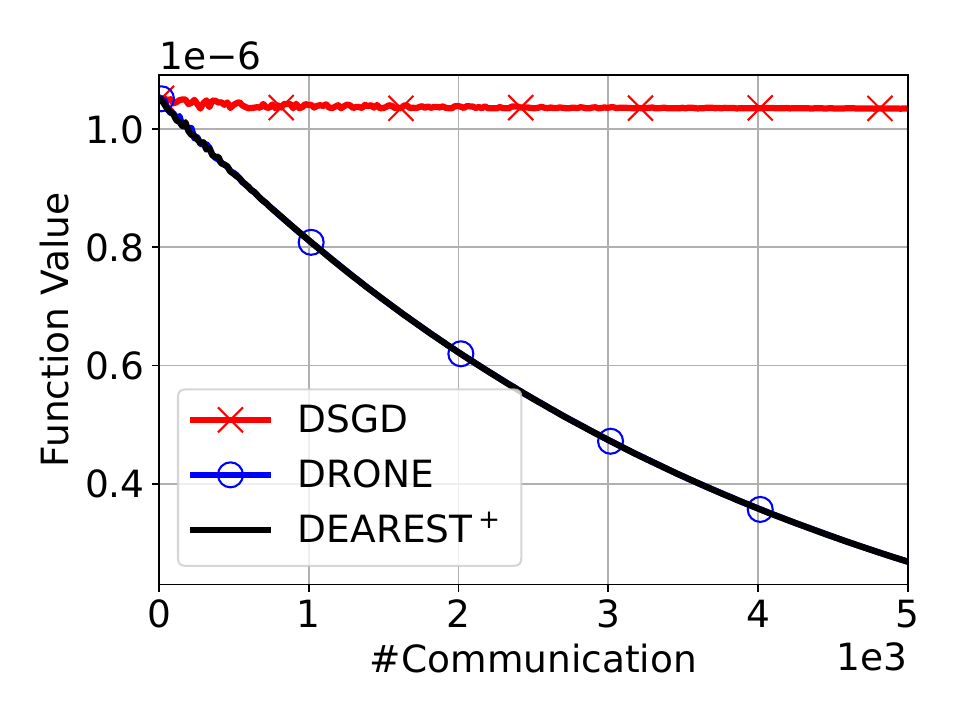} & 
    \includegraphics[scale=0.24]{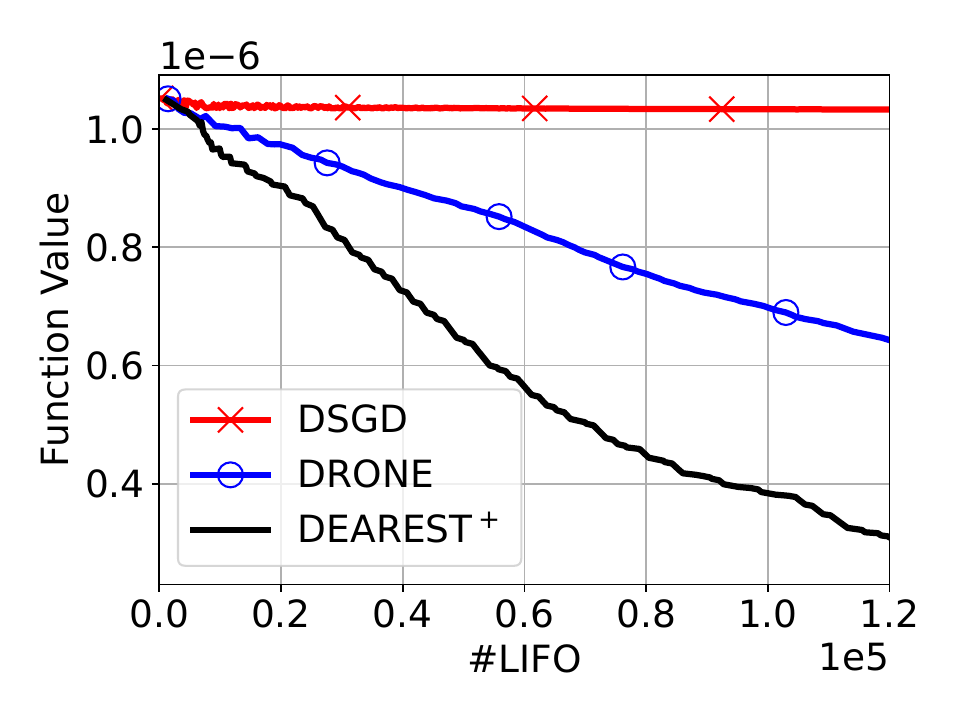}  & \includegraphics[scale=0.24]{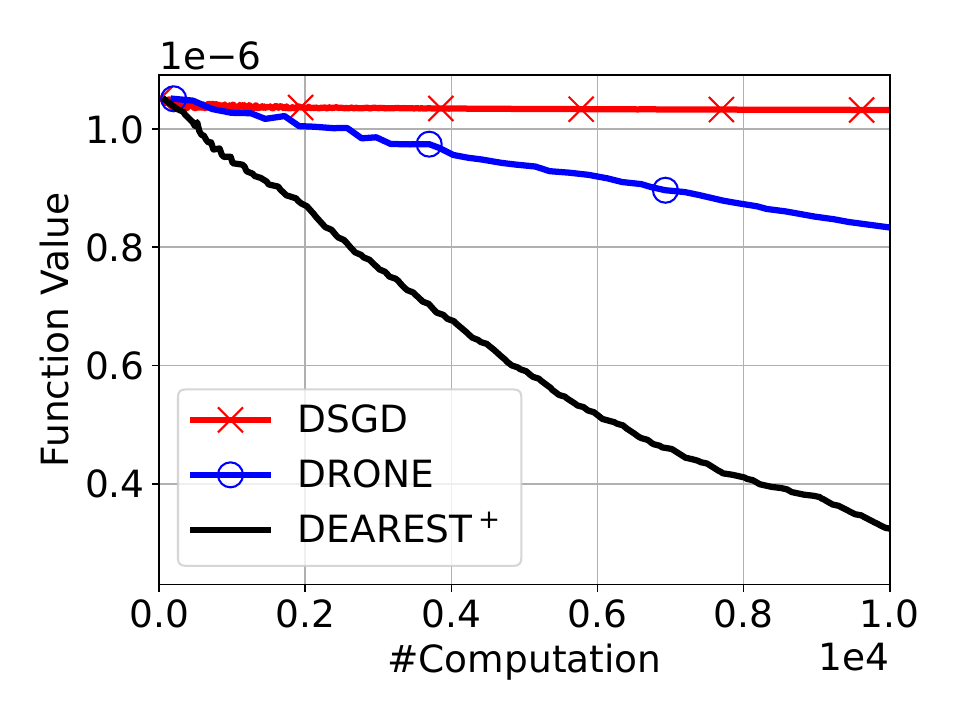}  \\
    \footnotesize (a) Communication &
    \footnotesize (b) LIFO  & 
    \footnotesize (c)  Computation
    \end{tabular} \vskip0.1cm
    \caption{Results on the linear regression (the PL case).
    }
    \label{fig:pl-linear}
\end{figure}

\begin{figure}
    \centering
    \begin{tabular}{c c c}
    \includegraphics[scale=0.24]{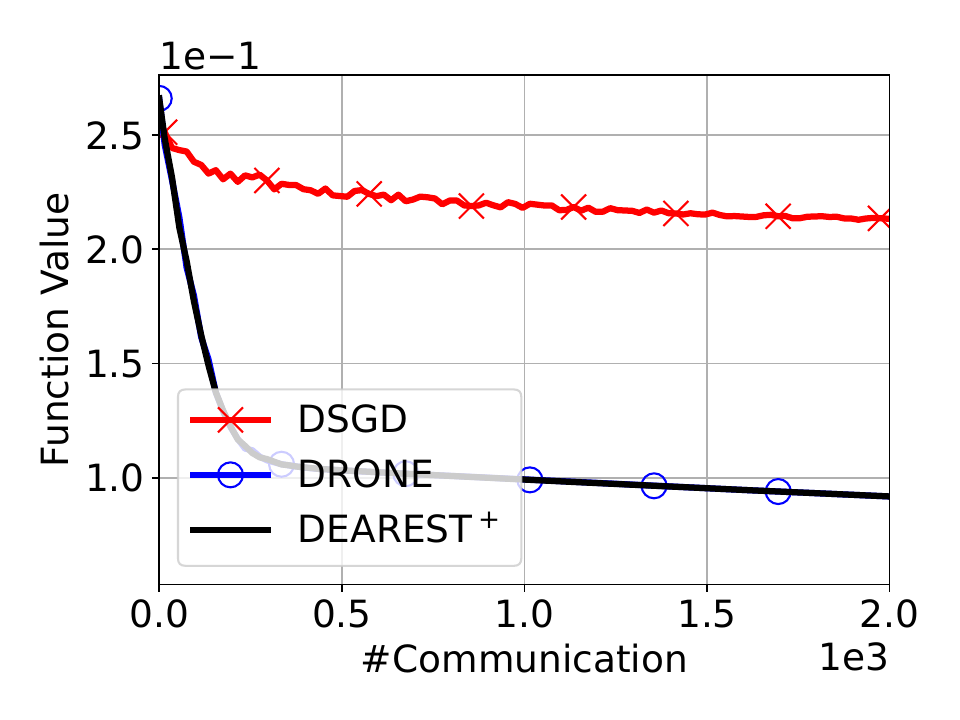} & 
    \includegraphics[scale=0.24]{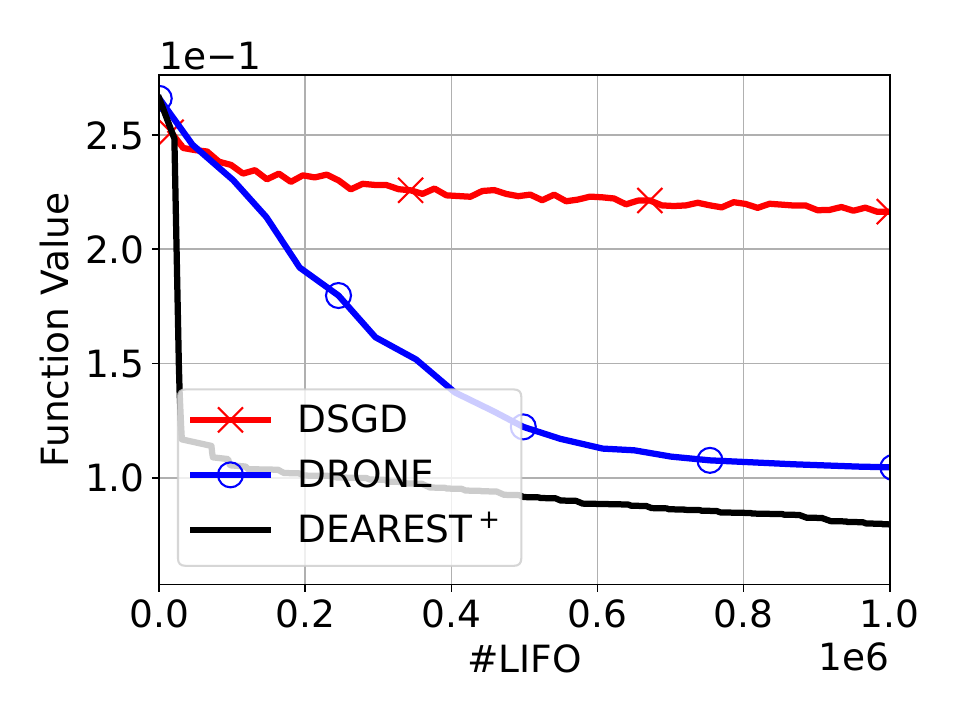}  & \includegraphics[scale=0.24]{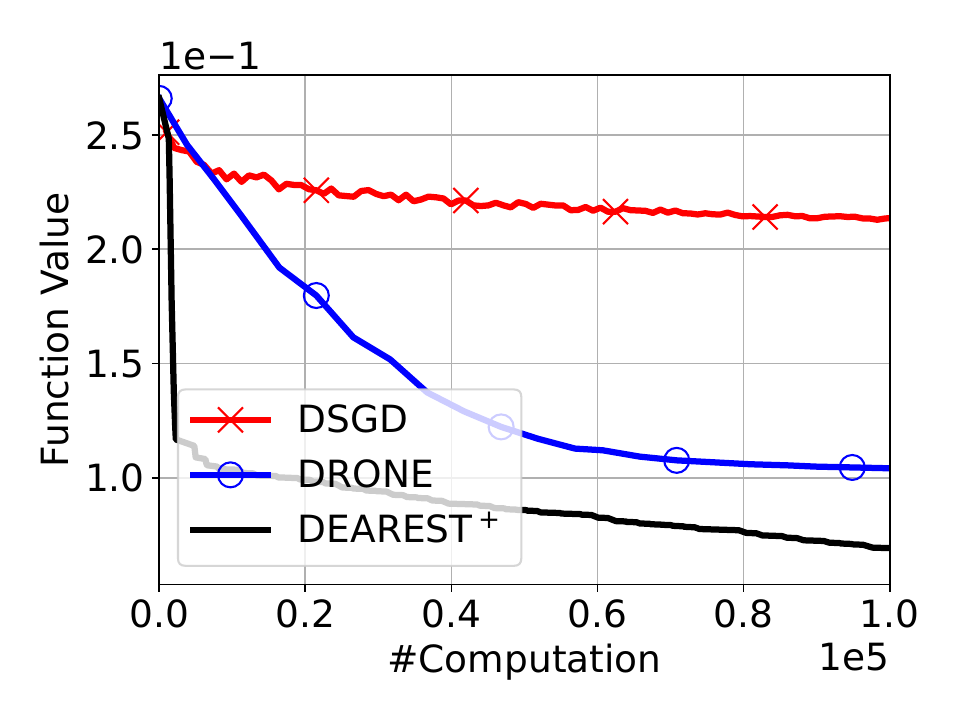}  \\
    \footnotesize (a) Communication &
    \footnotesize (b) LIFO  & 
    \footnotesize (c)  Computation
    \end{tabular} \vskip0.1cm
    \caption{Experiment results on the logistic regression (the PL case).
    }
    \label{fig:pl-logit}
\end{figure}

\section{Conclusion}

In this paper, we study the complexity of decentralized smooth nonconvex finite-sum optimization by considering both the smoothness of the global objective and the mean-squared smoothness of all individual functions. 
We propose \DEAREST~method, which simultaneously achieves the (near) optimal complexity in terms of the communication rounds, the LIFO calls, and the computation rounds.
We also extend our results to PL condition and show the (near) optimality in such case.
In future work, it is possible to apply our ideas to establish the sharper complexity bounds for the decentralized optimization problem over time-varying network \cite{metelev2024decentralized} and consider the problem in online setting~\cite{yuan2022revisiting,lu2021optimal,arjevani2023lower}.
We are also interested in extending our results to address the functions that satisfy the Kurdyka--{\L}ojasiewicz inequality \cite{bolte2014proximal,bolte2007lojasiewicz,attouch2009convergence,zhou2018convergence,fatkhullin2022sharp,jiang2022holderian}

\input{appendix.tex}

\bibliographystyle{plainnat}
\bibliography{sn-bibliography}

\end{document}

%% file: math.tex
\usepackage{amsmath}\allowdisplaybreaks
\usepackage{amsfonts,bm}
\usepackage{amssymb}

\def\eqref#1{equation~(\ref{#1})}

\def\1{\bf{1}}

\newcommand{\Norm}[1]{\left\| #1 \right\|}
\newcommand{\norm}[1]{\left\| #1 \right\|_2}
\def\inner#1#2{\langle #1, #2 \rangle}

\def\vzero{{\bf{0}}}
\def\vone{{\bf{1}}}

\def\va{{\bf{a}}}
\def\vb{{\bf{b}}}

\def\ve{{\bf{e}}}

\def\vg{{\bf{g}}}

\def\vs{{\bf{s}}}

\def\vu{{\bf{u}}}
\def\vv{{\bf{v}}}

\def\vx{{\bf{x}}}
\def\vy{{\bf{y}}}
\def\vz{{\bf{z}}}

\def\fE{{\mathcal{E}}}

\def\fG{{\mathcal{G}}}

\def\fI{{\mathcal{I}}}

\def\fM{{\mathcal{M}}}

\def\fO{{\mathcal{O}}}

\def\fR{{\mathcal{R}}}

\def\fV{{\mathcal{V}}}

\def\BE{{\mathbb{E}}}

\def\BI{{\mathbb{I}}}

\def\BP{{\mathbb{P}}}

\def\BR{{\mathbb{R}}}

\def\BZ{{\mathbb{Z}}}

\def\mF {{\bf F}}
\def\mG {{\bf G}}

\def\mI {{\bf I}}

\def\mL {{\bf L}}

\def\mS {{\bf S}}

\def\mU {{\bf U}}

\def\mW {{\bf W}}
\def\mX {{\bf X}}
\def\mY {{\bf Y}}
\def\mZ {{\bf Z}}

\def\prog{{\rm{prog}}}

\newcommand{\R}{\mathbb{R}}

\usepackage{amsthm}
\usepackage{etoolbox}

\makeatletter
\def\Ddots{\mathinner{\mkern1mu\raise\p@
\vbox{\kern7\p@\hbox{.}}\mkern2mu
\raise4\p@\hbox{.}\mkern2mu\raise7\p@\hbox{.}\mkern1mu}}
\makeatother

\makeatletter
\newcommand*{\rom}[1]{\expandafter\@slowromancap\romannumeral #1@}
\makeatother

\def\FM {{{\texttt{AccGossip}}}}
\def\bg {{{\bar \vg}}}
\def\bs {{{\bar \vs}}}

\def\bx {{{\bar \vx}}}
\def\by {{{\bar \vy}}}
\def\bL {{{\bar L}}}
\def\Lm {L_{\max}}

\def\hK {{{\hat K}}}
\def\hrho {{{\hat \rho}}}

\def\bz {{{\bar \vz}}}

\def\tf {{{\tilde f}}}

\def\bkappa {{{\bar \kappa}}}

\def\DEAREST{{\texttt{DEAREST$^+$}}}
\def\DESTRESS{{\texttt{DESTRESS}}}
\def\DGET{{\texttt{D-GET}}}
\def\GTSARAH{{\texttt{GT-SARAH}}}
\def\GTPAGE{{\texttt{GT-PAGE}}}
\def\EDSGD{{\texttt{EDSGD}}}
\def\DRONE{{\texttt{DRONE}}}
\def\DGDGT{{\texttt{DGD-GT}}}

\def\hK {{{\hat K}}}

%% file: appendix.tex
\begin{appendices}

\section{The Proofs in Section \ref{sec:settings}}\label{secA1}

This section provide the proofs for the relationships among the smoothness parameters shown in Section \ref{sec:settings}.

\subsection{The Proof of Proposition \ref{prop:smooth}}
\begin{proof} 
\textbf{Part (a):}
For any $i\in[m]$ and $j\in[n]$, we have
\begin{align}\label{eq:smooth-proof-1}
\Norm{\nabla f_{i,j}(\vx)-\nabla f_{i,j}(\vy)}^2
\leq \sum_{i=1}^m\sum_{j=1}^n\Norm{\nabla f_{i,j}(\vx)-\nabla f_{i,j}(\vy)}^2 
\overset{(\ref{eq:smooth-mean-squared})}{\leq} mn\bL^2\Norm{\vx-\vy}^2
\end{align}
for all $\vx,\vy\in\BR^d$.
where the last step is based on $\{f_{i,j}\}$ is $\bL$-mean-squared-smooth.
Therefore, each individual function $f_{i,j}$ is $\sqrt{mn}\bL$-smooth and we have proved equation~(\ref{eq:smooth-individual1}).
Consequently, we have
\begin{align*}
& \Norm{\nabla f_i(\vx) - \nabla f_i(\vy)}^2
= \Norm{\frac{1}{n}\sum_{j=1}^n(\nabla f_{i,j}(\vx) - \nabla f_{i,j}(\vy))}^2 \\
\leq & \frac{1}{n}\sum_{j=1}^n\Norm{\nabla f_{i,j}(\vx) - \nabla f_{i,j}(\vy)}^2 
\overset{(\ref{eq:smooth-individual1})}{\leq} mn\bL^2\Norm{\vx-\vy}^2
\end{align*}
for all $\vx,\vy\in\BR^d$.
where the first inequality is based on the fact $\Norm{\frac{1}{n}\sum_{i=1}^n \va_i}^2 \leq \frac{1}{n}\sum_{i=1}^n \Norm{\va_i}^2$ for all~$\va_1,\dots,\va_n\in\BR^d$; the second inequality is based on equation (\ref{eq:smooth-proof-1}).

\textbf{Part (b):} 
The mean-squared-smoothness condition means
\begin{align*}
 & \Norm{\nabla f(\vx) - \nabla f(\vy)}^2 \\
= & \Norm{\frac{1}{mn}\sum_{i=1}^m\sum_{j=1}^n\nabla f_{i,j}(\vx) -\frac{1}{mn}\sum_{i=1}^m\sum_{j=1}^n\nabla f_{i,j}(\vy)}^2 \\
\leq & \frac{1}{mn}\sum_{i=1}^m\sum_{j=1}^n\Norm{\nabla f_{i,j}(\vx) -\nabla f_{i,j}(\vy)}^2 
\overset{(\ref{eq:smooth-mean-squared})}{\leq} \bar L^2\Norm{\vx-\vy}^2
\end{align*}
for all $\vx,\vy\in\BR^d$.
where the first inequality is based on the fact 
\begin{align*}
\bigg\|\frac{1}{mn}\sum_{i=1}^m\sum_{j=1}^n \va_{i,j}\bigg\|^2 \leq \frac{1}{mn}\sum_{i=1}^m \sum_{j=1}^n \Norm{\va_{i,j}}^2     
\end{align*}
for all $\va_{1,1}, \dots, \va_{m,n}\in\BR^d$;
the last step is based on $\{f_{i,j}\}$ is $\bL$-mean-squared-smooth.
Therefore, the objective $f$ is $\bL$-smooth.

\textbf{Part (c):}
If the number of components $mn$ is even, we let
\begin{align}\label{eq:exmaple-even}
    f_{i,j}(\vx) = \begin{cases}
        (1+c_{\rm even})g(\vx), & \text{if $i+j$ is even}, \\[0.1cm]
        (1-c_{\rm even})g(\vx), & \text{if $i+j$ is odd},
    \end{cases}
\end{align}
where $g(\vx)=\frac{L}{2}\Norm{\vx}^2$ and
\begin{align}\label{eq:exmaple-even-c}
    c_{\rm even} = \sqrt{\dfrac{\bL^2}{L^2}-1}.
\end{align}
Then we have
\begin{align*}
    f(\vx) &\overset{\hphantom{(00)}}{=} \frac{1}{mn}\sum_{i=1}^m\sum_{j=1}^n f_{i,j}(\vx) \\
    &\overset{(\ref{eq:exmaple-even})}{=}  \frac{1}{mn} \left(\frac{mn(1+c_{\rm even})}{2} + \frac{mn\left(1-c_{\rm even}\,\right)}{2}\right)g(\vx) \\
    &\overset{(\ref{eq:exmaple-even-c})}{=}  g(\vx) = \frac{L}{2}\Norm{\vx}^2
\end{align*}
that means
\begin{align*}
    \Norm{\nabla f(\vx) - \nabla f(\vy)} \leq L\Norm{\vx-\vy},
\end{align*}
and
\begin{align*}
&  \overset{\hphantom{(00)}}{}  \frac{1}{mn}\sum_{i=1}^m\sum_{j=1}^n \Norm{\nabla f_{i,j}(\vx) - \nabla f_{i,j}(\vy)}^2 \\
&\overset{(\ref{eq:exmaple-even})}{=}  \frac{1}{mn}\left(\frac{mn(1+c_{\rm even})^2}{2} + \frac{mn(1-c_{\rm even})^2}{2}\Norm{\nabla g(\vx) -\nabla g(\vy)}^2\right) \\
&\overset{(\ref{eq:exmaple-even-c})}{=}  \left(\frac{(1+c_{\rm even})^2}{2} + \frac{(1-c_{\rm even})^2}{2}\right)\Norm{\nabla g(\vx) -\nabla g(\vy)}^2 \\
& \overset{\hphantom{(00)}}{=}  \left(1+c_{\rm even}^2\right)L^2\Norm{\vx-\vy}^2 = \bL^2\Norm{\vx-\vy}^2
\end{align*}
for all $\vx,\vy\in\BR^d$.
Therefore, the functions $\{f_{i,j}\}$ satisfy Assumption \ref{asm:smooth-global} and \ref{asm:smooth-mean-squared} with the tight $L$ and $\bL$, respectively.

If the number of components $mn$ is odd, we let
\begin{align}\label{eq:exmaple-odd}
    f_{i,j}(\vx) = \begin{cases}
        g(\vx), & \text{if $(i,j)=(1,1)$}, \\[0.1cm]
        \left(1+c_{\rm odd}\,\right)g(\vx), & \text{if $(i,j)\neq(1,1)$ and $i+j$ is even}, \\[0.1cm] 
        \left(1-c_{\rm odd}\,\right)g(\vx), & \text{if $(i,j)\neq(1,1)$ and $i+j$ is odd},
    \end{cases}
\end{align}
where $g:\BR^d\to\BR$ is some $L$-smooth function and 
\begin{align}\label{eq:exmaple-odd-c}
    c_{\rm odd} = \sqrt{\dfrac{mn\bL^2/L^2-1}{mn-1}-1}.
\end{align}
Then we have
\begin{align*}
    f(\vx) &  \overset{\hphantom{(00)}}= \frac{1}{mn}\sum_{i=1}^m\sum_{j=1}^n f_{i,j}(\vx) \\
    & \overset{(\ref{eq:exmaple-odd})}{=}  \frac{1}{mn} \left(g(\vx) + \frac{(mn-1)(1+c_{\rm odd})}{2}g(\vx) + \frac{(mn-1)(1-c_{\rm odd})}{2}g(\vx)\right) \\
    & \overset{\hphantom{(00)}}= g(\vx) = \frac{L}{2}\Norm{\vx}^2
\end{align*}
and
\begin{align*}
& \overset{\hphantom{(00)}}{}   \frac{1}{mn}\sum_{i=1}^m\sum_{j=1}^n \Norm{\nabla f_{i,j}(\vx) - \nabla f_{i,j}(\vy)}^2 \\
& \overset{(\ref{eq:exmaple-odd})}{=} \frac{1}{mn}\left(1+\frac{(mn-1)(1+c_{\rm odd})^2}{2} + \frac{(mn-1)(1-c_{\rm odd})^2}{2}\right)\Norm{\nabla g(\vx) -\nabla g(\vy)}^2 \\
& \overset{(\ref{eq:exmaple-odd-c})}{=}   \frac{1}{mn}\left(1+\frac{mn-1}{2}\left(2+\dfrac{2(mn\bL^2/L^2-1)}{mn-1}-2\right) \right)\Norm{\nabla g(\vx) -\nabla g(\vy)}^2 \\
& \overset{\hphantom{(00)}}{=}  \frac{1}{mn}\left(1+\frac{mn-1}{2}\cdot\dfrac{2(mn\bL^2/L^2-1)}{mn-1} \right)\Norm{\nabla g(\vx) -\nabla g(\vy)}^2 \\
& \overset{\hphantom{(00)}}{=}  \frac{1}{mn}\cdot\frac{mn\bL^2}{L^2}\cdot L^2\Norm{\vx-\vy}^2 
\leq\bL^2\Norm{\vx-\vy}^2 
\end{align*}
for all $\vx,\vy\in\BR^d$.
Therefore, the functions $\{f_{i,j}\}$ satisfy Assumption \ref{asm:smooth-global} and \ref{asm:smooth-mean-squared} with the tight $L$ and $\bL$, respectively.
\end{proof}

\subsection{The Proof of Proposition \ref{prop:smooth-2}}\label{appendix:smooth-2}

\begin{proof}
\textbf{Part (a):}
The local smoothness condition means
\begin{align*}
 & \Norm{\nabla f(\vx) - \nabla f(\vy)}^2 \\
= & \Norm{\frac{1}{m}\sum_{i=1}^m\nabla f_{i}(\vx) - \frac{1}{m}\sum_{i=1}^m\nabla f_{i}(\vy)}^2 \\
\leq & \frac{1}{m}\sum_{i=1}^m\Norm{\nabla f_{i}(\vx) -\nabla f_{i}(\vy)}^2 
\overset{(\ref{eq:smooth-local})}{\leq} L_\ell^2\Norm{\vx-\vy}^2,
\end{align*}
where the first inequality is based on the fact 
$\big\|\frac{1}{m}\sum_{i=1}^m \va_{i}\big\|^2 \leq \frac{1}{m}\sum_{i=1}^m \Norm{\va_{i}}^2$ for all~$\va_{1}, \dots, \va_{m}\in\BR^d$;
the last step is based on each $f_i$ is $L_\ell$-smooth  (Assumption \ref{asm:smooth-local}).
Therefore, the objective $f$ is $L_\ell$-smooth.   

\textbf{Part (b):}
The local mean-squared smoothness condition means
\begin{align*}
 \frac{1}{mn}\sum_{i=1}^m\sum_{j=1}^n\Norm{\nabla f_{i,j}(\vx) - \nabla f_{i,j}(\vy)}^2 
\overset{(\ref{eq:smooth-mean-squared-local})}{\leq}  \frac{1}{m}\sum_{i=1}^m \bL_\ell^2\Norm{\vx-\vy}^2 
= \bL_\ell^2\Norm{\vx-\vy}^2,
\end{align*}
where the first inequality is based on the functions $\{f_{i,j}\}_{j=1}^n$ are $\bL_\ell$-mean-squared smooth for all $i\in[m]$ (Assumption \ref{asm:smooth-local-mean-squared}).
Therefore, the individual functions $\{f_{i,j}\}_{i,j=1}^{m,n}$ are $\bL_\ell$-mean-squared smooth.  

\textbf{Part (c):}
We let
\begin{align*}
    f_{i}(\vx) = \begin{cases}
        \dfrac{L_\ell}{2}\Norm{\vx}^2, & \text{if $i=1$}, \\[0.25cm]
        \dfrac{mL-L_\ell}{2(m-1)}\Norm{\vx}^2, & \text{if $i\geq 2$}.
    \end{cases}
\end{align*}
Then we have 
\begin{align*}
    f(\vx) = \frac{1}{m}\sum_{i=1}^m f_i(\vx) = \frac{1}{m}\left(\frac{L_\ell}{2}+\dfrac{(m-1)(mL-L_\ell)}{2(m-1)}\right)\Norm{\vx}^2 = \frac{L}{2}\Norm{\vx}^2.
\end{align*}
Additionally, the function $f_1$ has the largest (tight) smoothness parameter $L_\ell$ among the functions $\{f_i\}_{i=1}^m$, since it holds
\begin{align*}
    L_\ell \geq \dfrac{mL-L_\ell}{m-1}.
\end{align*}
Therefore, the functions $\{f_{i}\}_{i=1}^m$ satisfy Assumption \ref{asm:smooth-global} and \ref{asm:smooth-local} with the tight $L$ and $L_\ell$, respectively.

If the number of agents $m$ is odd, we let
\begin{align}\label{eq:exmaple2-even}
    f_{i}(\vx) = \begin{cases}
        (1+c_{\rm even})g(\vx), & \text{if $i+j$ is even}, \\[0.1cm]
        (1-c_{\rm even})g(\vx), & \text{if $i+j$ is odd},
    \end{cases}
\end{align}
where $g(\vx)=\frac{L}{2}\Norm{\vx}^2$ and
\begin{align}\label{eq:exmaple-2-even-c}
    c_{\rm even} = \sqrt{\dfrac{\bL^2}{L^2}-1}.
\end{align}

\textbf{Part (d):}
We let
\begin{align*}
    f_{i,j}(\vx) = \begin{cases}
        \dfrac{\bL_\ell}{2}\Norm{\vx}^2, & \text{if $i=1$}, \\[0.25cm]
        \dfrac{m\bL-\bL_\ell}{2(m-1)}\Norm{\vx}^2, & \text{if $i\geq 2$}
    \end{cases}
\end{align*}
for all $j\in[n]$. 
We then have 
\begin{align*}
    & \frac{1}{mn}\sum_{i=1}^m\sum_{j=1}^n \Norm{\nabla f_{i,j}(\vx) - \nabla f_{i,j}(\vy)}^2 \\
    = & \frac{1}{m}\left(\bL_\ell+\dfrac{(m-1)(m\bL-\bL_\ell)}{m-1}\right)\Norm{\vx-\vy}^2 = \bL\Norm{\vx-\vy}^2.
\end{align*}
Additionally, we have
\begin{align*}
\frac{1}{n}\sum_{j=1}^n \Norm{\nabla f_{i,j}(\vx) - \nabla f_{i,j}(\vy)}^2 
= \begin{cases}
        \bL_\ell^2\Norm{\vx-\vy}^2, & \text{if $i=1$}, \\[0.25cm]
        \dfrac{m\bL-\bL_\ell}{m-1}\Norm{\vx-\vy}^2, & \text{if $i\geq 2$},
    \end{cases}
\end{align*}
which means the function set $\{f_{1,j}\}_{j=1}^n$ has the largest (tight) smoothness parameter~$\bL_\ell$ among the function sets $\{f_{1,j}\}_{j=1}^n,\dots,\{f_{m,j}\}_{j=1}^n$, since it holds 
\begin{align*}
    \bL_\ell \geq \dfrac{m\bL-\bL_\ell}{m-1}.
\end{align*}
Therefore, the functions $\{f_{i,j}\}_{i,j=1}^{m,n}$ satisfy Assumption \ref{asm:smooth-mean-squared} and \ref{asm:smooth-local-mean-squared} with the tight $L$ and~$L_\ell$, respectively.
\end{proof}

\section{The Finite-Sum Optimization with Convex Individual Functions}\label{appendix:convex-nonconvex}

Recall that Remark \ref{remark:IFO-tradeoff} says that the IFO complexity of vanilla gradient descent may be sharper than the stochastic recursive gradient method since the ratio $\bL/L$ can be arbitrary large.
However, the similar result does not hold in the convex case.
We consider the finite-sum optimization problem
\begin{align}\label{prob:cvx}
    \min_{\vx\in\BR^d} f^{\rm c} (\vx) \triangleq \frac{1}{n} \sum_{i=1}^n f^{\rm c}_i(\vx),
\end{align}
where $f^{\rm c}:\BR^d\to\BR$ is $L$-smooth and $\mu$-strongly convex, and each $f^{\rm c}_i:\BR^d\to\BR$ is $\Lm$-smooth and convex.
It is well-known that accelerated gradient descent (AGD) \cite{nesterov2018lectures} can find an $\epsilon$-suboptimal solution of problem (\ref{prob:cvx}) with $\fO(\sqrt{L/\mu}\ln(1/\epsilon))$ exact gradient calls, corresponding to the IFO complexity of 
$\fO(n\sqrt{L/\mu}\ln(1/\epsilon))$.
On the other hand, the stochastic variance reduced gradient methods (e.g., Katyusha \cite{allen2017katyusha}) can achieve the IFO complexity of $\fO((n+\sqrt{n\Lm/\mu})\ln(1/\epsilon))$, which is always better or equal to the result of AGD since the ratio $\Lm/L$ (for the tight $\Lm$ and $L$) is no larger than $n$ in the convex setting.

\begin{proposition}
Suppose the differentiable functions $f_1^{\rm c},\dots,f_n^{\rm c}:\BR^d\to\BR$ are convex, and the function $f^{\rm c}=\frac{1}{n}\sum_{i=1}^n f_i^{\rm c}$ is $L$-smooth, then each $f_i^{\rm c}$ is $nL$-smooth. 
\end{proposition}
\begin{proof}
    For all $\vx,\vy\in\BR^d$, the convexity of the individual function means
    \begin{align}\label{eq:cvx-1}
        f^{\rm c}_j(\vy) - f^{\rm c}_j(\vy) - \inner{\nabla f^{\rm c}_j(\vx)}{\vy-\vx} \geq 0
    \end{align}
    for all $j\in[n]$. Combining the smoothness of $f^{\rm c}=\frac{1}{n}\sum_{i=1}^n f_i^{\rm c}$, we have
    \begin{align}\label{eq:cvx-2}
        0 \leq f^{\rm c}(\vy) - f^{\rm c}(\vy) - \inner{\nabla f^{\rm c}(\vx)}{\vy-\vx} \leq \frac{L}{2}\Norm{\vx-\vy}^2.
    \end{align}
    Then we have 
    \begin{align*}
        & f^{\rm c}_i(\vy) - f^{\rm c}_i(\vy) - \inner{\nabla f^{\rm c}_i(\vx)}{\vy-\vx} \\
        \,\,\overset{(\ref{eq:cvx-1})}{\leq}\,\, & \sum_{j=1}^n \big(f^{\rm c}_j(\vy) - f^{\rm c}_j(\vy) - \inner{\nabla f^{\rm c}_j(\vx)}{\vy-\vx}\big) \\
        \,\,\,\,=\,\,\,\, & n\big(f^{\rm c}(\vy) - f^{\rm c}(\vy) - \inner{\nabla f^{\rm c}(\vx)}{\vy-\vx}\big) \\
        \overset{(\ref{eq:cvx-2})}{\leq} & \frac{nL}{2}\Norm{\vx-\vy}^2,
    \end{align*}
    for all $i\in[n]$ and $\vx,\vy\in\BR^d$, which finishes the proof.
\end{proof}

\section{The IFO Algorithm}\label{appendix:IFO}

We consider solving the finite-sum optimization problem
\begin{align*}
    \min_{\vx\in\BR^d} f(\vx)\triangleq \frac{1}{N}\sum_{i=1}^N f_i(\vx)     
\end{align*}
where each $f_i:\BR^d\to\BR$ is differentiable.
We formally present the definition of the incremental first-order oracle (IFO) algorithm as follows 
\cite{fang2018spider,zhou2019lower,li2021page, johnson2013accelerating,zhang2013linear,defazio2014saga,schmidt2017minimizing,allen2017katyusha,hofmann2015variance}.

\begin{definition}\label{def:linear-span}
An IFO algorithm for given initial point~$\vx^0$ is defined as a measurable mapping from functions $\{f_i\}_{i=1}^N$ to an infinite sequence of point and index
pairs $\{(\vx^t, i_t)\}_{t=0}^\infty$ with random variable $i_t\in[N]$, which satisfies
\begin{align*}
\vx^t \in \mathrm{span}\big(\{\vx^0, \ldots, \vx^{t-1}, 
\nabla f_{i_0}(\vx^0), \ldots, \nabla f_{i_{t-1}}(\vx^{t-1})\}\big),
\end{align*}
where $\mathrm{span}(\cdot)$ denotes the linear span and $i_t$ denotes the index of individual function chosen at the $t$-th step. 
\end{definition}

\end{appendices}

%% file: main.bbl
\begin{thebibliography}{95}
\providecommand{\natexlab}[1]{#1}
\providecommand{\url}[1]{\texttt{#1}}
\expandafter\ifx\csname urlstyle\endcsname\relax
  \providecommand{\doi}[1]{doi: #1}\else
  \providecommand{\doi}{doi: \begingroup \urlstyle{rm}\Url}\fi

\bibitem[Agarwal and Bottou(2015)]{agarwal2015lower}
Alekh Agarwal and Leon Bottou.
\newblock A lower bound for the optimization of finite sums.
\newblock In \emph{ICML}, 2015.

\bibitem[Agarwal et~al.(2021)Agarwal, Kakade, Lee, and Mahajan]{agarwal2021theory}
Alekh Agarwal, Sham~M. Kakade, Jason~D. Lee, and Gaurav Mahajan.
\newblock On the theory of policy gradient methods: Optimality, approximation, and distribution shift.
\newblock \emph{Journal of Machine Learning Research}, 22\penalty0 (1):\penalty0 4431--4506, 2021.

\bibitem[Allen-Zhu(2017)]{allen2017katyusha}
Zeyuan Allen-Zhu.
\newblock Katyusha: The first direct acceleration of stochastic gradient methods.
\newblock \emph{Journal of Machine Learning Research}, 18\penalty0 (1):\penalty0 8194--8244, 2017.

\bibitem[Allen-Zhu and Hazan(2016)]{allen2016variance}
Zeyuan Allen-Zhu and Elad Hazan.
\newblock Variance reduction for faster non-convex optimization.
\newblock In \emph{ICML}, 2016.

\bibitem[Allen-Zhu et~al.(2019)Allen-Zhu, Li, and Song]{allen2019convergence}
Zeyuan Allen-Zhu, Yuanzhi Li, and Zhao Song.
\newblock A convergence theory for deep learning via over-parameterization.
\newblock In \emph{ICML}, 2019.

\bibitem[Antoniadis et~al.(2011)Antoniadis, Gijbels, and Nikolova]{antoniadis2011penalized}
Anestis Antoniadis, Ir{\`e}ne Gijbels, and Mila Nikolova.
\newblock Penalized likelihood regression for generalized linear models with non-quadratic penalties.
\newblock \emph{Annals of the Institute of Statistical Mathematics}, 63\penalty0 (3):\penalty0 585--615, 2011.

\bibitem[Arioli and Scott(2014)]{arioli2014chebyshev}
Mario Arioli and Jennifer Scott.
\newblock Chebyshev acceleration of iterative refinement.
\newblock \emph{Numerical Algorithms}, 66\penalty0 (3):\penalty0 591--608, 2014.

\bibitem[Arjevani et~al.(2023)Arjevani, Carmon, Duchi, Foster, Srebro, and Woodworth]{arjevani2023lower}
Yossi Arjevani, Yair Carmon, John~C. Duchi, Dylan~J. Foster, Nathan Srebro, and Blake Woodworth.
\newblock Lower bounds for non-convex stochastic optimization.
\newblock \emph{Mathematical Programming}, 199\penalty0 (1):\penalty0 165--214, 2023.

\bibitem[Attouch and Bolte(2009)]{attouch2009convergence}
Hedy Attouch and J{\'e}r{\^o}me Bolte.
\newblock On the convergence of the proximal algorithm for nonsmooth functions involving analytic features.
\newblock \emph{Mathematical Programming}, 116:\penalty0 5--16, 2009.

\bibitem[Bai et~al.(2024)Bai, Liu, and Luo]{bai2024complexity}
Yunyan Bai, Yuxing Liu, and Luo Luo.
\newblock On the complexity of finite-sum smooth optimization under the {P}olyak--{$\L$}ojasiewicz condition.
\newblock In \emph{ICML}, 2024.

\bibitem[Bi et~al.(2022)Bi, Zhang, and Lavaei]{bi2022local}
Yingjie Bi, Haixiang Zhang, and Javad Lavaei.
\newblock Local and global linear convergence of general low-rank matrix recovery problems.
\newblock In \emph{AAAI}, 2022.

\bibitem[Bolte et~al.(2007)Bolte, Daniilidis, and Lewis]{bolte2007lojasiewicz}
J{\'e}r{\^o}me Bolte, Aris Daniilidis, and Adrian Lewis.
\newblock The {{\L}}ojasiewicz inequality for nonsmooth subanalytic functions with applications to subgradient dynamical systems.
\newblock \emph{SIAM Journal on Optimization}, 17\penalty0 (4):\penalty0 1205--1223, 2007.

\bibitem[Bolte et~al.(2014)Bolte, Sabach, and Teboulle]{bolte2014proximal}
J{\'e}r{\^o}me Bolte, Shoham Sabach, and Marc Teboulle.
\newblock Proximal alternating linearized minimization for nonconvex and nonsmooth problems.
\newblock \emph{Mathematical Programming}, 146\penalty0 (1-2):\penalty0 459--494, 2014.

\bibitem[Bottou et~al.(2018)Bottou, Curtis, and Nocedal]{bottou2018optimization}
L{\'e}on Bottou, Frank~E. Curtis, and Jorge Nocedal.
\newblock Optimization methods for large-scale machine learning.
\newblock \emph{Siam Review}, 60\penalty0 (2):\penalty0 223--311, 2018.

\bibitem[Bu et~al.(2019)Bu, Mesbahi, Fazel, and Mesbahi]{bu2019lqr}
Jingjing Bu, Afshin Mesbahi, Maryam Fazel, and Mehran Mesbahi.
\newblock {LQR} through the lens of first order methods: Discrete-time case.
\newblock \emph{arXiv preprint arXiv:1907.08921}, 2019.

\bibitem[Carmon and Duchi(2020)]{carmon2020first}
Yair Carmon and John~C. Duchi.
\newblock First-order methods for nonconvex quadratic minimization.
\newblock \emph{SIAM Review}, 62\penalty0 (2):\penalty0 395--436, 2020.

\bibitem[Carmon et~al.(2020)Carmon, Duchi, Hinder, and Sidford]{carmon2020lower}
Yair Carmon, John~C. Duchi, Oliver Hinder, and Aaron Sidford.
\newblock Lower bounds for finding stationary points {I}.
\newblock \emph{Mathematical Programming}, 184\penalty0 (1):\penalty0 71--120, 2020.

\bibitem[Carmon et~al.(2021)Carmon, Duchi, Hinder, and Sidford]{carmon2021lower}
Yair Carmon, John~C. Duchi, Oliver Hinder, and Aaron Sidford.
\newblock Lower bounds for finding stationary points {II}: first-order methods.
\newblock \emph{Mathematical Programming}, 185\penalty0 (1):\penalty0 315--355, 2021.

\bibitem[Chang and Lin(2011)]{CC01a}
Chih-Chung Chang and Chih-Jen Lin.
\newblock {LIBSVM}: A library for support vector machines.
\newblock \emph{ACM Transactions on Intelligent Systems and Technology}, 2:\penalty0 27:1--27:27, 2011.
\newblock Software available at \url{http://www.csie.ntu.edu.tw/~cjlin/libsvm}.

\bibitem[Chen and Sayed(2012)]{chen2012diffusion}
Jianshu Chen and Ali~H. Sayed.
\newblock Diffusion adaptation strategies for distributed optimization and learning over networks.
\newblock \emph{IEEE Transactions on Signal Processing}, 60\penalty0 (8):\penalty0 4289--4305, 2012.

\bibitem[Chen and Sayed(2015)]{chen2015learning}
Jianshu Chen and Ali~H. Sayed.
\newblock On the learning behavior of adaptive networks—part {I}: Transient analysis.
\newblock \emph{IEEE Transactions on Information Theory}, 61\penalty0 (6):\penalty0 3487--3517, 2015.

\bibitem[Defazio et~al.(2014)Defazio, Bach, and Lacoste-Julien]{defazio2014saga}
Aaron Defazio, Francis Bach, and Simon Lacoste-Julien.
\newblock {SAGA}: A fast incremental gradient method with support for non-strongly convex composite objectives.
\newblock In \emph{NIPS}, 2014.

\bibitem[Facchinei et~al.(2015)Facchinei, Scutari, and Sagratella]{facchinei2015parallel}
Francisco Facchinei, Gesualdo Scutari, and Simone Sagratella.
\newblock Parallel selective algorithms for nonconvex big data optimization.
\newblock \emph{IEEE Transactions on Signal Processing}, 63\penalty0 (7):\penalty0 1874--1889, 2015.

\bibitem[Fang et~al.(2018)Fang, Li, Lin, and Zhang]{fang2018spider}
Cong Fang, Chris~Junchi Li, Zhouchen Lin, and Tong Zhang.
\newblock {SPIDER}: Near-optimal non-convex optimization via stochastic path-integrated differential estimator.
\newblock In \emph{NeurIPS}, 2018.

\bibitem[Fatkhullin and Polyak(2021)]{fatkhullin2021optimizing}
Ilyas Fatkhullin and Boris Polyak.
\newblock Optimizing static linear feedback: Gradient method.
\newblock \emph{SIAM Journal on Control and Optimization}, 59\penalty0 (5):\penalty0 3887--3911, 2021.

\bibitem[Fatkhullin et~al.(2022)Fatkhullin, Etesami, He, and Kiyavash]{fatkhullin2022sharp}
Ilyas Fatkhullin, Jalal Etesami, Niao He, and Negar Kiyavash.
\newblock Sharp analysis of stochastic optimization under global {K}urdyka-{{\L}}ojasiewicz inequality.
\newblock In \emph{NeurIPS}, 2022.

\bibitem[Fazel et~al.(2018)Fazel, Ge, Kakade, and Mesbahi]{fazel2018global}
Maryam Fazel, Rong Ge, Sham Kakade, and Mehran Mesbahi.
\newblock Global convergence of policy gradient methods for the linear quadratic regulator.
\newblock In \emph{ICML}, 2018.

\bibitem[Hardt and Ma(2016)]{hardt2016identity}
Moritz Hardt and Tengyu Ma.
\newblock Identity matters in deep learning.
\newblock In \emph{ICLR}, 2016.

\bibitem[Hendrikx et~al.(2021)Hendrikx, Bach, and Massoulie]{hendrikx2021optimal}
Hadrien Hendrikx, Francis Bach, and Laurent Massoulie.
\newblock An optimal algorithm for decentralized finite-sum optimization.
\newblock \emph{SIAM Journal on Optimization}, 31\penalty0 (4):\penalty0 2753--2783, 2021.

\bibitem[Hofmann et~al.(2015)Hofmann, Lucchi, Lacoste-Julien, and McWilliams]{hofmann2015variance}
Thomas Hofmann, Aurelien Lucchi, Simon Lacoste-Julien, and Brian McWilliams.
\newblock Variance reduced stochastic gradient descent with neighbors.
\newblock In \emph{NIPS}, 2015.

\bibitem[Jiang and Li(2022)]{jiang2022holderian}
Rujun Jiang and Xudong Li.
\newblock H{\"o}lderian error bounds and {K}urdyka-{{\L}}ojasiewicz inequality for the trust region subproblem.
\newblock \emph{Mathematics of Operations Research}, 47\penalty0 (4):\penalty0 3025--3050, 2022.

\bibitem[Johnson and Zhang(2013)]{johnson2013accelerating}
Rie Johnson and Tong Zhang.
\newblock Accelerating stochastic gradient descent using predictive variance reduction.
\newblock In \emph{NIPS}, 2013.

\bibitem[Karimi et~al.(2016)Karimi, Nutini, and Schmidt]{karimi2016linear}
Hamed Karimi, Julie Nutini, and Mark Schmidt.
\newblock Linear convergence of gradient and proximal-gradient methods under the {P}olyak--{{\L}}ojasiewicz condition.
\newblock In \emph{ECML/PKDD}, 2016.

\bibitem[Kovalev et~al.(2020)Kovalev, Horv{\'a}th, and Richt{\'a}rik]{kovalev2020don}
Dmitry Kovalev, Samuel Horv{\'a}th, and Peter Richt{\'a}rik.
\newblock Don’t jump through hoops and remove those loops: {SVRG} and {K}atyusha are better without the outer loop.
\newblock In \emph{ALT}, 2020.

\bibitem[Lan and Zhou(2018)]{lan2018optimal}
Guanghui Lan and Yi~Zhou.
\newblock An optimal randomized incremental gradient method.
\newblock \emph{Mathematical programming}, 171\penalty0 (1):\penalty0 167--215, 2018.

\bibitem[Lei et~al.(2017)Lei, Ju, Chen, and Jordan]{lei2017non}
Lihua Lei, Cheng Ju, Jianbo Chen, and Michael~I. Jordan.
\newblock Non-convex finite-sum optimization via {SCSG} methods.
\newblock In \emph{NIPS}, 2017.

\bibitem[Lewis et~al.(2004)Lewis, Yang, Russell-Rose, and Li]{lewis2004rcv1}
David~D. Lewis, Yiming Yang, Tony Russell-Rose, and Fan Li.
\newblock {RCV1}: A new benchmark collection for text categorization research.
\newblock \emph{Journal of machine learning research}, 5:\penalty0 361--397, 2004.

\bibitem[Li et~al.(2020)Li, Cen, Chen, and Chi]{li2019communication}
Boyue Li, Shicong Cen, Yuxin Chen, and Yuejie Chi.
\newblock Communication-efficient distributed optimization in networks with gradient tracking and variance reduction.
\newblock \emph{Journal of Machine Learning Research}, 21:\penalty0 1--51, 2020.

\bibitem[Li et~al.(2022{\natexlab{a}})Li, Li, and Chi]{li2022destress}
Boyue Li, Zhize Li, and Yuejie Chi.
\newblock {DESTRESS}: Computation-optimal and communication-efficient decentralized nonconvex finite-sum optimization.
\newblock \emph{SIAM Journal on Mathematics of Data Science}, 4\penalty0 (3):\penalty0 1031--1051, 2022{\natexlab{a}}.

\bibitem[Li et~al.(2022{\natexlab{b}})Li, Lin, and Fang]{li2022variance}
Huan Li, Zhouchen Lin, and Yongchun Fang.
\newblock Variance reduced extra and diging and their optimal acceleration for strongly convex decentralized optimization.
\newblock \emph{Journal of Machine Learning Research}, 23\penalty0 (222):\penalty0 1--41, 2022{\natexlab{b}}.

\bibitem[Li et~al.(2018)Li, Ma, and Zhang]{li2018algorithmic}
Yuanzhi Li, Tengyu Ma, and Hongyang Zhang.
\newblock Algorithmic regularization in over-parameterized matrix sensing and neural networks with quadratic activations.
\newblock In \emph{COLT}, 2018.

\bibitem[Li et~al.(2021)Li, Bao, Zhang, and Richt{\'a}rik]{li2021page}
Zhize Li, Hongyan Bao, Xiangliang Zhang, and Peter Richt{\'a}rik.
\newblock {PAGE}: A simple and optimal probabilistic gradient estimator for nonconvex optimization.
\newblock In \emph{ICML}, 2021.

\bibitem[Lian et~al.(2017)Lian, Zhang, Zhang, Hsieh, Zhang, and Liu]{lian2017can}
Xiangru Lian, Ce~Zhang, Huan Zhang, Cho-Jui Hsieh, Wei Zhang, and Ji~Liu.
\newblock Can decentralized algorithms outperform centralized algorithms? a case study for decentralized parallel stochastic gradient descent.
\newblock \emph{NIPS}, 2017.

\bibitem[Liu et~al.(2022)Liu, Zhu, and Belkin]{liu2022loss}
Chaoyue Liu, Libin Zhu, and Mikhail Belkin.
\newblock Loss landscapes and optimization in over-parameterized non-linear systems and neural networks.
\newblock \emph{Applied and Computational Harmonic Analysis}, 59:\penalty0 85--116, 2022.

\bibitem[Liu and Morse(2011)]{liu2011accelerated}
Ji~Liu and A.~Stephen Morse.
\newblock Accelerated linear iterations for distributed averaging.
\newblock \emph{Annual Reviews in Control}, 35\penalty0 (2):\penalty0 160--165, 2011.

\bibitem[Liu et~al.(2024)Liu, Chen, and Luo]{liu2024decentralized}
Yuxing Liu, Lesi Chen, and Luo Luo.
\newblock Decentralized convex finite-sum optimization with better dependence on condition numbers.
\newblock In \emph{ICML}, 2024.

\bibitem[Lojasiewicz(1963)]{lojasiewicz1963topological}
Stanislaw Lojasiewicz.
\newblock A topological property of real analytic subsets.
\newblock \emph{Coll. du CNRS, Les {\'e}quations aux d{\'e}riv{\'e}es partielles}, 117\penalty0 (87-89):\penalty0 2, 1963.

\bibitem[Lu and De~Sa(2021)]{lu2021optimal}
Yucheng Lu and Christopher De~Sa.
\newblock Optimal complexity in decentralized training.
\newblock In \emph{ICML}, 2021.

\bibitem[Maranjyan et~al.(2022)Maranjyan, Safaryan, and Richt{\'a}rik]{maranjyan2022gradskip}
Artavazd Maranjyan, Mher Safaryan, and Peter Richt{\'a}rik.
\newblock {GradSkip}: Communication-accelerated local gradient methods with better computational complexity.
\newblock \emph{arXiv preprint arXiv:2210.16402}, 2022.

\bibitem[Mei et~al.(2020)Mei, Xiao, Szepesvari, and Schuurmans]{mei2020global}
Jincheng Mei, Chenjun Xiao, Csaba Szepesvari, and Dale Schuurmans.
\newblock On the global convergence rates of softmax policy gradient methods.
\newblock In \emph{ICML}, 2020.

\bibitem[Metelev et~al.(2024)Metelev, Chezhegov, Rogozin, Kovalev, Beznosikov, Sholokhov, and Gasnikov]{metelev2024decentralized}
Dmitry Metelev, Savelii Chezhegov, Alexander Rogozin, Dmitry Kovalev, Aleksandr Beznosikov, Alexander Sholokhov, and Alexander Gasnikov.
\newblock Decentralized finite-sum optimization over time-varying networks.
\newblock \emph{arXiv preprint arXiv:2402.02490}, 2024.

\bibitem[Mishchenko et~al.(2022)Mishchenko, Malinovsky, Stich, and Richt{\'a}rik]{mishchenko2022proxskip}
Konstantin Mishchenko, Grigory Malinovsky, Sebastian Stich, and Peter Richt{\'a}rik.
\newblock Proxskip: Yes! local gradient steps provably lead to communication acceleration! finally!
\newblock In \emph{ICML}, 2022.

\bibitem[Motwani and Raghavan(1995)]{motwani1995randomized}
Rajeev Motwani and Prabhakar Raghavan.
\newblock \emph{Randomized algorithms}.
\newblock Cambridge university press, 1995.

\bibitem[Nedic and Ozdaglar(2009)]{nedic2009distributed}
Angelia Nedic and Asuman Ozdaglar.
\newblock Distributed subgradient methods for multi-agent optimization.
\newblock \emph{IEEE Transactions on Automatic Control}, 54\penalty0 (1):\penalty0 48--61, 2009.

\bibitem[Nedic et~al.(2017)Nedic, Olshevsky, and Shi]{nedic2017achieving}
Angelia Nedic, Alex Olshevsky, and Wei Shi.
\newblock Achieving geometric convergence for distributed optimization over time-varying graphs.
\newblock \emph{SIAM Journal on Optimization}, 27\penalty0 (4):\penalty0 2597--2633, 2017.

\bibitem[Nedi{\'c} et~al.(2018)Nedi{\'c}, Olshevsky, and Rabbat]{nedic2018network}
Angelia Nedi{\'c}, Alex Olshevsky, and Michael~G. Rabbat.
\newblock Network topology and communication-computation tradeoffs in decentralized optimization.
\newblock \emph{Proceedings of the IEEE}, 106\penalty0 (5):\penalty0 953--976, 2018.

\bibitem[Nemirovskij and Yudin(1983)]{nemirovskij1983problem}
Arkadij~Semenovi{\v{c}} Nemirovskij and David~Borisovich Yudin.
\newblock \emph{Problem complexity and method efficiency in optimization}.
\newblock Wiley-Interscience, 1983.

\bibitem[Nesterov(2018)]{nesterov2018lectures}
Yurii Nesterov.
\newblock \emph{Lectures on convex optimization}, volume 137.
\newblock Springer, 2018.

\bibitem[Nguyen et~al.(2017)Nguyen, Liu, Scheinberg, and Tak{\'a}{\v{c}}]{nguyen2017sarah}
Lam~M. Nguyen, Jie Liu, Katya Scheinberg, and Martin Tak{\'a}{\v{c}}.
\newblock {SARAH}: A novel method for machine learning problems using stochastic recursive gradient.
\newblock In \emph{ICML}, 2017.

\bibitem[Pham et~al.(2020{\natexlab{a}})Pham, Nguyen, Phan, and Tran-Dinh]{pham2019proxsarah}
Nhan~H Pham, Lam~M Nguyen, Dzung~T Phan, and Quoc Tran-Dinh.
\newblock {ProxSARAH}: An efficient algorithmic framework for stochastic composite nonconvex optimization.
\newblock \emph{Journal of Machine Learning Research}, 21\penalty0 (110):\penalty0 1--48, 2020{\natexlab{a}}.

\bibitem[Pham et~al.(2020{\natexlab{b}})Pham, Nguyen, Phan, and Tran-Dinh]{pham2020proxsarah}
Nhan~H Pham, Lam~M Nguyen, Dzung~T Phan, and Quoc Tran-Dinh.
\newblock Proxsarah: An efficient algorithmic framework for stochastic composite nonconvex optimization.
\newblock \emph{Journal of Machine Learning Research}, 21\penalty0 (110):\penalty0 1--48, 2020{\natexlab{b}}.

\bibitem[Polyak(1963)]{polyak1963gradient}
Boris~Teodorovich Polyak.
\newblock Gradient methods for minimizing functionals.
\newblock \emph{Zhurnal Vychislitel'noi Matematiki i Matematicheskoi Fiziki}, 3\penalty0 (4):\penalty0 643--653, 1963.

\bibitem[Pu and Nedi{\'c}(2021)]{pu2021distributed}
Shi Pu and Angelia Nedi{\'c}.
\newblock Distributed stochastic gradient tracking methods.
\newblock \emph{Mathematical Programming}, 187\penalty0 (1):\penalty0 409--457, 2021.

\bibitem[Qian et~al.(2021)Qian, Qu, and Richt{\'a}rik]{qian2021svrg}
Xun Qian, Zheng Qu, and Peter Richt{\'a}rik.
\newblock L-{SVRG} and {L}-{Katyusha} with arbitrary sampling.
\newblock \emph{Journal of Machine Learning Research}, 22\penalty0 (112):\penalty0 1--47, 2021.

\bibitem[Qu and Li(2019)]{qu2019accelerated}
Guannan Qu and Na~Li.
\newblock Accelerated distributed nesterov gradient descent.
\newblock \emph{IEEE Transactions on Automatic Control}, 65\penalty0 (6):\penalty0 2566--2581, 2019.

\bibitem[Reddi et~al.(2016)Reddi, Hefny, Sra, Poczos, and Smola]{reddi2016stochastic}
Sashank~J. Reddi, Ahmed Hefny, Suvrit Sra, Barnabas Poczos, and Alex Smola.
\newblock Stochastic variance reduction for nonconvex optimization.
\newblock In \emph{ICML}, 2016.

\bibitem[Rozemberczki et~al.(2021)Rozemberczki, Scherer, He, Panagopoulos, Riedel, Astefanoaei, Kiss, Beres, Lopez, Collignon, et~al.]{rozemberczki2021pytorch}
Benedek Rozemberczki, Paul Scherer, Yixuan He, George Panagopoulos, Alexander Riedel, Maria Astefanoaei, Oliver Kiss, Ferenc Beres, Guzman Lopez, Nicolas Collignon, et~al.
\newblock Pytorch geometric temporal: Spatiotemporal signal processing with neural machine learning models.
\newblock In \emph{CIKM}, 2021.

\bibitem[Scaman et~al.(2018)Scaman, Bach, Bubeck, Massouli{\'e}, and Lee]{scaman2018optimal}
Kevin Scaman, Francis Bach, S{\'e}bastien Bubeck, Laurent Massouli{\'e}, and Yin~Tat Lee.
\newblock Optimal algorithms for non-smooth distributed optimization in networks.
\newblock In \emph{NeurIPS}, 2018.

\bibitem[Schmidt et~al.(2017)Schmidt, Le~Roux, and Bach]{schmidt2017minimizing}
Mark Schmidt, Nicolas Le~Roux, and Francis Bach.
\newblock Minimizing finite sums with the stochastic average gradient.
\newblock \emph{Mathematical Programming}, 162\penalty0 (1-2):\penalty0 83--112, 2017.

\bibitem[Shi et~al.(2015)Shi, Ling, Wu, and Yin]{shi2015extra}
Wei Shi, Qing Ling, Gang Wu, and Wotao Yin.
\newblock {EXTRA}: An exact first-order algorithm for decentralized consensus optimization.
\newblock \emph{SIAM Journal on Optimization}, 25\penalty0 (2):\penalty0 944--966, 2015.

\bibitem[Song et~al.(2024)Song, Shi, Pu, and Yan]{song2024optimal}
Zhuoqing Song, Lei Shi, Shi Pu, and Ming Yan.
\newblock Optimal gradient tracking for decentralized optimization.
\newblock \emph{Mathematical Programming}, 207\penalty0 (1):\penalty0 1--53, 2024.

\bibitem[Sun and Hong(2019)]{sun2019distributed}
Haoran Sun and Mingyi Hong.
\newblock Distributed non-convex first-order optimization and information processing: Lower complexity bounds and rate optimal algorithms.
\newblock \emph{IEEE Transactions on Signal processing}, 67\penalty0 (22):\penalty0 5912--5928, 2019.

\bibitem[Sun et~al.(2020)Sun, Lu, and Hong]{sun2020improving}
Haoran Sun, Songtao Lu, and Mingyi Hong.
\newblock Improving the sample and communication complexity for decentralized non-convex optimization: Joint gradient estimation and tracking.
\newblock In \emph{ICML}, 2020.

\bibitem[Sun et~al.(2016)Sun, Qu, and Wright]{sun2016complete}
Ju~Sun, Qing Qu, and John Wright.
\newblock Complete dictionary recovery over the sphere {I}: Overview and the geometric picture.
\newblock \emph{IEEE Transactions on Information Theory}, 63\penalty0 (2):\penalty0 853--884, 2016.

\bibitem[Sun et~al.(2018)Sun, Qu, and Wright]{sun2018geometric}
Ju~Sun, Qing Qu, and John Wright.
\newblock A geometric analysis of phase retrieval.
\newblock \emph{Foundations of Computational Mathematics}, 18\penalty0 (5):\penalty0 1131--1198, 2018.

\bibitem[Sundhar~Ram et~al.(2010)Sundhar~Ram, Nedi{\'c}, and Veeravalli]{sundhar2010distributed}
S.~Sundhar~Ram, Angelia Nedi{\'c}, and Venugopal~V. Veeravalli.
\newblock Distributed stochastic subgradient projection algorithms for convex optimization.
\newblock \emph{Journal of optimization theory and applications}, 147\penalty0 (3):\penalty0 516--545, 2010.

\bibitem[Wai et~al.(2017)Wai, Lafond, Scaglione, and Moulines]{wai2017decentralized}
Hoi-To Wai, Jean Lafond, Anna Scaglione, and Eric Moulines.
\newblock Decentralized {F}rank--{W}olfe algorithm for convex and nonconvex problems.
\newblock \emph{IEEE Transactions on Automatic Control}, 62\penalty0 (11):\penalty0 5522--5537, 2017.

\bibitem[Wang et~al.(2019)Wang, Ji, Zhou, Liang, and Tarokh]{wang2019spiderboost}
Zhe Wang, Kaiyi Ji, Yi~Zhou, Yingbin Liang, and Vahid Tarokh.
\newblock {SpiderBoost} and momentum: Faster variance reduction algorithms.
\newblock In \emph{NeurIPS}, 2019.

\bibitem[Watts and Strogatz(1998)]{watts1998collective}
Duncan~J. Watts and Steven~H. Strogatz.
\newblock Collective dynamics of ‘small-world’ networks.
\newblock \emph{nature}, 393\penalty0 (6684):\penalty0 440--442, 1998.

\bibitem[Woodworth and Srebro(2016)]{woodworth2016tight}
Blake~E. Woodworth and Nati Srebro.
\newblock Tight complexity bounds for optimizing composite objectives.
\newblock In \emph{NIPS}, 2016.

\bibitem[Xin et~al.(2022)Xin, Khan, and Kar]{xin2022fast}
Ran Xin, Usman~A. Khan, and Soummya Kar.
\newblock Fast decentralized nonconvex finite-sum optimization with recursive variance reduction.
\newblock \emph{SIAM Journal on Optimization}, 32\penalty0 (1):\penalty0 1--28, 2022.

\bibitem[Yang et~al.(2019)Yang, Yi, Wu, Yuan, Wu, Meng, Hong, Wang, Lin, and Johansson]{yang2019survey}
Tao Yang, Xinlei Yi, Junfeng Wu, Ye~Yuan, Di~Wu, Ziyang Meng, Yiguang Hong, Hong Wang, Zongli Lin, and Karl~H Johansson.
\newblock A survey of distributed optimization.
\newblock \emph{Annual Reviews in Control}, 47:\penalty0 278--305, 2019.

\bibitem[Ye et~al.(2023)Ye, Luo, Zhou, and Zhang]{ye2023multi}
Haishan Ye, Luo Luo, Ziang Zhou, and Tong Zhang.
\newblock Multi-consensus decentralized accelerated gradient descent.
\newblock \emph{Journal of Machine Learning Research}, 24\penalty0 (306):\penalty0 1--50, 2023.

\bibitem[Yuan et~al.(2016)Yuan, Ling, and Yin]{yuan2016convergence}
Kun Yuan, Qing Ling, and Wotao Yin.
\newblock On the convergence of decentralized gradient descent.
\newblock \emph{SIAM Journal on Optimization}, 26\penalty0 (3):\penalty0 1835--1854, 2016.

\bibitem[Yuan et~al.(2022{\natexlab{a}})Yuan, Huang, Chen, Zhang, Zhang, and Pan]{yuan2022revisiting}
Kun Yuan, Xinmeng Huang, Yiming Chen, Xiaohan Zhang, Yingya Zhang, and Pan Pan.
\newblock Revisiting optimal convergence rate for smooth and non-convex stochastic decentralized optimization.
\newblock \emph{Advances in Neural Information Processing Systems}, 35:\penalty0 36382--36395, 2022{\natexlab{a}}.

\bibitem[Yuan et~al.(2022{\natexlab{b}})Yuan, Gower, and Lazaric]{yuan2022general}
Rui Yuan, Robert~M. Gower, and Alessandro Lazaric.
\newblock A general sample complexity analysis of vanilla policy gradient.
\newblock In \emph{AISTATS}, 2022{\natexlab{b}}.

\bibitem[Yue et~al.(2023)Yue, Fang, and Lin]{yue2023lower}
Pengyun Yue, Cong Fang, and Zhouchen Lin.
\newblock On the lower bound of minimizing polyak-{\l}ojasiewicz functions.
\newblock In \emph{The Thirty Sixth Annual Conference on Learning Theory}, pages 2948--2968. PMLR, 2023.

\bibitem[Zeng et~al.(2018)Zeng, Ouyang, Lau, Lin, and Yao]{zeng2018global}
Jinshan Zeng, Shikang Ouyang, Tim Tsz-Kit Lau, Shaobo Lin, and Yuan Yao.
\newblock Global convergence in deep learning with variable splitting via the {K}urdyka-{{\L}}ojasiewicz property.
\newblock \emph{arXiv preprint arXiv:1803.00225}, 9, 2018.

\bibitem[Zhan et~al.(2022)Zhan, Wu, and Gao]{zhan2022efficient}
Wenkang Zhan, Gang Wu, and Hongchang Gao.
\newblock Efficient decentralized stochastic gradient descent method for nonconvex finite-sum optimization problems.
\newblock In \emph{Proceedings of the AAAI Conference on Artificial Intelligence}, volume~36, pages 9006--9013, 2022.

\bibitem[Zhang et~al.(2013)Zhang, Mahdavi, and Jin]{zhang2013linear}
Lijun Zhang, Mehrdad Mahdavi, and Rong Jin.
\newblock Linear convergence with condition number independent access of full gradients.
\newblock In \emph{NIPS}, 2013.

\bibitem[Zhong et~al.(2023)Zhong, Kuffner, and Lahiri]{zhong2023online}
Yanjie Zhong, Todd Kuffner, and Soumendra Lahiri.
\newblock Online bootstrap inference with nonconvex stochastic gradient descent estimator.
\newblock \emph{arXiv preprint arXiv:2306.02205}, 2023.

\bibitem[Zhou and Gu(2019)]{zhou2019lower}
Dongruo Zhou and Quanquan Gu.
\newblock Lower bounds for smooth nonconvex finite-sum optimization.
\newblock In \emph{ICML}, 2019.

\bibitem[Zhou et~al.(2018{\natexlab{a}})Zhou, Xu, and Gu]{zhou2018stochastic}
Dongruo Zhou, Pan Xu, and Quanquan Gu.
\newblock Stochastic nested variance reduction for nonconvex optimization.
\newblock In \emph{NeurIPS}, 2018{\natexlab{a}}.

\bibitem[Zhou et~al.(2019)Zhou, Yuan, and Feng]{zhou2019faster}
Pan Zhou, Xiao-Tong Yuan, and Jiashi Feng.
\newblock Faster first-order methods for stochastic non-convex optimization on {Riemannian} manifolds.
\newblock In \emph{AISTATS}, 2019.

\bibitem[Zhou et~al.(2018{\natexlab{b}})Zhou, Wang, and Liang]{zhou2018convergence}
Yi~Zhou, Zhe Wang, and Yingbin Liang.
\newblock Convergence of cubic regularization for nonconvex optimization under {KL} property.
\newblock In \emph{NeurIPS}, 2018{\natexlab{b}}.

\end{thebibliography}
